\documentclass{amsart}
\usepackage{amsmath,amsfonts,mathrsfs,amssymb,enumerate,graphicx,subfig,caption}
\usepackage{xcolor} 
\usepackage[backref=page]{hyperref} 
\renewcommand*{\backref}[1]{}
\renewcommand*{\backrefalt}[4]{%
    \ifcase #1 (Not cited.)%
    \or        (Cited on page~#2.)%
    \else      (Cited on pages~#2.)%
    \fi}

\numberwithin{equation}{section}

\newtheorem{theorem}{Theorem}

\newtheorem{introtheorem}{Theorem}

\newtheorem{proposition}{Proposition}[section]
\newtheorem{corollary}[proposition]{Corollary}
\newtheorem{lemma}[proposition]{Lemma}

\newtheorem{definition}[proposition]{Definition}
\theoremstyle{remark}
\newtheorem{remark}[proposition]{Remark}

\DeclareMathOperator{\tr}{tr}

\DeclareMathOperator{\Gr}{Gr}

\DeclareMathOperator{\im}{im}

\DeclareMathOperator{\rank}{rank}

\DeclareMathOperator{\dimaff}{dim_{\mathsf{aff}}}
\DeclareMathOperator{\dimaffQ}{dim_{\mathsf{aff}}^Q}
\DeclareMathOperator{\dimlyap}{dim_{\mathsf{Lyap}}}
\DeclareMathOperator{\dimlyapQ}{dim_{\mathsf{Lyap}}^Q}

\DeclareMathOperator{\dimh}{dim_{\mathsf{H}}}
\DeclareMathOperator{\dimb}{dim_{\mathsf{B}}}
\DeclareMathOperator{\dimp}{dim_{\mathsf{P}}}

\DeclareMathOperator{\dimbu}{\overline{dim}_{\mathsf{B}}}

\DeclareMathOperator{\dimloc}{dim_{\mathsf{loc}}}
\DeclareMathOperator{\dimlocl}{\underline{dim}_{\mathsf{loc}}}
\DeclareMathOperator{\dimlocu}{\overline{dim}_{\mathsf{loc}}}

\DeclareMathOperator{\GL}{GL}
\DeclareMathOperator{\SO}{SO}

\renewcommand{\O}{\mathrm{O}}

\DeclareMathOperator{\id}{id}
\DeclareMathOperator{\End}{End}

\DeclareMathOperator{\ess}{ess}

\DeclareMathOperator{\diam}{diam}

\DeclareMathOperator{\supp}{supp}
\DeclareMathOperator{\Span}{span}
\newcommand{\threebar}[1]{{\left\vert\kern-0.25ex\left\vert\kern-0.25ex\left\vert #1 
    \right\vert\kern-0.25ex\right\vert\kern-0.25ex\right\vert}}
\newcommand{\iii}{\mathtt{i}}
\newcommand{\jjj}{\mathtt{j}}
\newcommand{\kkk}{\mathtt{k}}

\newcommand{\spectralradius}{\mathscr{R}}

\newcommand{\R}{\mathbb{R}}

\newcommand{\N}{\mathbb{N}}

\newcommand{\I}{\mathcal{I}}
\newcommand{\J}{\mathcal{J}}
\newcommand{\A}{\mathsf{A}}
\newcommand{\B}{\mathsf{B}}
\newcommand{\Z}{\mathbb{Z}}

\usepackage{multicol}
\usepackage{changepage}
\usepackage[shortlabels]{enumitem}
\usepackage{tabularx}
\usepackage{array}
\usepackage{longtable}
\usepackage{booktabs}
\newcommand{\symref}[1]{%
{\footnotesize\hyperref[#1]{(p.~\pageref*{#1})}}}

\usepackage{soul}

\begin{document}

\title{Projections of self-affine fractals}
\author{Ian D. Morris and  Cagri Sert}
\address{School of Mathematical Sciences, Queen Mary University of London, Mile End Road, London E1 4NS, U.K.}
\email{i.morris@qmul.ac.uk }

\address{Mathematics Institute, Zeeman Building, University of Warwick, Coventry CV4 7AL, U.K.}
\email{cagri.sert@warwick.ac.uk}

\begin{abstract}
We extend Falconer's 1988 landmark result on the dimensions of self-affine fractals to encompass the dimensions of their projections, 
showing furthermore 
that their families of exceptional projections 
contain algebraic 
varieties which are
preserved by
the underlying linear algebraic group. 
The techniques which we develop allow us to construct examples of additional 
new 
phenomena: firstly, we give general examples of equilibrium measures on self-affine fractals which admit non-exact-dimensional projections. Secondly, we construct strongly irreducible self-affine sets which have small sumsets without any arithmetic resonance in their construction. 
\end{abstract}
\maketitle
\tableofcontents

\section{Introduction}\label{se:small-intro}

\subsection{Context: the theorems of Falconer and Marstand}
This article draws together two classic results in fractal geometry, each itself a major focus of contemporary research: Falconer's theorem on self-affine fractals, and Marstrand's theorem on projections of fractal sets. We first briefly review these two results.

Let $\I$ be a finite nonempty set and $(T_i)_{i \in \I}$ a set of transformations of $\R^d$ which are contracting with respect to some norm $\threebar{\cdot}$. \label{notation:arbitnorm} By a classic result of J.E.~Hutchinson \cite{Hu81} there exists a unique nonempty compact set $X\subset \R^d$ such that $\bigcup_{i \in \I}T_iX=X$. In this context $(T_i)_{i\in \I}$ is called an \emph{iterated function system} or \emph{IFS} and the set $X$ is called its \emph{attractor}. If the transformations $T_i$ are similitudes then $X$ is called a self-similar set, and if they are affinities, $X$ is called \textit{self-affine}. In the latter case, writing each $T_i$ in the form $T_ix \equiv A_ix+u_i$, we call $\A=(A_i)_{i \in \I}$ \label{notation:A} the \emph{linearisation} of the affine IFS $(T_i)_{i \in \I}$. In the 1988 article \cite{Fa88} Falconer defined a real number $\dimaff \A$ --- now conventionally called the \emph{affinity dimension}, and whose definition we defer to \S\ref{se:big-intro} below --- which is unconditionally an upper bound for the Hausdorff dimension of the attractor of an affine IFS with linearisation $\A$. Under the additional condition $\max_{i,j \in \I\colon i\neq j} \threebar{A_i}+\threebar{A_j}<1$, where $\threebar{\cdot}$ also denotes the operator norm induced by the norm $\threebar{\cdot}$ on $\R^d$, the current form of Falconer's theorem (see \cite{BaSiSo23,So98}) asserts that for Lebesgue almost every affine IFS with linearisation $\A$, the Hausdorff dimension of the attractor is \emph{equal} to the affinity dimension of $\A$. This foundational result has guided over three decades of research on self-affine fractals (to mention a few: \cite{BaHoRa19,DaSi17,Fe23,FeSh14,JoPoSi07,MoSe23a,MoSh19,Ra24}) guided especially by the problem of replacing the Lebesgue almost-every hypothesis in this result by explicit verifiable conditions on the linearisation $(A_i)$ and translation component $(u_i)$. In dimension $d \leq 3$ it is known to be sufficient that the linearisation should be \emph{strongly irreducible} -- that is, that no proper linear subspace of $\R^d$ should have finite orbit under the action of $\A$ -- and that the distinct images $T_iX$ satisfy the \emph{strong open set condition} which guarantees that the overlaps between these images are in an appropriate sense small. In higher dimensions the appropriate conditions remain somewhat speculative, and for self-affine sets in $\R^4$ which are not self-similar and which have affinity dimension higher than $2$, no explicit results are currently known. 

Our main result will extend Falconer's theorem to linear projections of self-affine sets, expanding the scope of the above-described dimension problem for self-affine sets to this more general setting.

The second theme of this work is the dimension theory of linear projections of Borel subsets of $\R^d$, as exemplified by a second landmark result, the Marstand Projection Theorem of 1954 (\cite{Ma54}), which we now describe. Let $X$ be a Borel subset of $\R^d$. If $Q$\label{notation:Q} is an orthogonal projection of $\R^d$ onto a linear subspace then the Hausdorff dimension of the image set $QX$ is subject to the trivial upper bound $\dimh QX \leq \min \{\dimh X, \rank Q\}$. \label{notation:dimhX} The projection $Q$ is called \emph{exceptional} if the previous inequality is strict. Marstrand's theorem --- proved in the planar case by Marstrand in \cite{Ma54} and extended and simplified by Kaufmann and Mattila in \cite{Ka68,Ma75} --- asserts in a precise sense that exceptional projections of a Borel set are rare. Specifically, for each $k=0,\ldots,d$ let $\Gr(k,d)$ denote the Grassmannian manifold of $k$-dimensional subspaces of $\R^d$ and for each $U \in \Gr(k,d)$ let $Q_U$ denote orthogonal projection onto $U$. In its modern form, Marstrand's theorem states that for every Borel $X\subseteq \R^d$ and $k=0,\ldots,d$, the set of all $U \in \Gr(k,d)$ such that $Q_U$ is exceptional is a set of Lebesgue measure zero. Subsequent developments and current active research focus on studying the size (e.g.\ Hausdorff dimension) of the set of exceptional projections, and on restricted Marstand problems, i.e. the problem of proving finer results which apply to exceptional projections within some proper submanifold of the Grassmanniann variety (see e.g.\ the recent breakthroughs \cite{GaGuGuHaMaWa22,KaOrVe17,PrYaZa22} around the  F\"assler--Orponen conjecture \cite{FaOr14} in dimension $d \leq 3$). Among other consequences, our main result will imply a new Marstand-type projection theorem for typical self-affine sets, stratified with respect to a filtration of subvarieties in Grassmannianns. We will also give new examples of algebraic varieties which can occur inside the set of exceptional projections of a self-affine set (see Remark \ref{re:varieties} in \S\ref{ss:examples} below).

We now proceed to state our principal results. The more technical statements which underlie them will be discussed in detail in \S \ref{se:big-intro}.

\subsection{A Falconer-type theorem for projections of self-affine sets}

The following is the first and main result of this article.
\label{notation:GLW}
\begin{introtheorem}\label{intro:stratified}
Let $\A=(A_i)_{i \in \I} \in \GL_d(\R)^\I$ and suppose that $\max_{i \in \I} \threebar{A_i}<1$ for some norm $\threebar{\cdot}$ on $\R^d$. For every $k=0,\ldots,d$ there exist an integer $m \geq 1$, a finite filtration $\emptyset = \mathcal{W}_{m+1}\subset \mathcal{W}_m \subset \cdots \subset \mathcal{W}_1=\Gr(k,d)$ of algebraic varieties each invariant under the linear algebraic group generated by $\A$, and real numbers $s_m<\cdots < s_1\leq k$ with the following properties:
\begin{enumerate}[(a)]
    \item 
For every affine iterated function system with linearisation $\A$, the attractor $X$ satisfies $\dimh Q_UX \leq s_j$ for every $U \in \mathcal{W}_j \setminus \mathcal{W}_{j+1}$.
\item \label{it:strat-b}
Suppose additionally that $\max_{i,j \in \I\colon i\neq j} \threebar{A_i}+\threebar{A_j}<1$. Then for every $U \in \mathcal{W}_j \setminus \mathcal{W}_{j+1}$, for Lebesgue almost every affine iterated function system with linearisation $\A$, the attractor $X$ satisfies $\dimh Q_UX = s_j$.
\end{enumerate}
\end{introtheorem}
This result strictly generalises Falconer's aforementioned theorem, which corresponds to the special case $k=d$. 

It follows from Theorem \ref{intro:stratified}\ref{it:strat-b} together with a simple slicing argument via the Fubini-Tonelli theorem that for Lebesgue almost every affine IFS with linearisation $\A$, the dimension of the attractor's image under a projection in the class $\mathcal{W}_j$ is almost surely constant, being equal to $s_j$, with respect to Lebesgue measure on $\mathcal{W}_j$. For such sets $X$ an interesting stratified extension of Marstrand's theorem (stated below as Corollary \ref{co:stratified}) therefore holds, in which the dimensions of projections are Lebesgue almost surely constant within the filtration of algebraic varieties of projections described by Theorem \ref{intro:stratified}.

Before proceeding further, we include some remarks on the objects considered in Theorem \ref{intro:stratified}:

\begin{remark}\label{rk.ThmA}
    1. (On the filtration $\mathcal{W}_j$.) Let $G$ \label{notation:G} denote the Zariski-closure of the semigroup generated by matrices in the tuple $\A$. The subvarieties $\mathcal{W}_j$ appearing in the statement are Schubert-type subvarieties: each is defined as the collection of subspaces $U$ such that for a certain $1 \leq \ell \leq d-1$ the kernel of the $\ell^{th}$-exterior power  $Q_{U}^{\wedge \ell}$ \label{notation:Awedgek} of $Q_U$ contains a $G$-invariant subspace $W$ of $\wedge^\ell \R^d$. In particular, if $G$ acts irreducibly on each exterior power $\wedge^\ell \R^d$ for $\ell=1,\ldots,d-1$, then the filtration $(\mathcal{W}_j)$ is trivial ($m=1$). This same fact is also seen to hold if $\A$ consists of similarity transformations. The conclusion in these both cases is that for every projection, one gets the same dimension $s_1$, for almost every affine iterated function system with linearisation $\A$. To the best of our knowledge, even these conclusions in these particular cases are new.
    \\[3pt] 
   2. (On the exponents $s_j$.)
   Each exponent $s_j$ appearing in the statement is the Poincaré exponent of an explicit series similar to that which characterises the affinity dimension itself.\\[3pt]
   3. (Completely reducible case.)
   In the case where the group $G$ is reductive (i.e. if $G$ acts on $\R^d$ by block-diagonal matrices with irreducible blocks) then the varieties $\mathcal{W}_j$ can be defined in terms of a highest-weight subrepresentation of an exterior power which maximises a certain pressure functional, and the exponents $s_j$ can be expressed via the rate function of the large deviations principle for random matrix products considered in \cite{sert.LDP}. This result will be the topic of a further article \cite{morris-sert.LDP}.
\end{remark}

Exceptional projections of self-affine sets appear to have been first investigated in unpublished research by Furstenberg in the late 1970s which is described in \cite{Ke97}. Furstenberg constructed exceptional projections of self-similar sets by considering those projections whose images are themselves self-similar sets of lower affinity dimension. To date this mechanism (including minor variations such as in \cite{Fa16,Fr12}) appears to be the only known method for constructing exceptional projections of self-affine sets, and a substantial body of work has developed exhibiting cases where this mechanism cannot apply and in which it can be demonstrated that no exceptional projections exist \cite{AlSh24,FaJi14,FeJoSh10,Ho14,PeSh09,Ra24}. The results of this article present a new  general mechanism for the construction of large classes of exceptional projections. In particular we will be able to exhibit open sets of tuples $\A \in G^\I$ admitting non-trivial filtrations $(\mathcal{W}_j)_{j=1,\ldots,m}$ in Theorem \ref{intro:stratified}, for Lie groups $G\leq \GL_d(\R)$ satisfying some simple and explicit conditions (see \S \ref{ss:examples} below).

\subsection{Fractal sumsets}
An unpublished conjecture formulated by Furstenberg in the 1980s stated that if $X,Y \subset [0,1]$ are compact sets which are invariant with respect to multiplication modulo 1 by $p$ and $q$ respectively, and if $\frac{\log p}{\log q}$ is irrational, then the Hausdorff dimension of the sumset
\[X+tY:=\left\{x+ty \colon x\in X\text{ and }y \in Y\right\}\]
should be equal to the trivial upper bound $\min\{1, \dimh X+ \dimh Y\}$ 
 which arises from the observation that $X+tY$ is a linear projection of the product set $X \times Y$. Equivalently, the only exceptional projections of $X \times Y$ should be the co-ordinate projetions. Thus the \emph{arithmetic independence} of $p$ and $q$ should imply a kind of \emph{geometric independence} of the sets $X$ and $Y$. This conjecture was one of several aimed at capturing the notion that it should be very difficult for an irrational number to be non-normal with respect to two independent number bases.  Furstenberg's conjecture was proved in the special case where $X$ and $Y$ are self-similar by Peres and Shmerkin in \cite{PeSh09} and in the general case by Hochman and Shmerkin in \cite{HoSh12}. These results have led to a more general expectation that when two sufficiently structured fractal sets $X,Y \subset \R^d$ have the property $\dimh (X+Y) <\min\{d, \dimh X+ \dimh Y\}$, there should be an arithmetic degeneracy relating the structure $X$ with that of $Y$. In the recent article \cite{Py24}, for example, A. Py\"or\"al\"a has proved that if the sumset of two planar self-affine sets $X,Y$ each satisfying certain strong irreducibility, proximality and separation conditions 
 has Hausdorff dimension smaller than the trivial bound $\min \{2,\dimh X+ \dimh Y\}$ then the logarithms of certain eigenvalues of the linearisations of the underlying affine transformations must be contained in a lattice $\alpha\Z \subset\R$.

 In this article we are able to construct robust families of self-affine sets which have small sumsets while lacking any of the arithmetic degeneracy or ``resonance'' constraints considered in works such as \cite{HoSh12,PeSh09,Py24} and without the corresponding linearisations preserving any proper subspaces of $\R^d$.

\begin{introtheorem}\label{intro:sumset}
Let $d_1,d_2 \geq 2$ and let $G<\GL_{d_1d_2}(\R)\simeq \GL(\R^{d_1}\otimes \R^{d_2})$ denote the group of all linear maps of the form $g\otimes h$ where $g \in \GL_{d_1}(\R)$ and $h \in \GL_{d_2}(\R)$. Then for all large enough finite sets $\I,\J$ there exists a nonempty open set of $(\A,\B) \in G^\I\times  G^\J$ with the following property: for Lebesgue almost every IFS $(T_i)_{i \in \I}$ with linearisation $\A$, and for Lebesgue almost every IFS $(T_j')_{j \in \J}$ with linearisation $\B$, the attractor $X$ of $(T_i)_{i \in \I}$ and the attractor $Y$ of $(T_j')_{j \in \J}$ satisfy
\[\dimh (X+Y) <\dimh X+\dimh Y <d_1d_2.\] 
\end{introtheorem}

If the tuples $\A,\B$ are allowed to be reducible, i.e.\ if they leave some non-trivial linear subspaces invariant, one can trivially construct such examples with small sumsets: for example, the product of an interval with a small Cantor set is a self-affine set of this type, and the sumset of such a set with itself obviously has small dimension due to dimension saturation of the interval. The underlying principle of such examples is simply that a set consisting of a low-dimensional bundle of dimension-saturated subspaces should have small sumset with itself, which is reminiscent of certain analogous matters in additive combinatorics: a subset of $\Z_2^d$ can have small sum set with itself only when it is contained in a union of a small number of additive cosets \cite{GoGrMaTa24}.  In our examples the dimension drop instead arises from the existence of low-dimensional bundles of linear subspaces which have nontrivial pairwise intersections, a mechanism which (similarly to the geometric resonance of \cite{PeSh09}) lacks  any obvious analogue in the additive-combinatorial context.

The method underlying Theorem \ref{intro:sumset} is flexible and allows for the construction of many other classes of examples. The development of a more comprehensive theory of sum sets of self-affine sets will be the subject of future research.

\subsection{Projections of measures}

If $\nu$ is a Borel probability measure on $\R^d$, denoting by $\mathbf{B}_d(x,r)$ the \label{notation:ball} closed ball in $\R^d$ with radius $r$ and centre $x$, we recall that the upper and lower local dimensions of $\nu$ at a point $x\in \supp \nu$ are respectively defined to be 
\[\dimlocu(\nu,x):=\limsup_{r \to 0}\frac{\log \mathbf{B}_d(x,r)}{\log r},\qquad  \dimlocl(\nu,x):=\liminf_{r \to 0}\frac{\log \mathbf{B}_d(x,r)}{\log r}.\]
If the two quantities are $\nu$-a.e.\ equal and constant, the measure $\nu$ is said to be \emph{exact-dimensional} and this common value is referred to simply as the dimension of the measure $\nu$. It is an established principle in the dimension theory of dynamical systems that ``dynamically natural'' ergodic invariant measures should be exact-dimensional, at least when the system is sufficiently smooth and when no zero Lyapunov exponents are present: in smooth ergodic theory this idea is manifested for example in the Eckmann-Ruelle conjecture for hyperbolic measures \cite{EcRu85} which was proved by Barreira, Pesin and Schmeling in \cite{BaPeSc99}. The attractor of an affine iterated function system $(T_i)_{i \in \I}$, on the other hand, may be naturally equipped with invariant measures by the following procedure. For each $(i_k)_{k=1}^\infty \in \I^\N$, the contractivity of $(T_i)$ guarantees that the limit
\[\Pi[(i_k)_{k=1}^\infty]:=\lim_{n \to \infty} T_{i_1}T_{i_2}\cdots T_{i_n}v\]
exists for all $v \in \R^d$ and is independent of the value of $v$. The resulting \emph{coding map} $\Pi \colon \I^\N \to X$ is then continuous with respect to the product topology on $\I^\N$. Given any probability measure $\mu$ on $\I^\N$ which is invariant with respect to the shift transformation $(i_k)_{k=1}^\infty\mapsto (i_{k+1})_{k=1}^\infty$, we may then consider the measure $\Pi_*\mu$ as an invariant measure on the attractor $X$. If the measure $\mu$ is ergodic then $\Pi_*\mu$ is known to be exact-dimensional by a theorem of D.-J.~Feng \cite{Fe23}, which is in effect the analogue for self-affine sets of the Eckmann-Ruelle conjecture.

The techniques developed in this article allow us to demonstrate a surprising fragility in Feng's result: its conclusions do not apply to the \emph{projections} of invariant measures on self-affine sets.
\begin{introtheorem}\label{intro:not-exact-dimensional}
For every $d \geq 2$ there exists an irreducible affine iterated function system $(T_i)_{i \in \I}$ acting on $\R^d$ which admits a unique and ergodic invariant measure $\Pi_*\mu$ whose dimension equals that of the attractor, and such that there exist projections $Q$ of $\R^d$ onto lower-dimensional subspaces such that the measure $Q_*\Pi_*\mu$ is not exact-dimensional.
\end{introtheorem}
In the case $d=2$ the projection $Q$ can be taken to be orthogonal projection onto either of the two co-ordinate axes of $\R^2$. Thus while the distribution of $(x,y) \in \R^2$ with respect to $\Pi_*\mu$ is exact-dimensional, the distributions of $x$ and $y$ individually are not.
In every even dimension $d:=2k \geq 4$ one may construct cases of Theorem \ref{intro:not-exact-dimensional} in which the set of rank-$k$ orthogonal projections such that $Q_*\Pi_*\mu$ is not exact-dimensional includes an algebraic variety of dimension $\frac{1}{2}k(k-1)$, and such that the linearisation of $(T_i)_{i\in \I}$ is additionally \textit{strongly irreducible}. We mention in passing that in all of these cases the measure $\mu$ is not only ergodic, but has the strongest qualitative mixing property possible in ergodic theory: its natural extension is measurably isomorphic to a Bernoulli process.

As with Theorem \ref{intro:sumset}, the example of Theorem \ref{intro:not-exact-dimensional} is \textit{robust} in the sense that the conclusion holds for Lebesgue almost every system with an appropriate linearisation, and for an open (for the Hausdorff topology) 
set of linearisations $\A$ within a certain algebraic subvariety of $\GL_d(\R)^\I$. 
The example of Theorem \ref{intro:not-exact-dimensional} is also compatible with the \emph{strong open set condition}: by including a very large number of affinities with the same linearisation, the system as above may be constructed in such a way that there exists an open neighbourhood $U$ of the attractor such that the images $T_iU$ are pairwise disjoint for distinct $i$. By a standard construction (see e.g. \cite[\S3]{ChPe10}) this implies the existence of a smooth expanding map $f \colon \bigcap_{i \in \I} T_i^{-1}U \to U$ which admits the attractor as a repelling set and $\Pi_*\mu$ as its unique invariant measure of maximal dimension. In this context the projection $Q$ can be interpreted as a smooth observable $\phi \colon U \to \R^k$ whose law with respect to the unique invariant measure of maximum dimension on the repelling set is not itself exact-dimensional.

\subsection{Recent related work}At the time that this work was being completed an alternative proof of a version of Theorem \ref{intro:stratified}
was presented by D.-J. Feng and Y.-H. Xie in \cite{FeXi25}, together with a result analogous to the case $d=2$ of Theorem \ref{intro:not-exact-dimensional}. Another version of the case $d=2$ of Theorem \ref{intro:not-exact-dimensional}, in this case allowing explicit translation parameters, will appear in forthcoming work by D.~Allen, A.~K\"aenm\"aki, R.~D.~Prokaj, K.~Simon and S.~Troscheit in \cite{AlKaPrSiTr25}.

\section{The setting and the main results}\label{se:big-intro}

\subsection{Fundamental notation}
In this section we will present the setting and the main results of this article in greater detail and generality. We first establish some fundamental notation. A detailed guide to the notations used in this article is given in Appendix \ref{ap:notation}.

If $\I$ is a nonempty finite set, \label{notation:I} a \emph{word} over $\I$ is defined to be a finite sequence $(i_k)_{k=1}^n$ of elements of $\I$, \label{notation:icylnd} which we write in the form $\iii=i_1 i_2 \cdots i_n$. \label{notation:ijk} We define the \emph{length} of $\iii=i_1\cdots i_n$, denoted $|\iii|$, \label{notation:|i|} to be the integer $n$. We denote the set of all words over $\I$ by $\Gamma_\I$, \label{notation:Gamma_I} and we define the \emph{concatenation} of two words $\iii=i_1\cdots i_n$, $\jjj=j_1\cdots j_m$ by $\iii\jjj:=i_1\cdots i_n j_1\cdots j_m$. It will at times be useful to us to regard  $\Gamma_\I$ as a semigroup equipped with this operation. 
By contrast we will use the notation $\underline{\iii}$ to indicate an \emph{infinite} sequence over the alphabet $\I$, i.e. an element of $\I^\N$. Given $\underline{\iii}=(i_k)_{k=1}^\infty \in \I^\N$ \label{notation:underlinei} and $n \geq 1$ we define $\underline{\iii}|_n:=i_1\cdots i_n \in \Gamma_\I$. \label{notation:i_n} We equip $\I^\N$ with the product topology (with respect to which it is compact and metrisable) and define the \emph{shift transformation} $\sigma \colon \I^\N \to \I^\N$ by $\sigma[(i_k)_{k=1}^\infty]:=(i_{k+1})_{k=1}^\infty$, which is a continuous transformation. We let $\mathcal{M}_\sigma(\I^\N)$ \label{notation:M_sigma} denote the set of all shift-invariant Borel probability measures on $\I^\N$. 
For each  $\jjj \in \Gamma_\I$ we define $[\jjj]$ to be the set $\{\underline{\iii} \in \I^\N \colon \underline{\iii}|_{|\jjj|}=\jjj\}$, which we call the \emph{cylinder set} corresponding to the word $\jjj$. 

If an affine iterated function system $(T_i)_{i \in \I}$ or a tuple of linear maps $(A_i)_{i \in \I}$ is understood, then for each $\iii=i_1\cdots i_n \in \Gamma_\I$ we write $T_\iii := T_{i_1}\cdots T_{i_n}$ and $A_\iii:=A_{i_1}\cdots A_{i_n}$. We define the \emph{coding map} of a given affine iterated function system $(T_i)_{i \in \I}$ acting on $\R^d$ to be the continuous function $\Pi \colon \I^\N \to \R^d$ which  satisfies
\[\Pi[(i_k)_{k=1}^\infty]:=\lim_{n \to \infty} T_{i_1}T_{i_2}\cdots T_{i_n}v\]
for every $(i_k)_{k=1}^\infty \in \I^\N$ and $v \in \R^d$.

\subsection{Thermodynamic formalism for linear images of self-affine sets}

To motivate the results to come, we will first recall some established methods from the dimension theory of self-affine sets. We begin with the dimension bound $\dimh X \leq \dimaff \A$ introduced in Falconer's article \cite{Fa88}. Let $(T_i)_{i \in \I}$ be an affine iterated function system with linearisation $\A=(A_i)_{i \in \I}$. Since the Hausdorff dimension of the attractor $X$ is by definition equal to the infimal value of $s \geq 0$ such that for every $\delta>0$ the quantity
\[\mathcal{H}_\delta^s(X):=\inf\left\{\sum_{n=1}^\infty (\diam U_n)^s \colon X\subseteq \bigcup_{n=1}^\infty U_n \text{ and }\sup_{n \geq 1} \diam U_n \leq \delta\right\}\]
is finite, in order to bound the Hausdorff dimension of $X$ one seeks a natural covering of the attractor by sets of small diameter. If $\mathbf{B}\subset \R^d$ is a closed ball containing $X$, then clearly
\begin{equation}\label{eq:falconer-cover}X=\bigcup_{|\iii|=n} T_\iii X \subseteq \bigcup_{|\iii|=n} T_\iii \mathbf{B}\end{equation}
for every $n \geq 1$, and by taking $n$ large enough we may arrange for this covering of $X$ to consist of sets of diameter at most $\delta$. In general the sets $T_\iii \mathbf{B}$ are ellipsoids, which makes their diameters an inexact parameter with which to measure their $s$-dimensional volume, and to compensate for this inefficiency Falconer argued as follows. Recall that the \emph{singular values} of $A \in \End(\R^d)$ 
\label{notation:EndW} are defined to be the non-negative square roots of the eigenvalues of the positive semidefinite transformation $A^TA$, and we denote these by $\sigma_1(A)\geq \sigma_2(A)\geq \cdots \geq \sigma_d(A)$, \label{notation:sing.values} repeating in the case of multiple eigenvalues. The image of the closed unit ball of $\R^d$ under $A$ is a closed ellipsoid with semiaxes $\sigma_1(A),\ldots,\sigma_d(A)$. We now define
\[\varphi^s(A):=
\sigma_1(A)\cdots \sigma_{\lfloor s\rfloor}(A) \sigma_{\lceil s\rceil}^{s-\lfloor s\rfloor}(A)\]
for all $s \in [0,d]$; \label{notation:varphis} the inequality $\varphi^s(AB)\leq \varphi^s(A)\varphi^s(B)$ is shown in \cite{Fa88} to hold for all $A,B \in \End(\R^d)$. One may show that an ellipsoid of the form $T_\iii \mathbf{B}$ may be covered with cubes $C_1,\ldots,C_k$ of suitable diameters in such a way that $\sum_{i=1}^k (\diam C_i)^s=O(\varphi^s(A_\iii))$. This being shown, from \eqref{eq:falconer-cover} one obtains directly the bound
\[\dimh X \leq \min \left\{d, \inf\left\{s \in [0,d] \colon \sum_{n=1}^\infty \sum_{|\iii|=n} \varphi^s(A_\iii)<\infty\right\} \right\}\]
with the usual convention $\inf \emptyset=\infty$. The expression on the right-hand side of this inequality is defined to be the affinity dimension of $\A$ and is denoted $\dimaff \A$. This quantity is independent of the translation parts $u_i \in \R^d$ of the affinities $T_ix \equiv A_ix + u_i$.

The argument so far generalises to linear images $QX$ of the attractor in a surprisingly straightforward way. If a linear map $Q \in \End(\R^d)$ is specified, 
then for every $n \geq 1$
\[QX = \bigcup_{|\iii|=n} QT_\iii \subseteq \bigcup_{|\iii|=n} QT_\iii \mathbf{B}.\]
The sets $QT_\iii \mathbf{B}$ are again  closed ellipsoids (though in general degenerate) and may be covered by following Falconer's strategy to obtain the almost identical bound
\[\dimh QX \leq \min\left\{\rank Q, \inf\left\{s \in [0,d] \colon \sum_{n=1}^\infty \sum_{|\iii|=n} \varphi^s(QA_\iii)<\infty\right\}\right\}.\]
This suggests the possibility of defining the latter quantity to be the ``$Q$-projected affinity dimension'' of $\A$ and developing a theory of projections of self-affine sets along the same lines 
as the theory of self-affine sets themselves as sketched above.
However, the required arguments diverge significantly when it comes to implementing the lower bound.

To show that $\dimh X = \dimaff \A$ for Lebesgue almost every system with linearisation $\A$, Falconer proceeds in \cite{Fa88} to construct for given $s \in(0,\dimaff \A)$ a measure $\mu$ on $\I^\N$ with the property $\mu([\iii])\leq C\varphi^s(\iii)$ for all $\iii \in \Gamma_\I$ and some constant $C>0$. By integration over the space of translation parameters it is shown that for Lebesgue almost every choice of translation parameter, the double integral $\int_{\I^\N}\int_{\I^\N} \|\Pi_*\mu(\underline{\iii})-\Pi_*\mu(\underline{\jjj})\|^{-s}d\mu(\underline{\iii})d\mu(\underline{\jjj})$ converges, implying that the Hausdorff dimension (\emph{i.e.} the essential infimum of the lower local dimension) of $\Pi_*\mu$ is at least $s$. Hence the dimension of the corresponding attractor must also be at least $s$, and by applying this argument for a countable dense set of $s \in (0,\dimaff\A)$ the result is proved. While in Falconer's original argument the measure $\mu$ is constructed from an outer measure, later extensions of the argument in \cite{JoPoSi07,Ka04} took $\mu$ to be a shift-invariant ergodic measure on $\I^\N$ such that
\begin{equation}\label{eq:classical-0}
\lim_{n \to\infty} \frac{1}{n} \log \left(\frac{ \varphi^s(A_{\underline{\iii}|_n})}{\mu([\underline{\iii}|_n])}\right)=0    
\end{equation}
$\mu$-a.e, where $s$ is the affinity dimension, and where we suppose for the moment that $s< d$. (Via the subadditive ergodic theorem and the Shannon-McMillan-Breiman theorem, the more usual form of the above expression is
\[h(\mu) + \lim_{n \to \infty} \frac{1}{n}\int_{\I^\N} \log \varphi^s(A_{\underline{\iii}|_n})d\mu(\underline{\iii})=0,\]
where $h(\mu)$ is the Kolmogorov-Sinai entropy of $\mu$, \label{notation:hmu} but we shall prefer the former for reasons which will later become evident.) This suggests the possibility of likewise constructing a measure $\mu$ such that either $\mu([\iii]) \leq C\varphi^s(QA_\iii)$ for all $\iii$ as in \cite{Fa88}, or such that $\mu$ is shift-invariant and for almost every $\underline{\iii}$,
\begin{equation}\label{eq:q-nonsense}
 \lim_{n \to\infty} \frac{1}{n} \log \left(\frac{ \varphi^s(QA_{\underline{\iii}|_n})}{\mu([\underline{\iii}|_n])}\right)=0   
\end{equation}
along the lines of \cite{JoPoSi07,Ka04}, followed by a naturally modified version of the multiple integration argument applied in \cite{Fa88}. (When the affinity dimension is equal to $d$ one instead works with a measure such that \eqref{eq:classical-0} is non-negative for $s=d$, and a similar consideration would apply to \eqref{eq:q-nonsense} when $s=\rank Q$.)

However, such a direct argument cannot possibly work. Theorem \ref{intro:not-exact-dimensional} already indicates that we cannot in general expect the essential supremum of the local dimension of a projected invariant measure to equal the essential infimum of the same. Since the double integral mentioned earlier estimates the latter and not the former, it will in general underestimate the dimension of the attractor and cannot be expected to produce sharp results. A more precise articulation of this obstacle is that the limit \eqref{eq:q-nonsense} need not be a.e.\ constant even when the measure is ergodic. These considerations will oblige us to work with local rather than global dimension properties of measures. 
Furthermore, whereas the limit \eqref{eq:classical-0} is accessible to the subadditive ergodic theorem in view of the subadditivity relation $\log\varphi^s(AB) \leq \log \varphi^s(A)+\log\varphi^s(B)$, no such inequality holds in general for the map $A \mapsto \varphi^s(QA)$, making the very existence of the limit \eqref{eq:q-nonsense} less obvious. This lack of subadditivity likewise prevents Falconer's arguments in \cite[\S4]{Fa88} from being directly extended to the construction of a non-invariant measure satisfying $\mu([\iii])=O(\varphi^s(QA_\iii))$. To obtain our lower bound, therefore, we will need to first conduct a thorough investigation of the shift-invariant measures appropriate to our problem and of the inherent behaviour of expressions such as \eqref{eq:q-nonsense}. We will need in particular to establish the existence of certain key limits despite the \emph{a priori} absence of the subadditivity  properties found in the classical theory.

We will proceed by embedding the above-described objects in a broader thermodynamic formalism as follows. We first recall the analogous constructions which have been established for the purpose of studying the attractor itself. For each $s \in [0,d]$ one may define the \emph{pressure} of $\A$ at $s$,
\begin{equation*}\label{notation:P(A,s)}
P(\A,s):=\lim_{n \to \infty}\frac{1}{n}\log \sum_{|\iii|=n} \varphi^s(A_\iii),\end{equation*}
a definition which dates back to \cite{Fa88}. For fixed arbitrary $\A \in \GL_d(\R)^\I$ this function is continuous with respect to $s$, and if $\A$ is contracting then it is also strictly decreasing; if it has a zero in $[0,d]$ then it is unique and it is precisely the affinity dimension of $\A$. It is also possible to demonstrate that
\begin{equation}\label{eq:classical-2}
P(\A,s)=\sup_{\mu \in \mathcal{M}_\sigma(\I^\N)} \left[h(\mu)+\lim_{n \to \infty}\frac{1}{n}\int_{\I^\N} \varphi^s(A_{\underline{\iii}|_n}) d\mu(\underline{\iii})\right]\end{equation}
for every $s \in [0,d]$, and this formulation is fundamental in establishing a range of properties of the affinity dimension itself and of high-dimensional measures on the attractor (see e.g. \cite{BoMo18,FeSh14,MoSe23a}). Hereafter we will refer to measures attaining the supremum \eqref{eq:classical-2} as \emph{$\varphi^s$-equilibrium states}; if $\mu \in \mathcal{M}_\sigma(\I^\N)$ is ergodic, its \emph{Lyapunov dimension} with respect to $\A$, denoted $\dimlyap (\mu,\A)$, \label{notation:dimlyapAmu} may be defined as the unique $s$ such that \eqref{eq:classical-0} holds $\mu$-a.e, or as $d$ if the limit is a.e. non-negative for all $s\in [0,d]$.\

Our first main result extends these notions to the context of projections of self-affine sets, as follows.
\begin{theorem}
\label{th:pressure-detail}
Let $\I$ be a nonempty finite set, let $d\geq 1$ and let $\A=(A_i)_{i \in \I} \in \GL_d(\R)^\I$. Then:
\begin{enumerate}[(a)]
\item\label{it:exist-subadditive}
\emph{Approximate subadditivity.}
For every $Q \in \End(\R^d)$ there exists $C>0$ depending only on $Q$ and $\A$ such that the sequence
\[\left( \log\left(C  \sum_{|\iii|=n} \varphi^s(QA_\iii)\right)\right)_{n=1}^\infty\]
is subadditive for every $s \in [0, \rank Q]$, and in particular the limit \label{notation:PQ(A,s)}
\[P_Q(\A, s):=\lim_{n \to \infty}\frac{1}{n}\log \sum_{|\iii|=n} \varphi^s(QA_\iii) \]
exists for all $s \in [0, \rank Q]$.

\item \label{it:thermodynamic}
\emph{Variational properties.}
For every $Q \in \End(\R^d)$ and every invariant measure $\mu \in \mathcal{M}_\sigma(\I^\N)$, for $\mu$-a.e.\ $\underline{\iii} \in \I^\N$ the limit
\[\lim_{n \to \infty} \frac{1}{n} \log \left(\frac{\varphi^s(QA_{\underline{\iii}|_n})}{\mu([\underline{\iii}|_n])}\right)\]
exists for all $s \in [0,\rank Q]$. For fixed $s \in [0,\rank Q]$
we moreover have
\begin{equation}\label{eq:variational-main}P_Q(\A,s)=\sup_{\mu \in \mathcal{M}_\sigma(\I^\N)} {\ess\sup}_{\mu,\underline{\iii}} \lim_{n \to \infty} \frac{1}{n} \log \left(\frac{\varphi^s(QA_{\underline{\iii}|_n})}{\mu([\underline{\iii}|_n])}\right).\end{equation}
Given nonzero $Q \in \End(\R^d)$ and $s \in [0,\rank Q]$, we call an ergodic measure a \emph{$(\varphi^s,Q)$-equilibrium state for $\A$} if it attains the above supremum. This class of measures has the following properties:
\begin{enumerate}[(i)]
    \item \label{it:eqm-at-least-one}
    For fixed nonzero $Q \in \End(\R^d)$ and $s \in [0,\rank Q]$ the number of $(\varphi^s,Q)$-equilibrium states is at least one. 
    \item \label{it:eqm-finitely-many}
    For every $s \in (0,d]$ the set
    \[\bigcup_{\substack{Q \in \End(\R^d)\\ \rank Q \geq s}} \left\{\mu\in \mathcal{M}_\sigma(\I^\N) \colon \mu\text{ is a }(\varphi^s,Q)\text{-equilibrium state for }\A\right\}\]
    is finite, and has cardinality not greater than 
 ${d\choose {\lfloor s\rfloor}}{d\choose{ \lceil s\rceil}}$ if $s$ is non-integer, or ${d \choose s}$ if $s$ is integer.
 \item \label{it:eqm-gibbs}
 If $\mu$ is a $(\varphi^s,Q)$-equilibrium state then it satisfies a Gibbs inequality of the form 
\[C^{-1}\leq \frac{\mu([\iii])}{e^{-|\iii|P(\Psi)}\Psi(\iii)} \leq C\]
for some constant $C\geq 1$ and function $\Psi \colon \Gamma_\I \to (0,\infty)$  such that $\Psi(\iii\jjj)\leq \Psi(\iii)\Psi(\jjj)$ for all finite words $\iii,\jjj$. 
\item \label{it:eqm-psi-mixing}
If $\mu$ is a $(\varphi^s,Q)$-equilibrium state then we may write $\mu=\frac{1}{m}\sum_{k=0}^{m-1}(\sigma^k)_*\nu$ for some $\sigma^m$-invariant Borel probability measure $\nu$ on $\I^\N$ satisfying the following \emph{$\psi$-mixing condition} with respect to $\sigma^m$:
\begin{equation*}
\lim_{n \to \infty} \sup_{\substack{\iii,\jjj \in \Gamma_\I\\ m\text{ divides }|\iii|, |\jjj|}} \left|\frac{\nu([\iii]\cap \sigma^{-|\iii|-nm}[\jjj])}{\nu([\iii])\nu([\jjj])}-1\right|=0.\end{equation*}
Here $1 \leq m\leq {d\choose \lfloor s\rfloor}{d\choose \lceil s\rceil}$ if $s$ is non-integer, or $1 \leq m \leq {d \choose s}$ if $s$ is integer.

Additionally, if $G\leq \GL_d(\R)$ denotes the Zariski closure of the semigroup $\{A_\iii \colon \iii \in \Gamma_\I\}$ and $G^o$ \label{notation:Go} the Zariski-connected component of $G$ which contains the identity, if there exists an integer $n$ such that $\bigcup_{|\iii|=n}G^oA_\iii =G$, then $\mu$ itself is $\psi$-mixing and we may take $m=1$ in the preceding description. In all cases $m$ divides the index of $G^o$ in $G$.
\end{enumerate}

\item\label{it:sublevel-algebraic}
\emph{Algebraicity of sub-level sets with respect to $Q$.}
For every $s \geq 0$ and $t \in \R$ define
\[\mathcal{V}_{s,t}:=\left\{Q \in \End(\R^d) \colon P_Q(\A, s)\leq t\right\}.\]
Then each $\mathcal{V}_{s,t}$ is an affine subvariety of
$\End(\R^d)$ which satisfies $\mathcal{V}_{s,t}A_i=\mathcal{V}_{s,t}$ for every $i \in \I$. Furthermore
\[\left|\left\{\mathcal{V}_{s,t} \colon s \in [0,d], t \in \R\right\}\right| \leq \prod_{k=1}^d \left(1+2^{{d\choose k}}\right).\]
\item\emph{Continuity properties with respect to $s$.}\label{it:lipschitz-convexity}
For fixed $Q \in \End(\R^d)$ the function $s \mapsto P_Q(\A, s)$ is Lipschitz continuous on the interval $[0, \rank Q]$ and is convex on every interval of the form $[\ell-1,\ell] \subseteq [0,\rank Q]$ with $\ell$ integer. If every $A_i$ is contracting with respect to some norm on $\R^d$ then additionally $s \mapsto P_Q(\A, s)$ is strictly decreasing on $[0, \rank Q]$.
\item \label{it:monotone-A}
\emph{Strict monotone dependence on $\A$.} Suppose that $\emptyset \subset \J\subset \I$, and define $\A'=(A_i)_{i \in \J}$.  Then for every $Q \in \End(\R^d)$ and $s \in [0,\rank Q]$ we have $P_Q(\A',s)<P_Q(\A,s)$. 
\item\label{it:monotone-ker-q} 

\emph{Monotone dependence on $\ker Q$.} If $Q_1, Q_2 \in \End(\R^d)$ satisfy $\ker Q_1 \subseteq \ker Q_2$ then $P_{Q_1}(\A, s) \geq P_{Q_2}(\A, s)$
for all $s \in [0, \rank Q_2]$. In particular, if $\ker Q_1 = \ker Q_2$ then $P_{Q_1}(\A, s) = P_{Q_2}(\A, s)$  for all $s \in [0, \rank Q_2]=[0, \rank Q_1]$.
\end{enumerate}
\end{theorem}

\emph{Contextual remarks.} 1. In the case where $Q$ is the identity, Theorem \ref{th:pressure-detail} recovers a range of results from works such as \cite{BoMo18,Fa88,FeKa11,Ka04,KaMo18,Mo18a,Mo21,Pi20} dealing with the existence of the limit, continuity in $s$, the properties of $\varphi^s$-equilibrium states and strict monotonicity. In the situation where $Q$ is non-invertible our potentials in general fail to be subadditive, introducing complications (and consequences) not encountered in any of those works: even the existence of the various limits in Theorem \ref{th:pressure-detail}\ref{it:exist-subadditive}--\ref{it:thermodynamic}, which in the $Q=\id$ case follows straightforwardly from Fekete's lemma and the subadditive ergodic theorem, requires further non-trivial analysis of potentials in our case. \\[3pt] 
2. A second new difficulty as compared to the previous works arises as follows. The analyses of $\varphi^s$-equilibrium states in \cite{BoMo18,FeKa11,KaMo18,Mo21} rely on the fact that in that context (i.e. when $Q=\id$), the pressure $P(\A,s)$ and the ergodic properties of $\varphi^s$ can be shown to depend only on the reductivisation of $\A$ (i.e.\ on irreducible block-diagonals appearing in relevant exterior power representations). However, in our analysis of $P_Q(\A,s)$ and of its associated equilibrium states this reduction is unavailable since the removal of off-diagonal blocks in general alters the value of $P_Q(\A,s)$. This can be seen in simple examples such as $\A=(A_1,A_2)$ with
\begin{equation*}
A_1=\begin{pmatrix}1&1 \\ 0&1\end{pmatrix},\qquad A_2=\begin{pmatrix}1&1 \\ 0&2\end{pmatrix},\qquad Q:=\begin{pmatrix}1&0\\0&0\end{pmatrix}.\end{equation*}
Here for example $P_Q(\A,1)=P(\A,1)=\log 3$, but if $\A'=(A_1',A_2')$ is derived from $\A$ by setting the upper-right matrix entries to zero, i.e.
\begin{equation*}
A_1'=\begin{pmatrix}1&0 \\ 0&1\end{pmatrix},\qquad A_2'=\begin{pmatrix}1&0 \\ 0&2\end{pmatrix},\end{equation*}
then $P_Q(\A',1)=\log 2$. As a result of this technical obstruction $(\varphi^s,Q)$-equilibrium states are not in general accessible to the classic approach of \cite{BoMo18,FeKa11,KaMo18,Mo21} via removal of off-diagonal blocks. \\[3pt]
3. As will be seen in \S\ref{se:technical-core}, the proof of Theorem \ref{th:pressure-detail} 
makes use of the fact that the linearisations $(A_i)_{i \in \I}$ generate a subsemigroup of a linear group. This makes it an 
open problem to extend this result to the context of non-invertible affine iterated function systems such as those studied in \cite{barany.non.inv,kaenmaki.non.inv}, where group-theoretic techniques are not available. For a further illustration of some subtleties which can arise in analogues of Theorem \ref{th:pressure-detail} for non-invertible affinities, see \cite[Proposition 4]{Mo19b}. 

\bigskip
\emph{Technical remarks.} 1. Whereas Theorem \ref{th:pressure-detail}\ref{it:thermodynamic} guarantees the pointwise a.e.\ existence of the limit $\frac{1}{n} \log \varphi^s(QA_{\underline{\iii}|_n})$ for every invariant measure $\mu$, this limit can in general be nonconstant even when  $\mu$ is ergodic. It is precisely this phenomenon which underlies the construction of non-exact-dimensional measures in Theorem \ref{intro:not-exact-dimensional}. The failure of the limit to be constant a.e.\ arises from the possibility that the distribution of Oseledets spaces of $\A=(A_i)_{i \in \I}$ with respect to $\mu$ may in general give positive measure to the set of all subspaces contained in a given hyperplane. It is in general a subtle matter to guarantee that the limit is constant a.e.: if $\mu$ is Bernoulli then this follows from  Furstenberg--Kifer's \cite[Theorem 3.9]{furstenberg-kifer}, and under additional algebraic hypotheses on $\A$ 
this is also the case when $\mu$ has local product structure by the work of Avila--Viana (see \cite[Proposition 5.1]{avila-viana.br} and also \cite[Theorem A.1]{avila-viana.acta}). The contemporaneous work \cite{FeXi25} also ensures this for supermultiplicative measures.\\[3pt] 

2. The above phenomenon is also responsible for a key new distinction between the notions of $\varphi^s$- and  $(\varphi^s,Q)$-equilibrium states: we allow the former class of measure to be non-ergodic (in which case all of its ergodic components must be $\varphi^s$-equilibrium states) but we require $(\varphi^s,Q)$-equilibrium states to be ergodic by definition, because the presence of an essential supremum as opposed to an integral or an a.e.\ constant pointwise limit implies that a non-ergodic measure attaining the supremum \eqref{eq:variational-main} may  not have any ergodic components which attain the same supremum.\\ [3pt]
3. Theorem \ref{th:pressure-detail}\ref{it:thermodynamic}\ref{it:eqm-psi-mixing} demonstrates that $(\varphi^s,Q)$-equilibrium states which are totally ergodic must also be $\psi$-mixing, and describes precisely the situation which arises when total ergodicity fails. In this respect it generalises earlier results of \cite{Mo21,Pi20} from the context of $\varphi^s$-equilibrium states to that of $(\varphi^s,Q)$-equilibrium states. The condition $\bigcup_{|\iii|=n} A_\iii G^o =G$ is very mild, and can fail only if there exists a Zariski-closed normal subgroup $H$ of  $G$ such that $G /H$ is a finite cyclic group and $\{A_i \colon i \in \I\}$ is contained in a single coset of $H$: see Lemma \ref{le:breuillard-sert-cyclic} below. The latter situation can always be circumvented by ``re-coding'' $(T_i)_{i \in \I}$ while retaining the same attractor $X$, such as by considering $(T_\iii)_{|\iii|=p}$ in place of $(T_i)_{i \in \I}$.\\ [3pt]
4. Where the term ${d\choose k}$ arises in Theorem \ref{th:pressure-detail}, it does so as a simple upper bound for the number of distinct composition factors of $\wedge^k\R^d$ as an $\R[G^o]$-module, \label{notation:grp.alg} where $G^o$ is the identity component of the Zariski closure of the semigroup generated by $\A$: see \S\ref{se:technical-core} for a description of these more precise bounds. In general we do not expect the bounds in \ref{it:thermodynamic}\ref{it:eqm-finitely-many} and \ref{it:sublevel-algebraic} to be sharp, and the optimal bound in the latter case in particular is likely to be significantly smaller.

\subsection{Almost sure dimension formulas}

Having defined the projected singular value pressure and established its fundamental properties, we may now properly describe the resulting bounds for the dimensions of sets and measures as follows:
\begin{definition}
Let $\A=(A_i)_{i\in \I} \in \GL_d(\R)^\I$ and suppose that $\max_{i \in \I}\threebar{A_i}<1$ for some norm $\threebar{\cdot}$ on $\R^d$. For each $Q \in \End(\R^d)$ we define the \emph{$Q$-projected affinity dimension of $\A$} to be the quantity \label{notation:dimaffQ}
\[\dimaffQ \A:=\inf \left\{s\in [0,\rank Q] \colon P_Q(\A,s)<0\right\}\]
if this set is nonempty, and $\rank Q$ otherwise. Additionally, given any $\mu \in \mathcal{M}_\sigma(\I^\N)$ we further define the \emph{$Q$-projected local Lyapunov dimension of $\A$ and $\mu$ at $\underline{\iii}$} to be \label{notation:dimlyapQ}
\[\dimlyapQ (\A,\mu)(\underline{\iii}):=\inf\left\{s\in [0,
\rank Q] \colon \lim_{n \to \infty}  \frac{1}{n}\log \left(\frac{\varphi^s(QA_{\underline{\iii}|_n})}{\mu([\underline{\iii}|_n])}\right)<0\right\} \]
for every $\underline{\iii} \in \I^\N$ such that the limit exists for all $s \in [0,d]$ and the set is nonempty. If instead the limit is well-defined and non-negative for every $s \in [0,\rank Q]$ then we define this quantity to be $\rank Q$.
\end{definition}
It follows directly from Theorem \ref{th:pressure-detail} that $\dimaffQ \A$ is the unique zero in $[0,\rank Q]$ of the strictly decreasing function $s \mapsto P_Q(\A,s)$ when such a zero exists, and is otherwise equal to $\rank Q$. Theorem \ref{th:pressure-detail} also guarantees that the Lyapunov dimension is well-defined pointwise a.e. The alternative characterisations

\[\dimaffQ \A = \min \left\{ d, \inf\left\{s \in [0,d] \colon \sum_{n=1}^\infty \sum_{|\iii|=n} \varphi^s(QA_\iii)<\infty\right\} \right\},\]
\[\dimlyapQ (\A,\mu)(\underline{\iii})=\max \left\{0, \sup\left\{s\in [0,d] \colon \lim_{n \to \infty}  \frac{1}{n}\log \left(\frac{\varphi^s(QA_{\underline{\iii}|_n})}{\mu([\underline{\iii}|_n])}\right)\geq 0\right\} \right\} \]
are easily derived, with the usual convention that $\sup \emptyset=-\infty$, and assuming again that the relevant limits exist at $\underline{\iii}$. When $Q$ is taken to be the identity (or indeed any invertible linear map) these quantities reduce to the classical notions of affinity dimension and Lyapunov dimension respectively.

The following result is a simple direct consequence of Theorem \ref{th:pressure-detail}\ref{it:sublevel-algebraic}:
\begin{proposition}\label{pr:dimaffq-subvarieties}
Let $\I$ be a nonempty finite set, let $d \geq 1$, let $\A=(A_i)_{i \in \I}\in \GL_d(\R)^\I$ and suppose that $\max_{i \in \I}\threebar{A_i}<1$ for some norm $\threebar{\cdot}$ on $\R^d$. Then each level set
\[\mathcal{W}_t:=\left\{Q\in \End(\R^d) \colon \dimaffQ \A\leq t\right\}\]
is an algebraic variety which satisfies $\mathcal{W}_tA_i=\mathcal{W}_t$ for every $i \in \I$, and the number of distinct level sets is not greater than $m:=\prod_{k=1}^d (1+2^{d \choose k})$. Moreover, denoting these by $\mathcal{W}_{t_1} \supset \ldots \supset \mathcal{W}_{t_m}$, we have $\dim \mathcal{W}_{t_j}>\dim \mathcal{W}_{t_{j+1}}$ for every $j=1,\ldots,m-1$.
\end{proposition}

By Theorem \ref{th:pressure-detail}\ref{it:monotone-ker-q} the value of $\dimaffQ \A$ depends only on the kernel of $Q$. The inequality $\dimaffQ \A<\dimaff \A$ is thus intimately connected with the existence of $\A$-invariant varieties in the Grassmannian of $\R^d$; we will explore this matter further in \S\ref{ss:examples} below. 
We have not attempted to optimise the bound $\prod_{k=1}^d (1+2^{d\choose k})$, which is presented only in order to demonstrate that there exists a finite \emph{a priori} bound in terms of $d$ alone. A sharp estimate is likely to be significantly smaller, perhaps even $o(2^d)$.  

We next extend the classical dimension estimates of \cite{Fa88,JoPoSi07} as follows:
\begin{theorem}\label{th:dimension-results}
Suppose that $\A=(A_i)_{i \in \I} \in \GL_d(\R)^\I$ satisfies $\max_{i \in \I} \threebar{A_i}<1$ for some norm $\threebar{\cdot}$ on $\R^d$, let $Q \in \End(\R^d)$ and let $\mu \in \mathcal{M}_\sigma(\I^\N)$. For every $\mathbf{v}=(v_i)_{i \in \I} \in (\R^d)^\I$ \label{notation:bfv} define an affine iterated function system  $(T_i^{\mathbf{v}})_{i \in \I}$ on $\R^d$ by $T_i^{\mathbf{v}}x\equiv A_ix+v_i$, \label{notation:T_iv} and let $X^{\mathbf{v}}$ \label{notation:Xv} denote the attractor of $(T_i^{\mathbf{v}})_{i \in \I}$ and $\Pi^{\mathbf{v}}\colon \I^\N \to \R^d$ the \label{notation:Piv} associated coding map. 
\begin{enumerate}[(a)]
    \item \label{it:dimension-upper-bounds}
For every $\mathbf{v} \in (\R^d)^\I$,
\[\dimbu QX^{\mathbf{v}} \leq \dimaffQ \A,\]
\label{notation:dimbu} and for $\mu$-a.e.\ $\underline{\iii}\in \I^\N$
\[\dimlocu ((Q\Pi^{\mathbf{v}})_* \mu, Q\Pi^{\mathbf{v}}(\underline{\iii})) \leq \dimlyapQ (\A,\mu)(\underline{\iii}).\]
    \item\label{it:dimension-lower-bounds}
If additionally $\max_{i,j \in \I \colon i \neq j} \threebar{A_i}+\threebar{A_j}<1$ then the following further properties hold. For Lebesgue a.e.\ $\mathbf{v} \in (\R^d)^\I$,
    \[\dimh QX^{\mathbf{v}}=\dimb QX^{\mathbf{v}}=\dimaffQ \A,\]
and for $\mu$-a.e.\ $\underline{\iii} \in \I^\N$,
    \[\dimloc ((Q\Pi^{\mathbf{v}})_*\mu, Q\Pi^{\mathbf{v}}(\underline{\iii})) = \dimlyapQ (\A,\mu)(\underline{\iii}).\]
If $P_Q(\A,\rank Q)>0$, then $QX^{\mathbf{v}}$ has positive $(\rank Q)$-dimensional Lebesgue measure for Lebesgue a.e.\ $\mathbf{v} \in (\R^d)^\I$.  If we define
\[\Upsilon:=\left\{\underline{\iii}\in \I^\N \colon  \liminf_{n \to \infty}\frac{1}{n}\log \left(\frac{\varphi^{\rank Q}(QA_{\underline{\iii}|_n})}{\mu([\underline{\iii}|_n])}\right)>0\right\}\]
and let $\mu_\Upsilon$ denote the restriction of $\mu$ to $\Upsilon$, then for Lebesgue a.e.\ $\mathbf{v} \in (\R^d)^\I$ the measure $(Q\Pi^{\mathbf{v}})_*\mu_\Upsilon$ is absolutely continuous with respect to Lebesgue measure on $\im Q$. 
\end{enumerate}
\end{theorem}
Here $\dimbu$ denotes upper box dimension. Taking $Q$ to be the identity map recovers the main results of \cite{Fa88,JoPoSi07} as special cases. The condition $\max_{i,j \in \I\colon i\neq j} \threebar{A_i}+\threebar{A_j}<1$ is a modern relaxation of Falconer's original condition $\max_i \|A_i\|<\frac{1}{3}$ which is presented in \cite{BaSiSo23}; it cannot be removed entirely: see \cite{Ed92,Mo23,SiSo02}. Although it is conventional in the literature to consider only the Euclidean norm $\|\cdot\|$, any norm on $\R^d$ may be used.

Theorem \ref{intro:stratified} follows directly from the combination of Theorem \ref{th:dimension-results} with Proposition \ref{pr:dimaffq-subvarieties}.
By a simple ``Fubini slicing'' argument we may also derive the following interesting corollary from Theorem \ref{th:dimension-results}. We retain the notation of Proposition \ref{pr:dimaffq-subvarieties}.
\begin{corollary}[Stratified Marstand-type projection theorem for self-affine sets]\label{co:stratified}
Let $\A=(A_i)_{i \in \I} \in \GL_d(\R)^\I$ be such that $\max_{i,j \in \I \colon i \neq j} \threebar{A_i}+\threebar{A_j}<1$ for some norm $\threebar{\cdot}$ on $\R^d$. Let $\mathcal{W}_{t_m} \subset \ldots \subset \mathcal{W}_{t_1}$ be the distinct level sets of $Q \mapsto \dimaffQ \A$ given by Proposition \ref{pr:dimaffq-subvarieties}. 
Then for every $j=1,\ldots,m$, for Lebesgue a.e.\ $\mathbf{u}\in (\R^d)^\I$ we have $\dimh QX^{\mathbf{u}}=t_j$ for Lebesgue a.e.\ $Q \in \mathcal{W}_{t_j}$. 
\end{corollary}
\begin{remark}In general this conclusion cannot be strengthened to include every $Q \in \mathcal{W}_{t_j}\setminus \mathcal{W}_{t_j+1}$. To see this, consider an iterated function system on $\R^2$ of the form $T_ix=\frac{1}{3}{x}+u_i$ for $i \in \I=\{1,2,3\}$. As long as $u_1$, $u_2$, $u_3$ are not colinear, up to an affine co-ordinate change the attractor is the ``one-dimensional Sierpi\'nski triangle'' studied by Furstenberg, Kenyon and Hochman, which has a countable dense family of exceptional projections arising from exact overlaps \cite{Ho14,Ke97}. In particular for Lebesgue a.e. $\mathbf{u}\in (\R^2)^\I$ the dimension $\dimh QX^{\mathbf{u}}$ is nonconstant with respect to $Q \in \mathcal{W}_1\setminus \mathcal{W}_2$. \end{remark}

\begin{figure}
    \centering
    \includegraphics[width=8.9cm]{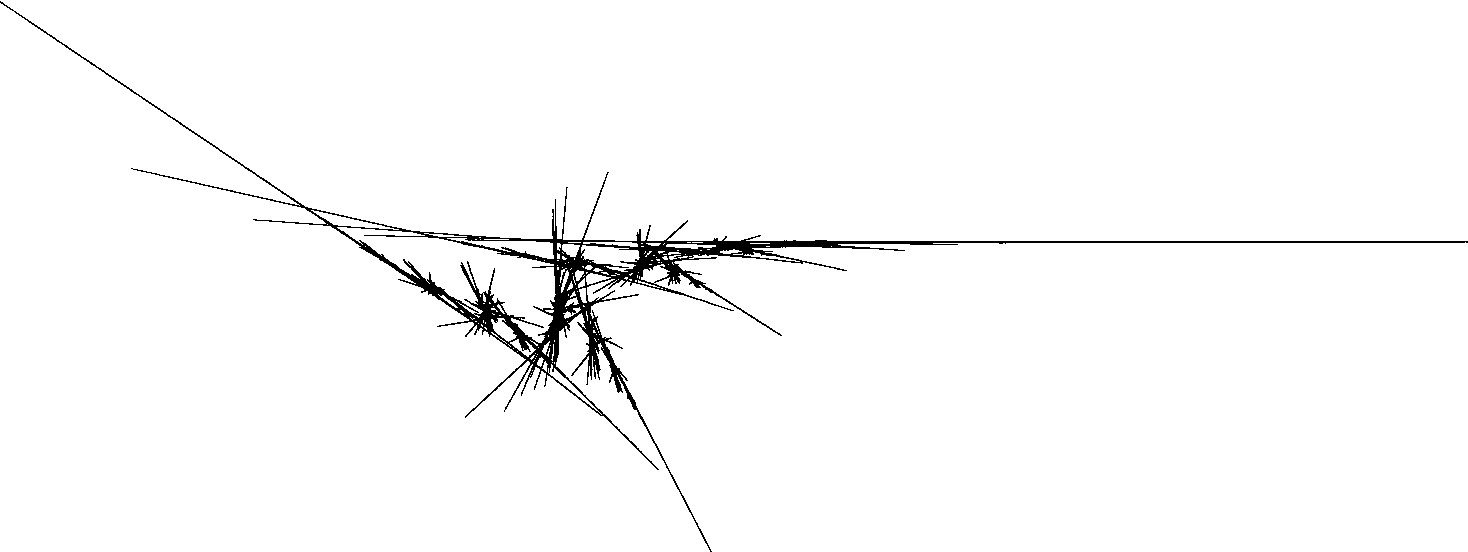}
    
    \bigskip
    
    \bigskip
    
    \includegraphics[width=8.9cm]{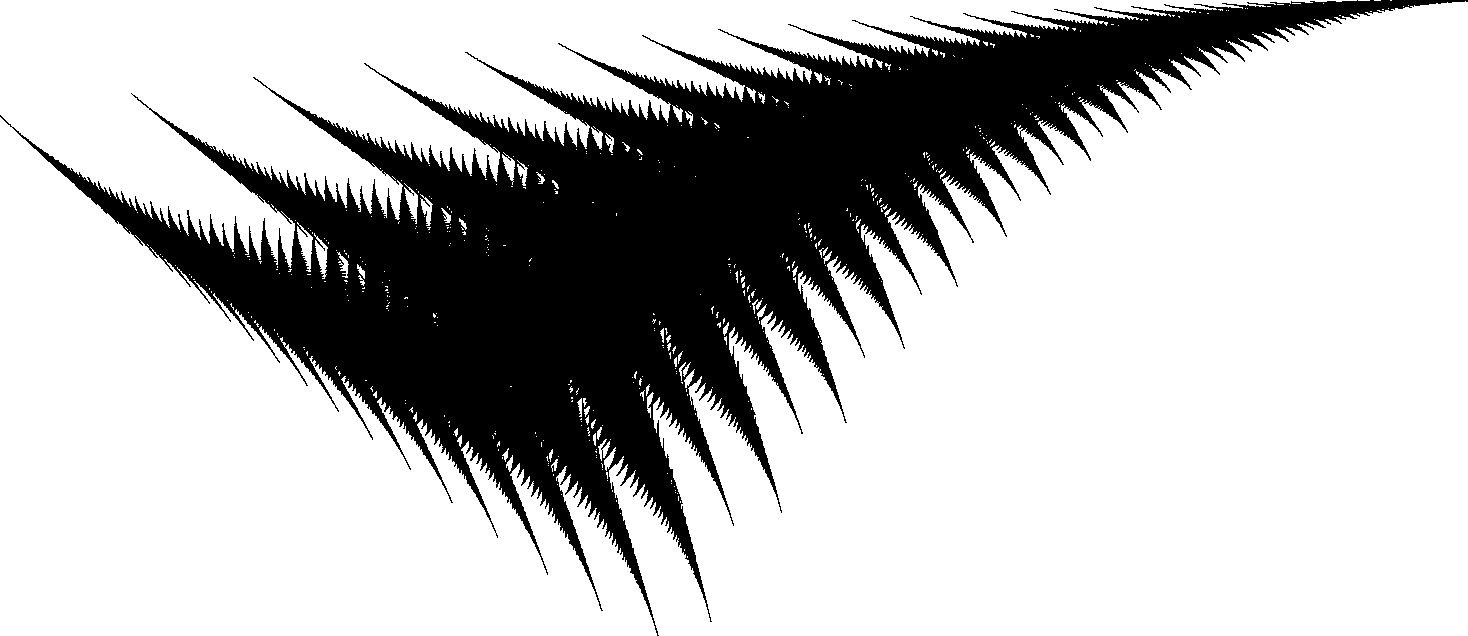}
     \captionsetup{singlelinecheck=off}
    \caption[.]{Two projections of the attractor of an iterated function system 
whose linearisation generates a Zariski-dense subsemigroup of $\R^* \SO(2,2)$. The two projections are onto two-dimensional isotropic subspaces of $\R^4$ with differing orientations.}
    \label{fig:enter-label}
\end{figure}

\subsection{Examples}\label{ss:examples} 

To round out this work we indicate some examples which demonstrate that the hierarchy of varieties identified in Theorem \ref{intro:stratified} is in general nontrivial. The following result establishes broad general criteria for the occurrence (and also the non-occurrence) of nontrivial varieties in Theorem \ref{intro:stratified}. For the statement, we recall that a semigroup \label{notation:Gamma} $\Gamma < \GL(V)$ is called \textit{irreducible} if it does not stabilise any proper nonzero subspace of $V$. It is called \textit{strongly irreducible} if it does not stabilise any finite union of  proper nonzero subspaces of $V$. We recall that $\A \in \GL(V)^\I$ is called $k$-dominated if 
\[\limsup_{n \to \infty} \frac{1}{n} \max_{|\iii|=n}\log \left(\frac{\sigma_{k+1}(A_\iii)}{\sigma_k(A_\iii)}\right)<0\]
for some (and then for every) Euclidean structure on $V$. Clearly $\A$ is $k$-dominated if and only if $(A_i^{\wedge k})_{i \in \I}$ is $1$-dominated.
\begin{theorem}\label{th:general-examples}
Let $\A \in \GL_d(\R)^\I$, let  $G\leq \GL_d(\R)$ denote the Zariski closure of the semigroup generated by $\A$ and let $1 \leq k< d$. Suppose that $G$ is real reductive.
\begin{enumerate}[(i)]
\item \label{it:split}
Suppose that $\wedge^k\R^d$ is a direct sum of two $G$-invariant subspaces each of which contains a nonzero pure $k$-wedge. If $\A$ is $k$-dominated then there exists $Q \in \End(\R^d)$ of rank $k$ such that $P_Q(\A,s)<P(\A,s)$ for every $s \in (k-1,k]$.
\item\label{it:no-split}
Suppose that $\wedge^k\R^d$ is \emph{not} a direct sum of two $G$-invariant subspaces each of which contains a nonzero pure $k$-wedge. Then $P_Q(\A,s)=P(\A,s)$ for every $Q\in \End(\R^d)$ of rank $k$ and for every $s \in [0,k]$. 
\item \label{it:bad-measure-projection}
Suppose that $G$ acts irreducibly but not strongly irreducibly on $\wedge^k \R^d$. If $\mu \in \mathcal{M}_\sigma(\I^\N)$ is ergodic and fully supported, and if $\A$ has distinct $k^{\mathrm{th}}$ and $(k+1)^{\mathrm{st}}$ Lyapunov exponents with respect to $\mu$, then there exists $Q \in \End(\R^d)$ of rank $k$ such that for every $s \in (k-1,k]$ the limit
\[\lim_{n \to \infty}\frac{1}{n}\log \varphi^s(QA_{\underline{\iii}|_n})\]
fails to be constant $\mu$-a.e.\
\end{enumerate}
\end{theorem}

We proceed with a few remarks to clarify this result.

\begin{remark}\label{re:varieties}

1. Examples where Theorem \ref{th:general-examples}\ref{it:split} applies include the case $G=\R^*\cdot \SO(k,k)<\GL_{2k}(\R)$ and the case in which $G$ is given by the tensor product representation of $\GL_{d_1}(\R)\times \GL_{d_2}(\R)$. A detailed exposition of these cases is given in the conference proceedings article \cite{Mo27}. In the former case, the proof of Theorem \ref{th:general-examples}\ref{it:split} may be specialised so as to construct examples in which every projection whose kernel has the form $\{(u,Ou) \colon u \in \R^k\}$ for some $O \in \SO(k)$ is an exceptional projection of the attractor. In the case of the tensor product representation of $\GL_{d_1}(\R)\times \GL_{d_2}(\R)$, by quite direct calculations one may construct examples in which every projection of the form $Q \otimes I$, where $1 \leq \rank Q<d_1$, is exceptional. In any case, if $\A$ is strongly irreducible then by Theorem \ref{th:pressure-detail}\ref{it:sublevel-algebraic} the existence of a single orthogonal projection $Q$ such that $\dimaffQ \A<\dimaff \A$ implies the existence of at least a one-dimensional variety of such projections.  These constructions yield examples of algebraic varieties embedded in the exceptional set in Marstrand's Theorem which have not been previously observed in the self-affine context.\\[3pt]
2. It follows from Theorem \ref{th:general-examples}\ref{it:no-split} that if $\A$ is $k$-irreducible (i.e.\ irreducible on $\wedge^{k}\R^d$) with $\dimaff \A \leq k$ then necessarily $\dimaffQ \A=\dimaff \A$ for every $Q \in \End(\R^d)$ whose rank is $k$. However, this outcome can also arise in certain cases where $\A$ is not $k$-irreducible: for example, if $\A$ generates a Zariski-dense subgroup of $\R^* \cdot \mathrm{Sp}(2m,\R)$ then Theorem \ref{th:general-examples}\ref{it:no-split} applies, since in this case $\wedge^k \R^{2m}$ splits into invariant subspaces exactly one of which contains a wedge.\\[3pt]
3. In the statement  \ref{it:bad-measure-projection} above, denote by $\Gamma$ the semigroup generated by $\A$ and $\Gamma^{-1}=\{\gamma^{-1} : \gamma \in \Gamma\}$ the corresponding inverse semigroup. One may show that the conclusion of \ref{it:bad-measure-projection} holds for every $Q \in \End(\R^d)$ whose kernel is represented via the Pl\"ucker embedding by an element of the \emph{limit set} of $\Gamma^{-1}$ in the projective space $\mathbf{P}(\wedge^{d-k} \R^d)$ \label{notation:projW}-- see \cite[\S 4.1]{BeQu16}. The limit set in this case is characterised as the unique minimal $\Gamma^{-1}$-invariant subset of $\mathbf{P}(\wedge^{d-k} \R^d)$.
\\[3pt]
4. We defer a more precise analysis of the pressure drop phenomenon expressed in Theorem \ref{th:general-examples} -- including, in particular, an estimation of the size of the resulting dimension drop --  to the upcoming work \cite{morris-sert.LDP}, where we will also give an alternative expression for $P_Q(\A,s)$ using the large deviations theory of random matrix products.
\end{remark}

Theorem \ref{th:general-examples}\ref{it:bad-measure-projection} implies the following general criterion for the existence of non-exact-dimensional projections of $\varphi^s$-equilibrium states:

\begin{corollary}\label{co:non-exact-dimensional}
Suppose that $\A \in \GL_d(\R)^\I$ is strongly $(k-1)$-irreducible, is $k$-proximal,  $k$-irreducible and not strongly $k$-irreducible, and is $(k+1)$-irreducible. Suppose also that $s:=\dimaff \A \in (k-1,k]$. Then there is a unique $\varphi^s$-equilibrium state $\mu$ for $\A$, and there exist rank-$k$ projections $Q$ such that the local Lyapunov dimension $\dimlyapQ (\A,\mu,\cdot)$ is not constant $\mu$-a.e.\
\end{corollary}

Here $\A$ is called $k$-proximal if, denoting by $\Gamma$ the semigroup generated by $\A$ in $\GL(\wedge^k\R^d)$, $\overline{\R \Gamma}$ contains a rank-one element in $\End(\wedge^k \R^d)$. In the presence of $k$-irreducibility this is equivalent to the existence of an element of $\overline{\R \Gamma}$ with a simple leading eigenvalue.

\begin{proof}
By \cite[Theorem 1.2(i)]{Mo18b} there exists a unique $\varphi^s$-equilibrium state $\mu$ for $\A$ and its $k^{\mathrm{th}}$  and $(k+1)^{\mathrm{st}}$ Lyapunov exponents are distinct, and by the main result of \cite{BoMo18} this measure is fully supported.  Since $\A$ is $k$-irreducible it generates a reductive linear algebraic group $G \leq \GL_d(\R)$ which acts irreducibly but not strongly irreducibly on $\wedge^k\R^d$. By Theorem \ref{th:general-examples}\ref{it:bad-measure-projection} there exists $Q\in \End(\R^d)$ of rank $k$ such that $\lim_{n \to \infty}\frac{1}{n}\log\varphi^s(QA_{\underline{\iii}|_n})$ is not constant a.e, and by Theorem \ref{th:general-examples}\ref{it:no-split} we have $P_Q(\A,s)=P(\A,s)$ as a consequence of $k$-irreducibility. Every $(\varphi^s,Q)$-equilibrium state for $\A$ is therefore also a $\varphi^s$-equilibrium state, so $\mu$ itself must be a $(\varphi^s,Q)$-equilibrium state. The essential supremum of $\dimlyapQ(\A,\mu,\cdot)$ thus equals $s$, but by Theorem \ref{th:general-examples}\ref{it:bad-measure-projection} the essential infimum cannot also equal $s$.
\end{proof}

By choosing $\A \in \GL_d(\R)^\I$ to generate a Zariski-dense subgroup of the group of generalised permutation matrices (i.e. the group of $d\times d$ real matrices having exactly one nonzero entry in each row and column), combining Corollary \ref{co:non-exact-dimensional} and Theorem \ref{th:dimension-results} in the case $k=1$ yields Theorem \ref{intro:not-exact-dimensional}. Here the projections in question are precisely the rank-one orthogonal projections onto the co-ordinate axes. In the case $d=2$ this corresponds to the examples considered in \cite{AlKaPrSiTr25,FeXi25}. In the case $d=2k \geq 4$, taking $\A$ to generate a Zariski-dense subgroup of $\R^*\cdot \O(k,k)$ yields strongly irreducible examples.  A careful reading of the proof of Theorem \ref{th:general-examples}\ref{it:bad-measure-projection} further demonstrates that  in this case the set of projections $Q$ such that $Q_*\Pi_*\mu$ fails to be exact-dimensional includes the set of all orthogonal projections onto $k$-dimensional isotropic subspaces of the group $\O(k,k)$. 

\subsection{Structure of the remainder of the article} In \S\ref{se:technical-core} below we state and prove a more general and more overtly algebraic formulation of Theorem \ref{th:pressure-detail}, namely Theorem \ref{th:main-tech}. Theorem \ref{th:pressure-detail} is then derived from this result in \S\ref{se:derivation}. Sections \ref{se:dimension} and \ref{se:general-examples} provide the proofs of Theorems \ref{th:dimension-results} and \ref{th:general-examples} respectively, each of which is relatively short and is mostly self-contained.

As noted earlier in this section, Theorems \ref{intro:stratified} and \ref{intro:not-exact-dimensional} follow from straightforward combinations of the results already stated above. Theorem \ref{intro:sumset} on the other hand follows from the combination  of Theorem \ref{th:dimension-results} with some additional explicit calculations which we present at the end of \S\ref{se:general-examples}.

\section{An abstract formulation of Theorem \ref{th:pressure-detail}}\label{se:technical-core}
In this section we state and prove Theorem \ref{th:main-tech}, a more general and more abstract formulation of Theorem \ref{th:pressure-detail} in which the map $\iii \mapsto \varphi^s(QA_\iii)$ is replaced with a general potential given by a product of operator norms of finitely many representations of a semigroup. Theorem \ref{th:pressure-detail} will be deduced from this result in the following Section \ref{se:derivation}. We begin by setting notation and stating some preliminaries.

\subsection{Preliminaries from subadditive thermodynamic formalism}

We retain the notation from the previous section, in particular, $\I$ denotes a nonempty finite set and $\Gamma_\I$ the set of all finite words over $\iii$.
We define a \emph{potential}\label{notation:Psi} to be an arbitrary function $\Gamma_\I \to (0,\infty)$. We call a potential  \emph{submultiplicative} if it satisfies $\Psi(\iii\jjj)\leq \Psi(\iii)\Psi(\jjj)$ for every $\iii,\jjj \in \Gamma_\I$. We further call a potential $\Psi \colon \Gamma_\I \to(0,\infty)$ \emph{quasi-multiplicative} if it is submultiplicative and if additionally there exist $\kappa>0$ and a finite set $F\subset \Gamma_\I$ such that 
\[\max_{\kkk \in F} \Psi(\iii \kkk \jjj) \geq \kappa \Psi(\iii)\Psi(\jjj)\]
for all $\iii,\jjj \in \Gamma_\I$. If the set $F$ can be chosen to consist of words of constant length then we will say that $\Psi$ is \emph{strongly quasi-multiplicative}. The ergodic-theoretic properties of these three classes of potentials are summarised in the following result, which encompasses the work of several previous authors:
\begin{theorem}[\cite{CaFeHu08,Fe11,Pi20}]\label{th:properties-of-potentials}
Let $\I$ be a nonempty finite set and $\Psi \colon \I^\N \to (0,\infty)$ be a potential. Then:
\begin{enumerate}[(i)]
    \item\label{it:variational-principle}
    If $\Psi$ is submultiplicative then the pressure\label{notation:P(Psi)}
    \[P(\Psi):=\lim_{n \to \infty} \frac{1}{n}\log \sum_{|\iii|=n}\Psi(\iii)\in [-\infty,\infty)\]
    exists and satisfies
    \[P(\Psi)=\sup_{\mu \in \mathcal{M}_\sigma(\I^\N)} \left[h(\mu) + \lim_{n \to \infty}\frac{1}{n}\int_{\I^\N} \log\Psi(\underline{\iii}|_n)d\mu(\underline{\iii})\right],\]
    and this supremum is attained by at least one ergodic measure. We call invariant measures which attain this supremum \emph{equilibrium states} of $\Psi$.
    \item\label{it:quasi-multiplicative}
    If $\Psi$ is quasi-multiplicative then it has a unique equilibrium state $\mu \in \mathcal{M}_\sigma(\I^\N)$. The measure $\mu$ is ergodic with respect to $\sigma$, and there exists $C\geq 1$ such that  
    \[C^{-1}\leq \frac{\mu([\iii])}{e^{-|\iii|P(\Psi)} \Psi(\iii)} \leq C\]
    for every $\iii \in \Gamma_\I$.
    \item \label{it:strong-quasi}
    If $\Psi$ is strongly quasi-multiplicative then its unique equilibrium state $\mu$  is additionally \emph{$\psi$-mixing}:
\[\lim_{n \to \infty}\sup_{\iii,\jjj\in \Gamma_\I} \left|\frac{\mu([\iii]\cap \sigma^{-|\iii|-n} [\jjj])}{\mu([\iii])\mu([\jjj])}-1\right| =0.\]
\end{enumerate}
\end{theorem}
\begin{proof}
Part \ref{it:variational-principle} follows from the main theorem of \cite{CaFeHu08}, and \ref{it:quasi-multiplicative} similarly follows from \cite[Theorem 5.5]{Fe11}. To derive the $\psi$-mixing property of equilibrium states of strongly quasi-multiplicative potentials, we first note that the $\psi$-mixing property defined above holds if and only if the same property holds for the natural lifting of the measure $\mu$ to a shift-invariant measure on the invertible, two-sided shift $\I^\Z$. Following the proof of \cite[Theorem 1.3]{Pi20}, the $\psi$-mixing property of the latter may be derived by applying Theorem 1.1 of \cite{Br83} and the remarks following it, or can be obtained more directly from \cite[Theorem 4.1(2)]{Br05}.
\end{proof}

We additionally say that two potentials $\Psi_1,\Psi_2 \colon \Gamma_\I \to (0,\infty)$ are \emph{equivalent} if their ratio $\Psi_1/\Psi_2$ is uniformly bounded away from zero and infinity. Clearly if two potentials are equivalent then they have identical sets of equilibrium states, and if both potentials are quasi-multiplicative then the converse also holds as a consequence of Theorem \ref{th:properties-of-potentials}\ref{it:quasi-multiplicative}.

\emph{Remark.} If a measure $\mu \in \mathcal{M}_\sigma(\I^\N)$ is $\psi$-mixing then its natural extension is measurably isomorphic to a Bernoulli process by a theorem of N. Friedman and D.S. Ornstein \cite{FrOr70}. The same conclusion also holds if $\mu$ is quasi-multiplicative and totally ergodic by a recent result of B. Call and K. Park \cite{CaPa23}, but to the best of our knowledge this combination of conditions is distinct from both $\psi$-mixing and strong quasi-multiplicativity in general.

\subsubsection*{Some notions and notation from the linear algebraic setup}
The potentials we will consider in the sequel will be given by products of operator norms of an elements of a semigroup in finitely many representations. To express these, we need to fix some linear algebraic terminology and notation before proceeding. All vector spaces are understood to be finite dimensional and over $\R$ (although we note that in this section, all results are valid if the base field is an arbitrary local field). A \emph{representation} $\rho$ \label{notation:rho} of a semigroup $\Gamma$ will simply mean a semigroup homomorphism $\Gamma \to \GL(W)$, where $W$ is a (finite dimensional real) vector space. A representation from a linear algebraic group $G$ to another linear algebraic group $H$ will always be understood to mean a homomorphism which is given by a polynomial map with respect to the respective algebraic structures. 
A representation $\rho: \Gamma \to \GL(W)$ is called \emph{irreducible} if the only linear subspaces of $W$ which are preserved by every element of the image of $\rho$ are $\{0\}$ and $W$. If $\rho \colon \Gamma \to G$ is a representation from a semigroup to a linear algebraic group then we will find it convenient to define $\Gamma^o:=\{\gamma \in \Gamma \colon \rho(\gamma)\in G^o\}$,\label{notation:Gammao} which is easily seen to be a subsemigroup of $\Gamma$. An element of $\End(W)$ is called \emph{proximal} if it has a unique eigenvalue of maximum modulus and if additionally that eigenvalue is simple.

If $G$ is an arbitrary group, its associated \emph{group ring} $\R[G]$ over $\R$ is the set of all formal linear combinations $\sum_{g \in G} f(g)\cdot g$ with $f \colon G \to \R$ finitely supported, equipped with coordinate-wise addition and with the natural distributive multiplication induced by the group operation, i.e.
$(\sum_{g \in G}f_1(g)\cdot g)(\sum_{g \in G}f_2(g)\cdot g):=\sum_{g \in G} (\sum_{h_1h_2=g}f_1(h_1)f_2(h_2))\cdot g.$
If $G$ acts on a vector space $W$ via a representation $\rho \colon G \to \GL(W)$ then $W$ acquires the structure of a module over $\R[G]$ via the scalar multiplication operation $(\sum_{g \in G}f(g)\cdot g , w)\mapsto \sum_{g \in G}f(g)\rho(g)w$. 
The submodules $V$ of the $\R[G]$-module $W$ are then precisely the linear subspaces of $W$ which are preserved by every $\rho(g)\in \GL(W)$, and likewise the quotient  $W/V$ of two $\R[G]$-modules is precisely the vector space $W/V$ equipped with its quotient $\R[G]$-module structure. We write $V \leq W$\label{notation:submodule} to mean that $V$ is a submodule of $W$. Two $\R[G]$-modules are isomorphic if and only if they are isomorphic as vector spaces via a linear transformation which commutes with the action of every $\rho(g)$. A module is called \emph{simple} if it has no proper submodules except the zero module. Via the Jordan-H\"older theorem the $\R[G]$-module induced by a finite-dimensional representation $G \to \GL(W)$ may be  decomposed uniquely (up to isomorphism) into a finite composition series.

If $W$ is a vector space equipped with a norm $\|\cdot\|_W$, \label{notation:Wnorm} $\rho\colon G\to \GL(W)$ is a representation and $V\leq W$ an $\R[G]$-submodule of $W$, then to simplify notation we will write $\|\rho(g)\|_W$ for the (operator) norm of $\rho(G)$ considered as an element of $\GL(W)$, $\|\rho(g)\|_V$ for the norm of $\rho(g)|_V \in \GL(V)$, $\|\cdot\|_{W/V}$ for the induced norm on $W/V$ and hence $\|\rho(g)\|_{W/V}$ for the norm of the linear transformation of $W/V$ induced by $\rho(g)$. Finally, where $\rho(g)$ induces a linear map from a normed vector space $V_1$ to another one $V_2$, we write $\|\rho(g)\|_{V_1 \to V_2}$ \label{notation:W12} for the norm of this induced linear map. 
The notation $\mathbf{S}_V$ will denote the unit sphere of the normed vector space $V$.  \label{notation:sphere}

\subsection{Statement of the abstract theorem}
We may now state the general result from which Theorem \ref{th:pressure-detail} will be derived.
\begin{theorem}\label{th:main-tech}
Let $\I$ be a finite nonempty set, let $W$ be a finite-dimensional real vector space, for every $i \in \I$ let $g_i \in \GL(W)$ and for every $\iii = i_1\cdots i_n \in \Gamma_\I$ define $g_\iii:=g_{i_1}g_{i_2}\cdots g_{i_n} \in G$. Let $G \leq \GL(W)$ denote the linear algebraic group which is the Zariski closure of $\{g_\iii \colon \iii \in \Gamma_\I\}$,  let $r \geq 1$ and for every $j=1,\ldots,r$ let $V_j$ be a finite-dimensional real vector space and $\rho_j \colon G \to \GL(V_j)$ a representation. For each $j=1,\ldots,r$ let $n_j$ denote the number of distinct isomorphism classes of composition factors of $V_j$ as an $\R[G^o]$-module.

Let $\mathfrak{U}$ \label{notation:frakU} denote the set of all tuples $\mathbf{U}=(U_j/U_j')_{j=1}^r$ \label{notation:Ujs} \label{notation:bfU} such that for every $j=1,\ldots,r$, $U_j$ and $U_j'$ are $\R[G^o]$-submodules of $V_j$ satisfying $0 \leq U_j' < U_j \leq V_j$. Let $\mathfrak{U}_S$  \label{notation:frakUS} denote the set of all $\mathbf{U}=(U_j/U_j')_{j=1}^r \in \mathfrak{U}$ having the additional property that every $U_j/U_j'$ is a simple $\R[G^o]$-module. For every $\mathbf{U} \in \mathfrak{U}$ and $\mathbf{b}=(\beta_j)_{j=1}^r \in [0,\infty)^r$ \label{notation:bfb} 
define a potential $\Psi_{\mathbf{U}, \mathbf{b}} \colon \Gamma_\I \to (0,\infty)$ by

\[\Psi_{\mathbf{U}, \mathbf{b}}(\iii):=\max_{ hG\in G/G^o} \left(\prod_{j=1}^r \left\|\rho_j(g_\iii)\right\|_{\frac{\rho_j(h)U_j}{\rho_j(h)U_j'} \to \frac{\rho_j(g_\iii 
h)U_j}{\rho_j(g_\iii h)U_j'}}^{\beta_j}\right).\]\label{notation:PsiUb}
Then the following properties hold:
\begin{enumerate}[(a)]
    \item\label{it:potential-basics}
Let $\mathbf{b}\in [0,\infty)^r$. If $\mathbf{U}\in \mathfrak{U}$ then the potential $\Psi_{\mathbf{U}, \mathbf{b}}$ is submultiplicative; if $\mathbf{U} \in \mathfrak{U}_S$, then $\Psi_{\mathbf{U}, \mathbf{b}}$ is quasi-multiplicative; and if $\mathbf{U} \in \mathfrak{U}_S$ and additionally $\bigcup_{|\iii|=m} g_\iii G^o=G$ for some integer $m \geq 1$, then $\Psi_{\mathbf{U}, \mathbf{b}}$ is strongly quasi-multiplicative.
If $\mathbf{U}=(U_j/U_j')_{j=1}^r$ and  $\hat{\mathbf{U}}=(\hat{U}_j/ \hat{U}_j')_{j=1}^r$ are elements of $\mathfrak{U}$ having the property that the $\R[G^o]$-modules $U_j/U_j'$ and $\hat{U}_j/\hat{U}_j'$ are isomorphic for every $j=1,\ldots,r$ such that $\beta_j \neq 0$, then $\Psi_{\mathbf{U}, \mathbf{b}}$ and $\Psi_{\hat{\mathbf{U}}, \mathbf{b}}$ are equivalent potentials.
\item \label{it:max-over-composition-factors}
Define a partial order relation $\preceq$ on $\mathfrak{U}$ as follows: $\mathbf{U}=(U_j/U_j')_{j=1}^r$ and  $\hat{\mathbf{U}}=(\hat{U}_j/\hat{U}_j')_{j=1}^r$ satisfy $\mathbf{U}\preceq \hat{\mathbf{U}}$ if and only if  $\hat{U}_j' \leq U_j' < U_j \leq \hat{U}_j$ for every $j=1,\ldots,r$. Then for every $\mathbf{U} \in \mathfrak{U}$ and $\mathbf{b} \in [0,\infty)^r$,
\[P(\Psi_{\mathbf{U}, \mathbf{b}}) = \max\left\{ 
 P\left(\Psi_{\hat{\mathbf{U}}, \mathbf{b}}\right)\colon \hat{\mathbf{U}} \in \mathfrak{U}_S\text{ and }\hat{\mathbf{U}}\preceq \mathbf{U}\right\},\]
and an ergodic measure $\mu \in \mathcal{M}_\sigma(\I^\N)$ is an equilibrium state of $\Psi_{\mathbf{U},\mathbf{b}}$ if and only if it is the unique equilibrium state of some potential $\Psi_{\hat{\mathbf{U}},\mathbf{b}}$ which attains the above maximum.

\item \label{it:psi-mixing-tech}
Let $\mathbf{U}\in \mathfrak{U}_S$ and $\mathbf{b}=(\beta_j)_{j=1}^r \in [0,\infty)^r$ and let $\mu$ be the unique equilibrium state of the quasi-multiplicative potential $\Psi_{\mathbf{U},\mathbf{b}}$. Then there exist an integer $q\geq1$ and a Borel probability measure $\nu$ on $\I^\N$ such that $\mu=\frac{1}{q}\sum_{k=0}^{q-1}(\sigma^k)_*\nu$ and such that $\nu$ is $\psi$-mixing with respect to $\sigma^q$:
\begin{equation*}
\lim_{n \to \infty} \sup_{\substack{\iii,\jjj \in \Gamma_\I\\ q\text{ divides }|\iii|, |\jjj|}} \left|\frac{\nu([\iii]\cap \sigma^{-|\iii|-nq}[\jjj])}{\nu([\iii])\nu([\jjj])}-1\right|=0.\end{equation*}
The integer $q$ is not greater than $\prod_{1 \leq j \leq r \colon \beta_j \neq 0}n_j$ and is a factor of $[G \colon G^o]$. If $\bigcup_{|\iii|=m} g_\iii G^o=G$ for some integer $m \geq 1$ then we may take $q:=1$. 

\item \label{it:q-pressure}
For every $j=1,\ldots,r$ let $Q_j \in \End(V_j)$ be nonzero, let $Z_j$ be the unique maximal $\R[G^o]$-submodule of $V_j$ which is a subset of $\ker Q_j$, and let $\mathbf{Z}=(V_j/Z_j)_{j=1}^r \in \mathfrak{U}$. Then there exist $C,\kappa>0$ such that for every $\mathbf{b}=(\beta_j)_{j=1}^r \in [0,\infty)^r$,
\begin{equation}\label{eq:q-potential-ratio}
C^{-\sum_{j=1}^r \beta_j}\kappa \sum_{|\iii|=n} \Psi_{\mathbf{Z}, \mathbf{b}}(\iii) \leq \sum_{|\iii|=n} \prod_{j=1}^r \|Q_j \rho_j(g_\iii)\|^{\beta_j} \leq C^{\sum_{j=1}^r \beta_j} \sum_{|\iii|=n} \Psi_{\mathbf{Z}, \mathbf{b}}(\iii)\end{equation}
for all $n \geq 1$, and hence in particular
\begin{equation}\label{eq:general-limit}\lim_{n \to \infty} \frac{1}{n}\log \sum_{|\iii|=n} \prod_{j=1}^r \|Q_j \rho_j(g_\iii)\|^{\beta_j} = P(\Psi_{\mathbf{Z}, \mathbf{b}}). 
\end{equation}
Furthermore, if $\mu \in \mathcal{M}_\sigma(\I^\N)$
is ergodic then the limit 
\[\lim_{n \to \infty} \frac{1}{n}\log \left(\frac{\prod_{j=1}^r \|Q_j \rho_j(g_{\underline{\iii}|_n})\|^{\beta_j}}{\mu([\underline{\iii}|_n])}\right)\]
exists $\mu$-a.e and satisfies 
\[{\ess\sup}_{\mu,\underline{\iii}} \lim_{n \to \infty} \frac{1}{n}\log \left(\frac{\prod_{j=1}^r \|Q_j \rho_j(g_{\underline{\iii}|_n})\|^{\beta_j}}{\mu([\underline{\iii}|_n])}\right)\leq P(\Psi_{\mathbf{Z},\mathbf{b}})\]
with equality if and only if $\mu$ is an equilibrium state of $\Psi_{\mathbf{Z},\mathbf{b}}$.

\item\label{it:exist-subvarieties}
For every  $\mathbf{b} \in [0,\infty)^r$ and $t \in \R$ the sublevel set
\[\mathcal{V}_{\mathbf{b},t}:=\left\{(Q_j)_{j=1}^r \in \prod_{j=1}^r \End(V_j)\colon \lim_{n \to \infty} \frac{1}{n}\log \sum_{|\iii|=n} \prod_{j=1}^r \|Q_j \rho_j(g_\iii)\|^{\beta_j}\leq t\right\}\]
is a subvariety of the affine variety $\prod_{j=1}^r \End(V_j)$. Furthermore
\[|\{\mathcal{V}_{\mathbf{b},t} \colon \mathbf{b} \in [0,\infty)^r\text{ and } t \in \R\}|\leq \prod_{j=1}^r \left(1+2^{n_j}\right),\]
and for every $\mathbf{b} \in [0,\infty)^r$
\[|\{\mathcal{V}_{\mathbf{b},t} \colon t \in \R\}|\leq 1+ \prod_{\substack{1 \leq j \leq r \\ \beta_j \neq 0}} n_j .\]
\end{enumerate}
\end{theorem}
\emph{Historical remarks.} Clauses \ref{it:q-pressure} and \ref{it:exist-subvarieties} of Theorem \ref{th:main-tech} are wholly novel and lack any antecedents in the literature which we are able to discern. 

\ref{it:potential-basics}--\ref{it:max-over-composition-factors}: The first considerations of quasi-multiplicativity in fractal geometry trace back to the work of D.-J. Feng and K.-S. Lau in \cite{FeLa02}, which was gradually extended by Feng in \cite{Fe09} and further by Feng and K\"aenm\"aki in \cite{FeKa11}. The latter treated a particular case (corresponding to $r=1$) of \ref{it:max-over-composition-factors}. The study of the general case ($r \geq 1$ and arbitrary dimension) for $\varphi^s$-equilibrium states, required additional algebraic ideas and  was initiated in \cite{KaMo18} and completed in \cite{BoMo18}. Therefore the case $Q=\id$ above \ref{it:potential-basics} and \ref{it:max-over-composition-factors} recovers these previous results. Strongly quasi-multiplicative potentials were introduced by M. Piraino in \cite{Pi20} in the case $r=1$ and were studied for general $r$ in \cite{Mo21} in a situation corresponding to the particular case of Theorem \ref{th:main-tech}\ref{it:potential-basics} in which every $V_j$ is irreducible as an $\R[G]$-module.

\ref{it:psi-mixing-tech}: In the case $r=1$ only, the existence of equilibrium states which are ergodic but not totally ergodic was originally reported in \cite{Mo18a} and fully characterised in that case in the subsequent work \cite{Mo19b}. A result providing the same structural description as \ref{it:psi-mixing-tech} was presented in \cite{Mo21} for a more particular setting that recovers the case of invertible $Q_j$'s in our result. The results of \cite{Mo21} also do not relate failure of total ergodicity to any relationship between $G^o$ and $G$. The extension of this result to the more general class of potentials needed to prove Theorem \ref{th:pressure-detail} is thus the key achievement of Theorem \ref{th:main-tech}\ref{it:psi-mixing-tech} and also of the strong quasi-multiplicativity clause of \ref{it:potential-basics} on which the former depends.

\bigskip

We lastly note that while the results in this section are formulated over the field $\R$, both Theorem \ref{th:main-tech} in particular and the results of \S\ref{se:technical-core} more broadly go through entirely unchanged if $\R$ is consistently replaced with an arbitrary local field. 

\subsection{Proof of Theorem \ref{th:main-tech}\ref{it:potential-basics}}
Throughout the proof of Theorem \ref{th:main-tech} we let $\Gamma^o_\I$ denote the subsemigroup $\{\iii \in \Gamma_\I\colon g_\iii \in G^o\}\leq \Gamma_\I$. We begin with the arguments needed to demonstrate quasi-multiplicativity. At the core of it lies the following ``bridging lemma'' that was used in various forms by Benoist \cite{benoist.annals, benoist.prop.asymp} and several others (e.g.\ \cite{quint.divergence,Fe09,Wi02}), going back to Abels--Margulis--Soifer \cite{AbMaSo}.

\begin{proposition}\label{pr:long-bridge}
Let $H$ be a Zariski-connected linear algebraic group, $\Gamma<H$ a Zariski-dense subsemigroup of $H$ and for $j=1,\ldots,r$ let $\rho_j \colon H \to \GL(W_j)$ be an irreducible representation. Then there exist a constant $\kappa>0$ and a finite set $F \subset \Gamma$ such that for every $g,g'\in H$, there exists $f \in F$ with the property that
\[\inf_{n \in \N} \min_{1 \leq j \leq r}  \frac{\|\rho_j(g f^ng')\|}{\|\rho_j(g)\|\cdot \|\rho_j(f^n)\|\cdot \|\rho_j(g')\|} > \kappa.\]
\end{proposition}

We will include a brief proof of this proposition for completeness. The proof will require the following lemma which we borrow from \cite[Lemma 4.13 (c)]{BQ.CLT}. It will allow us to work, without loss of generality, with proximal representations. The use of this lemma can be traced back to Tits' \cite{tits.free} in the proof of the Tits alternative.

\begin{lemma}\label{le:prox-redux}
Let $H$ be a group and $\rho \colon H \to \GL(V)$ an irreducible representation. Then there exist a proximal irreducible representation $\hat \rho \colon H \to \GL(\hat{V})$, a real number $C>0$ and an integer $\ell \geq 1$ such that $C^{-1}\leq \|\rho(h)\|^\ell /\|\hat\rho(h)\| \leq C$ for all $h \in H$.  \qed 
\end{lemma}

\begin{proof}[Proof of Proposition \ref{pr:long-bridge}]

As a consequence of Lemma \ref{le:prox-redux} we may assume without loss of generality that every $\rho_j$ is a proximal representation.

If $q \in \End(W_j)$ is a proximal element, let $\pi_q \in \End(W_j)$ denote the linear projection whose image is the leading eigenspace of $q$ and whose kernel is the complementary invariant hyperplane. Let $\Gamma_{\mathrm{prox}}$ denote the set of all $h \in \Gamma$ with the property that $\rho_j(h)\in \End(W_j)$ is proximal for every $j=1,\ldots,r$. 
We claim that if $q_j, q_j' \in \End(W_j)$ are both nonzero for every $j=1,\ldots,r$, then there exists $h \in \Gamma_{\mathrm{prox}}$ such that  $q_j \pi_{\rho_j(h)} q_j' \neq 0$ for all $j=1,\ldots,r$. Indeed, given nonzero $q_1,\ldots,q_r$ and $q_1',\ldots,q_r'$, by \cite[Lemma 6.25]{BeQu16} there exists $h_0 \in \Gamma_{\mathrm{prox}}$ such that  $q_j \pi_{\rho_j(h_0)} \neq 0$ for every $j=1,\ldots,r$. Now, for each $j=1,\ldots,r$ the sets
\[\mathcal{U}_j:=\left\{ f \in H  \colon q_j \pi_{\rho_j(h_0)}\rho_j(f)q_j' \neq 0\right\},\qquad  \mathcal{V}_j:=\left\{f \in H \colon\tr \pi_{\rho_j(h_0)}\rho_j(f)\neq 0\right\}\]
are clearly both Zariski open. If $\mathcal{U}_j$ were empty then  $\Span \{\rho_j(f)q_j'w \colon f \in  H\text{ and }w \in W_j\}$ would be a nonzero invariant subspace of $W_j$ which is contained in the proper subspace $\ker q_j\pi_{\rho_j(h_0)}$ of $W_j$, contradicting the irreducibility of $\rho_j$, so every $\mathcal{U}_j$ must be nonempty. Clearly every $\mathcal{V}_j$ contains the identity and hence is nonempty also. Since $H$ is Zariski-connected its nonempty Zariski-open subsets are all Zariski dense, so the intersection $\bigcap_{j=1}^r \mathcal{U}_j\cap\mathcal{V}_j$ is a nonempty Zariski-open set and therefore contains an element $f_0 \in \Gamma$, say. For each $j=1,\ldots,r$ the element $\pi_{\rho_j(h_0)}\rho_j(f_0) \in \End(W_j)$ has rank one and nonzero trace, hence is proximal, so $\rho_j(h_0^n f_0)$ must also be proximal for every large enough $n$ since every accumulation point of the sequence $\|\rho_j(h_0^n)\|^{-1}\rho_j(h_0^n f_0) $ is a nonzero scalar multiple of the proximal linear map $\pi_{\rho_j(h_0)}\rho_j(f_0)$. Since moreover $\lim_{n \to \infty} \pi_{\rho_j(h_0^nf_0)} = \pi_{\pi_{\rho_j(h_0)} \rho_j(f_0)}$  for each $j=1,\ldots,r$, it follows that  if $n$ is large enough then $h_0^nf_0 \in \Gamma_{\mathrm{prox}}$ and additionally $q_j \pi_{\rho_j(h_0^nf_0)} q_j' \neq 0$ for every $j=1,\ldots,r$. The claim is proved.

Now, for every $h \in \Gamma_{\mathrm{prox}}$  the set
\[O_h:=\left\{(q_j,q_j')_{j=1}^r \in \prod_{j=1}^r \mathbf{S}_{\End(W_j)}^2  \colon q_j \pi_{\rho_j(h)}q_j' \neq 0\text{ for all }j=1,\ldots,r\right\}\]
is clearly open (in the norm topology) and by the preceding observation the union $\bigcup_{h \in \Gamma_{\mathrm{prox}}} O_h$ covers the compact set $\prod_{j=1}^r \mathbf{S}_{\End(W_j)}^2$. Passing to a finite subcover, we deduce the existence of a finite set $F_0\subset \Gamma_{\mathrm{prox}}$ such that 
\[
\max_{f \in F_0} \min_{1 \leq j \leq r}\|q_j \pi_{\rho_j(f)}q_j'\|>0\]
for all $(q_j, q_j')_{j=1}^r \in \prod_{j=1}^r \mathbf{S}_{\End(W_j)}^2$. Appealing to pointwise convergence as $n \to \infty$ and to equicontinuity, it follows that for some $\kappa>0$ and all large enough $n$
\[\inf_{g,g' \in H} \max_{f \in F_0} \min_{1 \leq j \leq r}\frac{\|\rho_j(g f^ng')\|}{\|\rho_j(g)\|\cdot \|\rho_j(f^n)\|\cdot \|\rho_j(g')\|}>\kappa\]
and the conclusion follows by taking $F:=\{f^m : f\in F_0\}$ with $m$ suitably large. 
\end{proof}

We may now prove the core statements of Theorem \ref{th:main-tech}\ref{it:potential-basics}. We first observe that every $\Psi_{\mathbf{U},\mathbf{b}}$ is submultiplicative: given $\mathbf{U}=(U_j/U_j')_{j=1}^r \in \mathfrak{U}$, $\mathbf{b}=(\beta_j)_{j=1}^r \in [0,\infty)^r$ and $\iii,\jjj \in \Gamma_\I$, choosing $h \in G$ such that
\[\Psi_{\mathbf{U},\mathbf{b}}(\iii\jjj) = \prod_{j=1}^r \left\|\rho_j(g_{\iii\jjj})\right\|^{\beta_j}_{\frac{\rho_j(h)U_j}{\rho_j(h)U_j'} \to \frac{\rho_j(g_{\iii\jjj}h)U_j}{\rho_j(g_{\iii\jjj} h)U_j'}}\]
yields
\[\Psi_{\mathbf{U},\mathbf{b}}(\iii\jjj) \leq\left(\prod_{j=1}^r\left\|\rho_j(g_{\iii})\right\|^{\beta_j}_{\frac{\rho_j(g_\jjj h)U_j}{\rho_j(g_\jjj h)U_j'} \to \frac{\rho_j(g_{\iii\jjj}h)U_j}{\rho_j(g_{\iii\jjj} h)U_j'}} \right)\left(\prod_{j=1}^r\left\|\rho_j(g_{\jjj})\right\|^{\beta_j}_{\frac{\rho_j(h)U_j}{\rho_j(h)U_j'} \to \frac{\rho_j(g_{\jjj}h)U_j}{\rho_j(g_{\jjj} h)U_j'}} \right) \]
and the latter is clearly bounded above by $\Psi_{\mathbf{U},\mathbf{b}}(\iii) \Psi_{\mathbf{U},\mathbf{b}}(\jjj)$.

We may now approach the quasi-multiplicativity and strong quasi-multiplicativity clauses of Theorem \ref{th:main-tech}\ref{it:potential-basics}. Fix $\mathbf{U}=(U_j/U_j')_{j=1}^r \in \mathfrak{U}_S$ and $\mathbf{b}=(\beta_j)_{j=1}^r\in [0,\infty)^r$. We claim that there exist a finite set $F_1\subset \Gamma^o_\I$ and real number $\tau_1>0$ such that for every $\iii,\jjj \in \Gamma^o_\I$
\begin{equation}\label{eq:new-tau-claim}
\max_{\jjj \in F_1} \prod_{j=1}^r \|\rho_j(g_{\iii_1} g_\jjj g_{\iii_2})\|^{\beta_j}_{U_j/U_j'} \geq \tau_1\left( \prod_{j=1}^r\|\rho_j(g_{\iii_1})\|^{\beta_j}_{U_j/U_j'}\right)\left( \prod_{j=1}^r\|\rho_j(g_{\iii_2})\|^{\beta_j}_{U_j/U_j'}\right)\end{equation}
and such that additionally every word in $F_1$ has the same length. To see this we apply Proposition \ref{pr:long-bridge} with $H:=G^o$, $\Gamma:=\{g_\iii \in G^o \colon \iii \in \Gamma^o_\I\}$, $W_j:=U_j/U_j'$ and with each representation $G^o \to \GL(U_j/U_j')$ being the representation induced by the corresponding representation $\rho_j \colon G \to \GL(V_j)$;  these representations are irreducible since $\mathbf{U}\in\mathfrak{U}_S$. Let $F\subset \Gamma$ be the finite set and $\kappa>0$ the constant given by that proposition.
Fix $\jjj_1,\ldots,\jjj_p \in \Gamma_\I^o$ such that $F=\{g_{\jjj_t} \colon 1 \leq t \leq p\}$ and let $n$ be a large integer which is divisible by $|\jjj_t|$ for every $t=1,\ldots,p$. Clearly we have
\begin{equation}
    \label{eq:pre-tau}
\max_{1 \leq t \leq p} \min_{1 \leq j \leq r} \frac{\left\|\rho_j\left(g_{\iii_1} g_{\jjj_t}^{n/|\jjj_t|} g_{\iii_2}\right)\right\|_{U_j/U_j'}}{\left\|\rho_j\left(g_{\iii_1}\right)\right\|_{U_j/U_j'}\left\|\rho_j\left(g_{\jjj_t}^{n/|\jjj_t|}\right)\right\|_{U_j/U_j'}\left\|\rho_j\left(g_{\iii_2}\right)\right\|_{U_j/U_j'}}>\kappa\end{equation}
for every $\iii_1,\iii_2 \in \Gamma_\I^o$ and for every such $n \in \N$, so we define
\[F_1:=\left\{\jjj_t^{n/|\jjj_t|} \colon 1 \leq t \leq p\right\}\subseteq \{\jjj \in \Gamma_\I^o \colon |\jjj|=n\}\]
and
\[\tau_1:= \prod_{j=1}^r\left( \kappa \cdot \min_{\jjj \in F_1} \|\rho_j(g_\jjj)\|_{U_j/U_j'}\right)^{\beta_j}\]
and observe that \eqref{eq:new-tau-claim} now follows from \eqref{eq:pre-tau} by simple rearrangements. 

We may now complete the proof. Fix a finite set $F_2 \subseteq \Gamma_\I$ such that $\bigcup_{\kkk \in F_2}  g_\kkk G^o   = G$. Given arbitrary $\iii_1, \iii_2 \in \Gamma_\I$ we may choose $\kkk_1, \kkk_2 \in F_2$ such that \[\Psi_{\mathbf{U},\mathbf{b}}(\iii_t)=\prod_{j=1}^r \left\|\rho_j(g_{\iii_t})\right\|^{\beta_j}_{\frac{\rho_j(g_{\kkk_t})U_j}{\rho_j(g_{\kkk_t})U_j'}\to \frac{\rho_j(g_{\iii_t}g_{\kkk_t})U_j}{\rho_j(g_{\iii_t}g_{\kkk_t)U_j'}}}\]
for $t=1,2$, and we may further choose $\kkk_1',\kkk_2' \in F_2$ such that $\kkk_t'\iii_t\kkk_t\in \Gamma^o_\I$ for $t=1,2$. We observe that
\[\prod_{j=1}^r \left\|\rho_j(g_{\kkk_t'\iii_t\kkk_t})\right\|^{\beta_j}_{U_j/U_j'}\geq \tau_2 \Psi_{\mathbf{U},\mathbf{b}}(\iii_t)\]
for $t=1,2$, where \[\tau_2:=\left(\max_{\kkk \in F_2} \prod_{j=1}^r \left\|\rho_j(g_\kkk)^{-1}\right\|^{\beta_j}_{V_j}  \right)^{-2}.\] Applying \eqref{eq:new-tau-claim} yields
\begin{align*}\max_{\jjj \in F_1} \Psi_{\mathbf{U},\mathbf{b}}(\kkk_1'\iii_1 \kkk_1 \jjj\kkk_2'\iii_2 \kkk_2) &\geq \max_{\jjj \in F_1} \prod_{j=1}^r \left\|\rho_j(g_{\kkk_1'\iii_1 \kkk_1 \jjj\kkk_2'\iii_2 \kkk_2})\right\|^{\beta_j}_{U_j/U_j'} \\
&\geq \tau_1 \left(\prod_{j=1}^r \left\|\rho_j(g_{\kkk_1'\iii_1 \kkk_1})\right\|^{\beta_j}_{U_j/U_j'}\right)\left(\prod_{j=1}^r \left\|\rho_j(g_{\kkk_2'\iii_2 \kkk_2})\right\|^{\beta_j}_{U_j/U_j'}\right) \\
&\geq \tau_1\tau_2^2 \Psi_{\mathbf{U},\mathbf{b}}(\iii_1) \Psi_{\mathbf{U},\mathbf{b}}(\iii_2)\end{align*}
and it follows directly that
\[\max_{\jjj \in F_1} \max_{\kkk_1, \kkk_2' \in F_2} \Psi_{\mathbf{U},\mathbf{b}}(\iii_1 \kkk_1 \jjj \kkk_2' \iii_2) \geq \tau_1\tau_2^2\left(\max_{\kkk \in F_2} \Psi_{\mathbf{U},\mathbf{b}}(\kkk)\right)^{-2} \Psi_{\mathbf{U},\mathbf{b}}(\iii_1)\Psi_{\mathbf{U},\mathbf{b}}(\iii_2).\]
Since $\iii_1, \iii_2 \in \Gamma_\I$ were arbitrary, we have proved quasi-multiplicativity. If the additional condition $\bigcup_{|\kkk|=m} g_\kkk G^o=G$ is satisfied for some integer $m \geq 1$ then clearly we may choose $F_2\subseteq\{\kkk \in \Gamma_\I \colon |\kkk|=m\}$ in the preceding argument, in which case every element of $F_2F_1F_2$ is a word of the same length $n+2m$ and the preceding arguments demonstrate strong quasi-multiplicativity of $\Psi_{\mathbf{U},\mathbf{b}}$. The proof of the final clause of Theorem \ref{th:main-tech}\ref{it:potential-basics}, concerning equivalence of potentials, is elementary.

\subsection{Proof of Theorem \ref{th:main-tech}\ref{it:max-over-composition-factors}}
Throughout this section we fix $\mathbf{U}=(U_j/U_j')_{j=1}^r\in \mathfrak{U}$ and $\mathbf{b}=(\beta_j)_{j=1}^r \in [0,\infty)^r$. For each $j=1,\ldots,r$ we also fix a filtration of $\R[G^o]$-modules,
\begin{equation}\label{eq:series}\{0\}\leq U_j'=U_j^0<U_j^1<\cdots <U_j^{m_j}=U_j,\end{equation}
with the property that each of the quotients $U^\ell_j/U^{\ell-1}_j$ is a simple $\R[G^o]$-module. 
Let $\mathcal{L}$ denote the set of all tuples $\mathfrak{l}=(\ell_j)_{j=1}^r\in \N^r$ which satisfy $1 \leq \ell_j \leq m_j$ for every $j=1,\ldots,r$, and for each $\mathfrak{l}=(\ell_j)_{j=1}^r\in\mathcal{L}$ define $\mathbf{U}_{\mathfrak{l}}:=(U^{\ell_j}_j/U^{\ell_j-1}_j)_{j=1}^r \in \mathfrak{U}_S$. Evidently $\mathbf{U}_{\mathfrak{l}} \preceq \mathbf{U}$ for every $\mathfrak{l}\in\mathcal{L}$. Given any $\hat{\mathbf{U}}=(\hat{U}_j/\hat{U}_j')_{j=1}^r \in \mathfrak{U}_S$ satisfying $\hat{\mathbf{U}}\preceq \mathbf{U}$, it follows from the Jordan-H\"older theorem 
that each  $\hat{U}_j/\hat{U}'_j$ is isomorphic as an $\R[G^o]$-module to some quotient module $U_j^{\ell_j}/U_j^{\ell_j-1}$ arising from \eqref{eq:series}. Defining $\mathfrak{l}:=(\ell_j)_{j=1}^r \in \mathcal{L}$ it follows directly that the potentials $\Psi_{\hat{\mathbf{U}},\mathbf{b}}$ and $\Psi_{\mathbf{U}_{\mathfrak{l}},\mathbf{b}}$ are equivalent. This demonstrates that
\[\max_{\mathfrak{l}\in\mathcal{L}} P\left(\Psi_{\mathbf{U}_{\mathfrak{l}},\mathbf{b}}\right) = \max_{\substack{\hat{\mathbf{U}}\in \mathfrak{U}_S \\ \hat{\mathbf{U}}\preceq \mathbf{U}}} P\left(\Psi_{\hat{\mathbf{U}},\mathbf{b}}\right)\]
and so to prove Theorem \ref{th:main-tech}\ref{it:max-over-composition-factors} it is sufficient for us to show that 
\begin{equation}
    \label{eq:max-L}
\max_{\mathfrak{l}\in\mathcal{L}} P\left(\Psi_{\mathbf{U}_{\mathfrak{l}},\mathbf{b}}\right) = P\left(\Psi_{\mathbf{U},\mathbf{b}}\right)\end{equation}
and that each  ergodic measure $\mu \in \mathcal{M}_\sigma(\I^\N)$ is an equilibrium state of $\Psi_{\mathbf{U},\mathbf{b}}$ if and only if it is an equilibrium state of a potential $\Psi_{\mathbf{U}_{\mathfrak{l}},\mathbf{b}}$ whose pressure equals that of $P(\Psi_{\mathbf{U},\mathbf{b}})$. 
In view of Theorem \ref{th:properties-of-potentials}\ref{it:variational-principle} and of the inequality $\max_{\mathfrak{l} \in\mathfrak{L}} \Psi_{\mathbf{U}_{\mathfrak{l}},\mathbf{b}} \leq \Psi_{\mathbf{U},\mathbf{b}}$ which follows directly from the definition of the tuples $\mathbf{U}_{\mathfrak{l}}$,  both statements will hold if we can show that for every ergodic measure $\mu \in \mathcal{M}_\sigma(\I^\N)$,
\[\lim_{n \to \infty}\frac{1}{n}\int_{\I^\N}\log \Psi_{\mathbf{U},\mathbf{b}}(\underline{\iii}|_n)d\mu(\underline{\iii}) =\max_{\mathfrak{l}\in \mathcal{L}}\lim_{n \to \infty}\frac{1}{n}\int_{\I^\N}\log \Psi_{\mathbf{U}_{\mathfrak{l}},\mathbf{b}}(\underline{\iii}|_n)d\mu(\underline{\iii}),\]
and by the subadditive ergodic theorem this in turn is satisfied if and only if for every ergodic measure $\mu \in \mathcal{M}_\sigma$, for $\mu$-a.e.\ $\underline{\iii} \in \I^\N$
\begin{equation}\label{eq:quod-est-demonstrandum}
    \lim_{n \to \infty} \frac{1}{n} \log \Psi_{\mathbf{U},\mathbf{b}}(\underline{\iii}|_n) = \max_{\mathfrak{l}\in\mathcal{L}} \lim_{n \to \infty} \frac{1}{n} \log \Psi_{\mathbf{U}_{\mathfrak{l}},\mathbf{b}}(\underline{\iii}|_n).
\end{equation}
For the remainder of this section, therefore, we fix an ergodic measure $\mu \in \mathcal{M}_\sigma(\I^\N)$ and proceed to verify \eqref{eq:quod-est-demonstrandum}. We will apply the following lemma\footnote{We take this opportunity to note an alternative proof to the argument followed in \cite{Mo12}. By passing to an exterior power one may assume without loss of generality that the top Lyapunov exponent is simple; and under this assumption the method used by Avila and Bochi in \cite[Theorem 15]{AvBo02}, relying on Poincar\'e recurrence of the Oseledets spaces to demonstrate that the limit-superior exponential growth rate of the trace equals the top Lyapunov exponent, may be applied. This approach also generalises straightforwardly to the case of vector spaces over an arbitrary local field. On the other hand this simpler method is not applicable in the infinite-dimensional context which was the main subject of \cite{Mo12}.}, which may be found in \cite{Mo12}. In the following statement and the remainder of this subsection $\spectralradius(B)$ \label{notation:specrad} denotes the spectral radius of a linear map $B \in \End(V)$.
\begin{lemma}\label{le:Mo12} 
Let $T$ be an ergodic measure-preserving transformation of a probability space $(X,\mathcal{F},\nu)$, let $V$ be a finite-dimensional real vector space  and let $\mathcal{A} \colon X \times \mathbb{N} \to \GL(V)$ be a measurable function. Suppose that $\mathcal{A}(x,n+m)=\mathcal{A}(x,n)\mathcal{A}(T^nx,m)$ a.e.\ for every $n,m \geq 1$ and that $\int_X \log \|\mathcal{A}(x,1)\|d\nu(x)<\infty$. Then for $
\nu$-a.e.\ $x \in X$,
\[\limsup_{n \to \infty} \frac{1}{n}\log\spectralradius(\mathcal{A}(x,n)) = \lim_{n \to \infty} \frac{1}{n}\log \|\mathcal{A}(x,n)\|.\]
\end{lemma}
\begin{remark}
    An alternative proof can be pursued replacing the use of this lemma by Hennion's \cite[Proposition 1]{hennion}, who, together with Furstenberg--Kifer \cite{furstenberg-kifer} studied random matrix products without any algebraic assumptions and provided an invariant deterministic flag of subspaces which distinguishes growth of the matrices.
\end{remark}

Given $g \in G^o$ and a subquotient  $\R[G^o]$-module $W_j$ of $V_j$, we write $\spectralradius_{W_j}(\rho_j(g))$ for the spectral radius of the linear transformation of $W_j$ induced by $\rho_j(g)$. For each $h \in G$ and $g \in G^o$, by elementary linear algebra the characteristic polynomial of $\rho_j(g)$ acting on the vector space $\rho_j(h)U_j/\rho_j(h)U_j'$ factorises into the product over $\ell=1,\ldots,m_j$ of the characteristic polynomials of $\rho_j(g)$ acting on the vector spaces $\rho_j(h)U_j^\ell /\rho_j(h)U_j^{\ell-1}$, so in particular
\begin{equation}\label{eq:rho-split}\spectralradius_{\frac{\rho_j(h)U_j}{\rho_j(h)U_j'}}(\rho_j(g)) = \max_{1 \leq \ell \leq m_j} \spectralradius_{\frac{\rho_j(h)U_j^\ell}{\rho_j(h)U_j^{\ell-1}}}(\rho_j(g))\end{equation}
for every $j=1,\ldots,r$, $h \in G$ and $g \in G^o$.

Now let $\Lambda \subseteq \I^\N$ denote the set of all $\underline{\iii} \in \I^\N$ such that $g_{\underline{\iii}|_n} \in G^o$ for infinitely many $n\geq1$, 
 which is clearly a Borel set. For each $\underline{\iii} \in \Lambda$ and $n \geq 1$ define $\mathfrak{r}(\underline{\iii},n)$ \label{notation:return} to be the $n^{\mathrm{th}}$-smallest integer $k\geq1$   such that $g_{\underline{\iii}|_k}\in G^o$. We observe that the identity $\mathfrak{r}(\underline{\iii},n+m) = \mathfrak{r}(\sigma^{\mathfrak{r}(\underline{\iii},n)}\underline{\iii},m)+\mathfrak{r}(\underline{\iii},n)$ is satisfied for all $\underline{\iii}\in \Lambda$ and $n,m \geq 1$.  We claim that $\mu(\Lambda)=1$; that $\int_\Lambda \mathfrak{r}(\cdot,1)d\mu<\infty$; and that for $\mu$-a.e.\ $\underline{\iii} \in \I^\N$
\begin{align}\label{eq:sensible-internal-tag}
{\lefteqn{\lim_{n \to \infty} \frac{1}{\mathfrak{r}(\underline{\iii},n)}\log \left\|\rho_j\left(g_{\underline{\iii}|_{\mathfrak{r}(\iii,n)}}\right)\right\|_{\frac{\rho_j(h)U_j}{\rho_j(h)U_j'}}}}&\\\nonumber
&=\max_{1 \leq \ell \leq m_j}\lim_{n \to \infty} \frac{1}{\mathfrak{r}(\underline{\iii},n)}\log \left\|\rho_j\left(g_{\underline{\iii}|_{\mathfrak{r}(\iii,n)}}\right)\right\|_{\frac{\rho_j(h)U_j^\ell}{\rho_j(h)U_j^{\ell-1}}}    
\end{align}
for all $j=1,\ldots,r$ and $h \in G$. 

To prove these assertions 
we will modify the dynamical system $\sigma \colon \I^\N \to \I^\N$ in two ways. First, define a continuous function $T \colon \I^\N \times G/G^o\to\I^\N \times G/G^o$ by $T(\underline{\iii}, h G^o):=(\sigma \underline{\iii}, g_{i_1}^{-1} h G^o)$ and choose arbitrarily  an ergodic $T$-invariant measure $\nu$ on $\I^\N \times G/G^o$ which projects to the measure $\mu$ on $\I^\N$ and satisfies $\nu(\I^\N \times \{G^o\})>0$. (For example, we could take an ergodic decomposition of the $T$-invariant measure $[G : G^o]^{-1} \sum_{h G^o \in G/G^o} \mu \times \delta_{h G^o}$ and choose $\nu$ to be an arbitrary ergodic component satisfying $\nu(\I^\N \times \{G^o\}) \geq [G:G^o]^{-1}$.) Clearly $T^n (\underline{\iii},hG^o)=(\sigma^n\underline{\iii}, g_{\underline{\iii}|_n}^{-1}hG^o)$ for every $n \geq 1$ and $(\underline{\iii},hG^o) \in \I^\N \times G/G^o$, and for every coset  $hG^o \in G/G^o$ the set $\Lambda \times \{hG^o\}$ is easily seen to equal the set of all $(\underline{\iii},hG^o) \in \I^\N \times \{hG^o\}$ which return infinitely many times to $\I^\N \times \{hG^o\}$ under the action of $T$. Using the Poincar\'e recurrence theorem we find that 
\begin{align*}
\mu(\Lambda)=\nu\left(\Lambda \times G/G^o\right) &= \sum_{h G^o \in G/G^o} \nu\left(\Lambda \times \{hG^o\}\right)\\
&=\sum_{hG^o \in G/G^o} \nu\left(\I^\N \times \{hG^o\}\right)=\nu\left(\I^\N \times G/G^o\right)=1
\end{align*}
as required.  
To see that $\int_\Lambda \mathfrak{r}(\cdot,1)d\mu<\infty$ we note that the restriction of $\mathfrak{r}(\cdot,1)$ to $\Lambda \times \{hG^o\}$ is precisely the time of first return to $\Lambda \times \{hG^o\}$ under the transformation $T$, so using Kac's Lemma
\begin{align*}\int_\Lambda \mathfrak{r}(\cdot,1)d\mu&=\int_{\Lambda \times G/G^o}\mathfrak{r}(\cdot,1)d \nu\\
&=\sum_{hG^o \in G/G^o} \int_{\Lambda \times \{hG^o\}} \mathfrak{r}(\cdot,1)d\nu=\sum_{\substack{hG^o \in G/G^o\\ \nu(\I^\N\times \{hG^o\})>0}}\frac{1}{\nu(\I^\N \times \{hG^o\})}<\infty\end{align*}
as claimed. To recover \eqref{eq:sensible-internal-tag} we make a second modification to the dynamical system: we define $\sigma_{\mathfrak{r}} \colon \Lambda \to \Lambda$ by $\sigma_{\mathfrak{r}}(\underline{\iii}):=\sigma^{\mathfrak{r}(\underline{\iii},1)}\underline{\iii}$. We observe that by a simple induction $\sigma_{\mathfrak{r}}^n(\underline{\iii})=\sigma^{\mathfrak{r}(\underline{\iii},n)}(\underline{\iii})$ for all $n \geq 1$ and $\underline{\iii}\in \Lambda$. The first-return map of $T$ on $\Lambda \times \{G^o\}$ is precisely $(\underline{\iii},G^o)\mapsto (\sigma_{\mathfrak{r}}\underline{\iii},G^o)$, and since the first-return map of an ergodic transformation is also ergodic, it follows trivially that $\sigma_{\mathfrak{r}}$ is ergodic with respect to the measure $\mu$ on $\Lambda$. Now for each $j=1,\ldots,r$ and $hG^o \in G/G^o$ define a function $\mathcal{A}_{j,hG^o} \colon \Lambda \times \N \to \GL(\rho_j(h)U_j/\rho_j(h)U_j')$ by taking $\mathcal{A}_{j,hG^o}(\underline{\iii},n)$ to be the linear transformation of the $\R[G^o]$-module $\rho_j(h)U_j/\rho_j(h)U_j'$ which is induced by $\rho_j(g_{\underline{\iii}|_{\mathfrak{r}(\underline{\iii},n)}})$, and similarly for each $j=1,\ldots,r$, $\ell=1,\ldots,m_j$ and $hG^o \in G/G^o$ let $\mathcal{A}_{j,\ell,hG^o}(\underline{\iii},n)$ denote the linear transformation of $\rho_j(h)U_j^\ell/\rho_j(h)U_j^{\ell-1}$ induced by $\rho_j(g_{\underline{\iii}|_{\mathfrak{r}(\underline{\iii},n)}})$. Each of these functions $\mathcal{A}$ satisfies the right cocycle identity $\mathcal{A}(\underline{\iii},m+n)\equiv \mathcal{A}(\underline{\iii},n)\mathcal{A}(\sigma_{\mathfrak{r}}^n\underline{\iii},m)$, and in each case the integrability of $\log \|\mathcal{A}(\cdot,1)\|$ follows easily from the integrability of $\mathfrak{r}(\cdot,1)$. We may therefore apply Lemma \ref{le:Mo12} and the subadditive ergodic theorem to the map $\sigma_{\mathfrak{r}}$ to obtain, for every $j=1,\ldots,r$ and $h G^o \in G/G^o$ and for $\mu$-almost-every $\underline{\iii}\in \Lambda$,
\begin{align*}\lim_{n \to \infty} \frac{1}{n}\log \|\rho_j(g_{\underline{\iii}|_{\mathfrak{r}(\underline{\iii},n)}})\|_{\frac{\rho_j(h)U_j}{\rho_j(h)U_j'}}&=\limsup_{n \to \infty} \frac{1}{n}\log \spectralradius_{\frac{\rho_j(h)U_j}{\rho_j(h)U_j'}}(\rho_j(g_{\underline{\iii}|_{\mathfrak{r}(\underline{\iii},n)}})) \\
&=\limsup_{n \to \infty} \frac{1}{n}\log \max_{1 \leq \ell \leq m_j}\spectralradius_{\frac{\rho_j(h)U_j^\ell}{\rho_j(h)U_j^{\ell-1}}}(\rho_j(g_{\underline{\iii}|_{\mathfrak{r}(\underline{\iii},n)}})) \\
&=\max_{1 \leq \ell \leq m_j}\limsup_{n \to \infty} \frac{1}{n}\log \spectralradius_{\frac{\rho_j(h)U_j^\ell}{\rho_j(h)U_j^{\ell-1}}}(\rho_j(g_{\underline{\iii}|_{\mathfrak{r}(\underline{\iii},n)}})) \\
&=\max_{1 \leq \ell \leq m_j}\lim_{n \to \infty} \frac{1}{n}\log \left\|\rho_j(g_{\underline{\iii}|_{\mathfrak{r}(\underline{\iii},n)}})\right\|_{\frac{\rho_j(h)U_j^\ell}{\rho_j(h)U_j^{\ell-1}}},\end{align*}
where we have used \eqref{eq:rho-split}. By the Birkhoff ergodic theorem 
\[\lim_{n \to \infty} \frac{\mathfrak{r}(\underline{\iii},n)}{n} = \lim_{n \to \infty} \frac{1}{n}\sum_{k=0}^{n-1}\mathfrak{r}(\sigma_{\mathfrak{r}}^k\underline{\iii},1) = \int_\Lambda \mathfrak{r}(\cdot,1)d\mu \in [1,\infty)\]
and the result \eqref{eq:sensible-internal-tag} follows directly.

We may now prove the desired result \eqref{eq:quod-est-demonstrandum}. Using the subadditive ergodic theorem with respect to $\sigma$, for almost every $\underline{\iii} \in \Lambda$
\begin{align*} {\lefteqn{ \lim_{n \to \infty} \frac{1}{n}\log \Psi_{\mathbf{U},\mathbf{b}}(\underline{\iii}|_n) }}&\\
&=\lim_{n \to \infty} \frac{1}{\mathfrak{r}(\underline{\iii},n)}\log \Psi_{\mathbf{U},\mathbf{b}}(\underline{\iii}|_{\mathfrak{r}(\underline{\iii},n)})\\
&=\lim_{n \to \infty} \frac{1}{\mathfrak{r}(\underline{\iii},n)}\max_{hG^o \in G/G^o} \sum_{j=1}^r \beta_j \log \|\rho_j(g_{\underline{\iii}|_{\mathfrak{r}(\underline{\iii},n)}})\|_{\frac{\rho_j(h)U_j}{\rho_j(h)U_j'}}  \\
&=\max_{hG^o \in G/G^o}\sum_{j=1}^r \lim_{n \to \infty}  \frac{\beta_j}{\mathfrak{r}(\underline{\iii},n)} \log \|\rho_j(g_{\underline{\iii}|_{\mathfrak{r}(\underline{\iii},n)}})\|_{\frac{\rho_j(h)U_j}{\rho_j(h)U_j'}}\\
&=\max_{h G^o \in G/G^o}\sum_{j=1}^r \max_{1 \leq \ell_j \leq m_j} \lim_{n \to \infty}  \frac{\beta_j}{\mathfrak{r}(\underline{\iii},n)} \log \|\rho_j(g_{\underline{\iii}|_{\mathfrak{r}(\underline{\iii},n)}})\|_{\frac{\rho_j(h)U_j^{\ell_j}}{\rho_j(h)U_j^{\ell_j-1}}} \\
&=\max_{\mathfrak{l}=(\ell_j)_{j=1}^r \in \mathcal{L}} \lim_{n \to \infty}\frac{1}{\mathfrak{r}(\underline{\iii},n)} \max_{hG^o \in G/G^o}\sum_{j=1}^r \beta_j \log \|\rho_j(g_{\underline{\iii}|_{\mathfrak{r}(\underline{\iii},n)}})\|_{\frac{\rho_j(h)U_j^{\ell_j}}{\rho_j(h)U_j^{\ell_j-1}}} \\
&=\max_{\mathfrak{l}\in \mathcal{L}} \lim_{n \to \infty}\frac{1}{\mathfrak{r}(\underline{\iii},n)} \log \Psi_{\mathbf{U}_{\mathfrak{l}},\mathbf{b}}(\underline{\iii}|_{\mathfrak{r}(\underline{\iii},n)})=\max_{\mathfrak{l}\in \mathcal{L}} \lim_{n \to \infty}\frac{1}{n} \log \Psi_{\mathbf{U}_{\mathfrak{l}},\mathbf{b}}(\underline{\iii}|_n)\end{align*}
which is \eqref{eq:quod-est-demonstrandum}. We deduce \eqref{eq:max-L}, and we have proved Theorem \ref{th:main-tech}\ref{it:max-over-composition-factors}.

\subsection{Proof of Theorem \ref{th:main-tech}\ref{it:psi-mixing-tech}}
We first claim that there exists an integer $p \geq 1$ which divides $[G \colon G^o]$ and has the following property: if $H$ denotes the Zariski closure in $\GL(W)$ of the semigroup $\{g_\iii \colon p\text{ divides }|\iii|\}$, then $H$ is a finite-index normal subgroup of $G$ and satisfies $\bigcup_{|\iii|=pn} g_\iii H^o=H$ for all large enough $n \geq 1$. To prove this we will use the below modified statement of \cite[Lemma 3.11]{breuSert}. For subsets $A,B$ of an ambient group $G$ and an integer $n \geq 1$, we denote $AB:=\{ab: a \in A, b \in B\}$, $A^{-1}:=\{a^{-1}: a \in A\}$, $A^n:=\{a_1\cdots a_n:a_i \in G\}$, and $A^{-n}=(A^{-1})^n$.
\begin{lemma}\label{le:breuillard-sert-cyclic}
Let $F$ be a finite group and $E\subseteq F$ a generating set. Then there exist an integer $p \geq 1$ and a normal subgroup $N\unlhd F$ \label{notation:normal} such that $F/N$ is cyclic of order $p$ and $E^{np}=N$ for all large enough $n$. 
\end{lemma}
\begin{proof}
The increasing sequence of finite sets $E^{-n} E^n$ must be constant for all $n \geq m$, say. Define $N:=E^{-m}E^m$ and observe that $$NN= E^{-m}E^mE^{-m}E^m \subseteq (E^{(|F|-1)m}E^m)^{-1}E^{(|F|-1)m}E^m=E^{-m}E^m=N$$ so that $N$ is a nonempty subsemigroup of the finite group $F$, hence a group. Clearly $N$ is fixed under conjugation by all elements of the generating set $E$, hence $N \unlhd F$. 
If $f\in E$ is arbitrary then $\bigcup_{k \geq 1}f^kN  \supseteq\bigcup_{k \geq 1} E^k=F$ demonstrating that $F/N$ is cyclic; let $p$ denote its order. Choose $q \geq 1$ such that $E^{pq}$ contains the identity and note that the sequence of sets $E^{npq}$ increases to the semigroup generated by $E^{pq}$, which is the group generated by $E^{pq}$, which contains $E^{npq}E^{-npq}$ for all $n\geq 1$ and hence contains $N$. Since obviously $E^p\subseteq N$ we have shown that $E^{np}=N$ for some $n$, hence for all large enough $n$.
\end{proof}
To prove the claim, apply the lemma with $F:=G/G^o$ and $E:=\{g_i G^o  \colon i \in \I\}$. We observe that the group $H$ defined above equals $\bigcup_{ g_\iii G^o \in N}g_\iii G^o $. Indeed, obviously $H$ is contained in this set since the latter is Zariski closed and contains every $g_\iii$ such that $|\iii|$ is a large enough multiple of $p$; on the other hand $H$ contains $G^o$ since it has finite index in $G$, hence contains  $\bigcup_{p\text{ divides }|\iii|}g_\iii G^o \supseteq \bigcup_{g_\iii G^o  \in N}g_\iii G^o $. Obviously also $H^o=G^o$. We conclude that $\bigcup_{|\iii|=np} g_\iii H^o =H$ for all large enough $n$ and that $G/ H\simeq F/ N$ is cyclic of order $p$, where in particular $p$ divides $[G \colon G^o]$.

Choose cosets $h_1 H^o,\ldots,h_p H^o \in G/H^o$ whose images $h_1H,\ldots,h_pH$ in $G/H$ are all distinct, and for each $h_t H^o$ define  $\mathbf{U}_{h_t H^o}:=(\rho_j(h_t)U_j/\rho_j(h_t)U_j')_{j=1}^r \in \mathfrak{U}_S$. Consider the functions $\hat\Psi_{\mathbf{U}_{h_t H^o},\mathbf{b}} \colon \Gamma_{\I^p} \to (0,\infty)$ defined by 
\[\hat\Psi_{\mathbf{U}_{h_t H^o}, \mathbf{b}}(\jjj):=\max_{h' H^o \in H/H^o} \left(\prod_{j=1}^r \left\|\rho_j(g_\jjj)\right\|_{\frac{\rho_j(h'h_t)U_j}{\rho_j(h'h_t)U_j'} \to \frac{\rho_j(g_\jjj 
h'h_t)U_j}{\rho_j(g_\jjj h'h_t)U_j'}}^{\beta_j}\right)\]
defined for words $\jjj$ over the alphabet $\I^p$ (which we may naturally identify with words over $\I$ whose length is divisible by $p$). We apply Theorem \ref{th:main-tech}\ref{it:potential-basics} to the analysis of this potential, using the generators $(g_\iii)_{\iii \in \I^p}$ in place of the generators $(g_i)_{i \in \I}$ and the group $H$ in place of the group $G$, but leaving the other parameters unchanged. Since $H^o=G^o$ the $\R[H^o]$-submodules of each $V_j$ are precisely the $\R[G^o]$-submodules of the same $V_j$, so the sets $\mathfrak{U}$ and $\mathfrak{U}_S$ and the integers $n_j$ are unaffected by the change of parameters. In view of the fact that $\bigcup_{|\iii|=pn} g_\iii H^o=H$ for some $n \geq 1$ we deduce from Theorem \ref{th:main-tech}\ref{it:potential-basics} that each $\hat\Psi_{\mathbf{U}_{h_tH^o},\mathbf{b}}$ is a strongly quasi-multiplicative potential, and hence by Theorem \ref{th:properties-of-potentials}\ref{it:strong-quasi} each has a unique equilibrium state in $\mathcal{M}_\sigma((\I^p)^\N)$ which is ergodic and $\psi$-mixing. 

We may also naturally identify $\Psi_{\mathbf{U},\mathbf{b}}$ with a potential $\hat\Psi_{\mathbf{U},\mathbf{b}} \colon \Gamma_{\I^p} \to (0,\infty)$ which we define by restricting its domain to those words over $\I$ whose length is divisible by $p$ and then identifying these words with elements of $\Gamma_{\I^p}$ in the obvious fashion. The potential  $\hat\Psi_{\mathbf{U},\mathbf{b}}$ obviously satisfies $\hat\Psi_{\mathbf{U},\mathbf{b}} =\max_{t=1,\ldots,p} \hat\Psi_{\mathbf{U}_{h_tH^o},\mathbf{b}}$, so by a straightforward analysis using the subadditive ergodic theorem in the same manner as in the beginning of the previous section, the ergodic equilibrium states of $\hat\Psi_{\mathbf{U},\mathbf{b}}$ in $\mathcal{M}_\sigma((\I^p)^\N)$ are a nonempty subset of the set of equilibrium states of the potentials $\hat\Psi_{\mathbf{U}_{h_tH^o},\mathbf{b}}$. We now note that the equilibrium state $\mu \in \mathcal{M}_\sigma(\I^\N)$ of $\Psi_{\mathbf{U},\mathbf{b}}$ may be naturally identified with an invariant (but not necessarily ergodic) measure on $(\I^p)^\N$ which is an equilibrium state of $\hat\Psi_{\mathbf{U},\mathbf{b}}$ and hence is a convex combination of the latter's ergodic equilibrium states. We conclude from this analysis that the equilibrium state $\mu \in \mathcal{M}_\sigma(\I^\N)$ is necessarily a convex combination of not more than $\prod_{1 \leq j \leq r \colon \beta_j \neq 0}n_j$ distinct ergodic $\sigma^p$-invariant measures $\nu$ on $\I^\N$, each of which corresponds to the unique equilibrium state of one of the potentials $\hat{\Psi}_{\mathbf{U}_{h_tH^o},\mathbf{b}}$ and is therefore $\psi$-mixing with respect to $\sigma^p$ in the sense that
\[\lim_{n \to \infty} \sup_{\substack{\iii,\jjj \in \Gamma_\I\\ p\text{ divides }|\iii|, |\jjj|}} \left|\frac{\nu([\iii]\cap \sigma^{-|\iii|-np}[\jjj])}{\nu([\iii])\nu([\jjj])}-1\right|=0.\]
Choose arbitrarily one of these ergodic $\sigma^p$-invariant measures on $\I^\N$, which we denote hereafter by $\nu$. Let $q\geq 1$ denote the smallest integer such that $\sigma^q_*\nu=\nu$ and observe that $q$ is necessarily a factor of $p$, hence in particular is a factor of $[G\colon G^o]$. The measure $\frac{1}{q}\sum_{j=0}^{q-1} \sigma^j_*\nu$ is a $\sigma$-invariant Borel probability measure on $\I^\N$ which is absolutely continuous with respect to the ergodic $\sigma$-invariant measure $\mu$ and hence must be equal to $\mu$. Since this measure has precisely $q$ distinct ergodic components with respect to $\sigma^p$ we deduce that necessarily $q \leq \prod_{1 \leq j \leq r \colon \beta_j \neq 0}n_j$.

To complete the proof it remains only show that $\nu$ is $\psi$-mixing with respect to $\sigma^q$, which yields to a careful direct analysis as follows. Given $\varepsilon>0$, choose $N_\varepsilon$ large enough that if $n \geq N_\varepsilon$ then
\begin{equation}
    \label{eq:placeholder}
\left|\nu([\iii] \cap \sigma^{-|\iii|-np}[\jjj]) - \nu([\iii])\nu([\jjj])\right| <\varepsilon \nu([\iii])\nu([\jjj])\end{equation}
for all words $\iii,\jjj \in \Gamma_\I$ whose length is divisible by $p$. Now suppose that $\iii', \jjj'$ are arbitrary words of length divisible by $q$. Define $mq:=|\iii'|$, let $k$ be an arbitrary integer in the range $0 \leq k< p/q$, and let $\ell \geq 0$ be the smallest integer such that $p|(m+\ell)q$. The sets $[\iii']$ and $\sigma^{-p-(k-\ell)q}[\jjj']$ can respectively be written as the union of not more than $|\I|^{p}$ and $|\I|^{2p}$ cylinders of length divisible by $m$, with the cylinders used to obtain $[\iii']$ each having length precisely $(m+\ell)q$. It follows that for each $n \geq N_\varepsilon$
\begin{align*}&\left|\nu([\iii'] \cap \sigma^{-(m+\ell)q-np-(p-(k-\ell)q)}[\jjj']) - \nu([\iii'])\nu\left(\sigma^{-(p+(k-\ell)q)}[\jjj']\right)\right|\\
&<\varepsilon\nu([\iii'])\nu\left(\sigma^{-(p+(k-\ell)q)}[\jjj']\right)\end{align*}
by summing not more than $|\I|^{3p}$ appropriate instances of \eqref{eq:placeholder}, so in particular
\begin{align*}{\lefteqn{\left|\nu\left([\iii'] \cap \sigma^{-|\iii'|-(n+1)p-kq}[\jjj']\right) - \nu([\iii'])\nu([\jjj'])\right|}}&\\
&=\left|\nu\left([\iii'] \cap \sigma^{-mq-(n+1)p-kq}[\jjj']\right) - \nu([\iii'])\nu\left(\sigma^{-(p+(k-\ell)q)}[\jjj']\right)\right| \\
&=\left|\nu\left([\iii'] \cap \sigma^{-(m+\ell)q-np-(p+(k-\ell)q)}[\jjj']\right) - \nu([\iii'])\nu\left(\sigma^{-(p+(k-\ell)q)}[\jjj']\right)\right| \\
&<\varepsilon\nu([\iii'])\nu\left(\sigma^{-(p+(k-\ell)q)}[\jjj']\right)=\varepsilon\nu([\iii'])\nu([\jjj'])\end{align*}
where we have used the $\sigma^q$-invariance of $\nu$. We have shown that for arbitrary $n'=(n+1)(p/q)+k \geq (N_\varepsilon+2)(p/q)>(N_\varepsilon +1)(p/q)+k$,
\[\sup_{\substack{\iii',\jjj' \in \Gamma_\I\\ q\text{ divides }|\iii'|, |\jjj'|}}  \left|\frac{\nu([\iii']\cap \sigma^{-|\iii'|-n'q}[\jjj'])}{\nu([\iii'])\nu([\jjj'])}-1\right|\leq \varepsilon,\]
and since $\varepsilon>0$ was arbitrary, this gives $\psi$-mixing of $\nu$ with respect to $q$. The proof of Theorem \ref{th:main-tech}\ref{it:psi-mixing-tech} is complete.

\subsection{Proof of Theorem \ref{th:main-tech}\ref{it:q-pressure}}
For each $j=1,\ldots,r$ fix a nonzero linear map $Q_j \in \End(V_j)$ and let $Z_j$ \label{notation:Zj} be the maximal $\R[G^o]$-submodule of $V_j$ which is a subset of $\ker Q_j$. Define $\mathbf{Z}:=(V_j/Z_j)_{j=1}^r$. We require two lemmas:
\begin{lemma}\label{le:pupper}
There exists $K>0$ such that for every $\iii \in \Gamma_\I$ and $\mathbf{b}=(\beta_j)_{j=1}^r \in [0,\infty)^r$,
\[\prod_{j=1}^r \left\|Q_j\rho_j(g_\iii)\right\|^{\beta_j} \leq K^{\sum_{j=1}^r \beta_j}\Psi_{\mathbf{Z},\mathbf{b}}(\iii).\]
\end{lemma}
\begin{proof}
Define $C_1:=\max_j \|Q_j\|$. If $\iii \in \Gamma_\I^o$  then clearly
\begin{align*}\prod_{j=1}^r \|Q_j\rho_j(g_\iii)\|^{\beta_j} &=\prod_{j=1}^r \inf_{\substack{L_j \in \End(V_j)\\ L_jV_j \subseteq Z_j}} \|Q_j(\rho_j(g_\iii)+L_j)\|^{\beta_j}\\
&\leq \prod_{j=1}^r \left(\|Q_j\|^{\beta_j} \cdot \inf_{\substack{L_j \in \End(V_j)\\ L_jV_j \subseteq Z_j}}\|\rho_j(g_\iii)+L_j\|^{\beta_j}\right)\\
&\leq C_1^{\sum_{j=1}^r\beta_j} \prod_{j=1}^r\inf_{\substack{L_j \in \End(V_j)\\ L_jV_j \subseteq Z_j}}\|\rho_j(g_\iii)+L_j\|^{\beta_j}\\
&= C_1^{\sum_{j=1}^r\beta_j} \prod_{j=1}^r\|\rho_j(g_\iii)\|^{\beta_j}_{V_j/Z_j}\leq C_1^{\sum_{j=1}^r\beta_j} \Psi_{\mathbf{Z}, \mathbf{b}}(\iii).\end{align*}
Now let $F\subset \Gamma_\I$ be a finite set such that $\bigcup_{\kkk \in F} g_\kkk G^o  = G$, and define 
\[C_2:=\max_{\kkk \in F}\max_{1 \leq j \leq r} \max\{\|\rho_j(g_\kkk)\|, \|\rho_j(g_\kkk)^{-1}\|\}.\]
Given any $\iii \in \Gamma_\I$, choose $\kkk \in F$ such that $g_{\iii\kkk} \in G^o$ and observe that
\begin{align*}\prod_{j=1}^r\|Q_j \rho_j(g_\iii)\|^{\beta_j}&= \prod_{j=1}^r\|Q_j \rho_j(g_{\iii\kkk} )\rho_j(g_{\kkk})^{-1}\|^{\beta_j}  \leq C_2^{\sum_{j=1}^r \beta_j}\prod_{j=1}^r\|Q_j \rho_j(g_{\iii\kkk} )\|^{\beta_j}   \\
& \leq \left(C_1C_2\right)^{\sum_{j=1}^r \beta_j}\Psi_{\mathbf{Z},\mathbf{b}}(\iii \kkk)\leq \left(C_1C_2^2\right)^{\sum_{j=1}^r \beta_j} \Psi_{\mathbf{Z},\mathbf{b}}(\iii)\end{align*}
using the previous inequality. The lemma is proved.
\end{proof}

\begin{lemma}\label{le:plower}
There exist a finite set $F\subset \Gamma_\I$ and a real number $\kappa>0$ such that for every $\mathbf{b}=(\beta_j)_{j=1}^r \in [0,\infty)^r$ and $\iii \in \Gamma_\I$,
\[\max_{\kkk \in F} \prod_{j=1}^r \|Q_j \rho_j(g_{\kkk\iii})\|^{\beta_j}\geq \kappa^{\sum_{j=1}^r \beta_j} \Psi_{\mathbf{Z},\mathbf{b}}(\iii).\]
\end{lemma}\begin{proof}
For each $j=1,\ldots,r$ let $\mathbf{S}_{V_j/Z_j}$ denote the unit sphere of the real vector space $V_j/Z_j$. We first claim that there exists a finite set $F_1 \subset \Gamma_\I^o$ such that the real number
\[\tau:=\min_{([v_j])_{j=1}^r \in \prod_{j=1}^r \mathbf{S}_{V_j/Z_j}} \max_{\kkk \in F_1} \min_{1 \leq j \leq r} \|Q_j\rho_j(g_\kkk)v_j\|\]
is positive. Here we observe that the expression $\|Q_j \rho_j(g_\kkk)v_j\|$ is well-defined with respect to the choice of coset representative $v_j \in [v_j]$, since if $v_j \in V_j$ and $z_j \in Z_j$ then $Q_j
\rho_j(g_\kkk) (v_j+z_j)=Q_j
\rho_j(g_\kkk) v_j$ using the fact that $\rho_j(g_\kkk)Z_j=Z_j\subseteq \ker Q_j$. The same reasoning demonstrates that the value of $\|Q_j \rho_j(g_\kkk)v_j\|$ depends continuously on $[v_j] \in \mathbf{S}_{V_j/Z_j}$ when the other parameters are fixed.

To prove the claim we adopt a variation of the strategy used to prove Proposition \ref{pr:long-bridge}. 
We observe that it is sufficient to prove the above claim in the following reduced form: for every
 $([v_j])_{j=1}^r \in \prod_{j=1}^r \mathbf{S}_{V_j/Z_j}$ there exists $\kkk \in \Gamma_\I^o$ such that
 \[\min_{1 \leq j\leq r} \|Q_j \rho_j(g_\kkk)v_j\|\neq 0.\]
 To see that the reduced statement implies the claim, we 
 observe that if the former holds then the open sets
 \[O_\kkk:=\left\{[v_j]_{j=1}^r \in \prod_{j=1}^r \mathbf{S}_{V_j/Z_j} \colon \min_{1 \leq j\leq r} \|Q_j \rho_j(g_\kkk)v_j\|\neq 0\right\}\]
 for $\kkk \in \Gamma_\I^o$ form an open cover of $\prod_{j=1}^r \mathbf{S}_{V_j/Z_j}$, so choosing $F_1 \subset \Gamma_\I^o$ such that $(O_\kkk)_{\kkk \in F_1}$ is a finite subcover will then yield the positivity of $\tau$ by continuity and compactness. The reduced statement thus implies the claim, and we proceed to prove the former.

 Fix $([v_j])_{j=1}^r$ and a choice of coset representatives $v_1,\ldots,v_r$, and consider the sets
\[\mathcal{U}_j:=\left\{g \in G^o \colon Q_j\rho_j(g)v_j\neq 0\right\}.\]
Each $\mathcal{U}_j$ is obviously Zariski open, and each is also nonempty: if some $\mathcal{U}_j$ were empty, the $\R[G^o]$-linear span of $v_j+Z_j$ would be an $\R[G^o]$-module which strictly contains $Z_j$ and is a subset of $\ker Q_j$, contradicting the definition of $Z_j$. Since $G^o$ is Zariski connected, its nonempty Zariski-open subsets are Zariski dense, so each $\mathcal{U}_j$ is open and dense in $G^o$. In particular $\bigcap_{j=1}^r\mathcal{U}_j$ is a nonempty Zariski-open subset of $G^o$, so choosing any element of $\{\iii \in \Gamma_\I^o \colon g_\iii \in \bigcap_{j=1}^r \mathcal{U}_j\}$ yields   $\min_{1 \leq j\leq r} \|Q_j \rho_j(g_\iii)v_j\|\neq 0$ as required. We have proved the reduced form of the claim, and the existence of the finite set $F_1$ and positive real number $\tau$ follow. We fix such a set $F_1$ and constant $\tau>0$ for the remainder of the proof of the lemma.

We now claim that for every $\iii \in \Gamma_\I^o$ there exists $\kkk \in F_1$ such that
\[\prod_{j=1}^r \|Q_j \rho_j(g_\kkk g_\iii)\|^{\beta_j}  \geq \tau^{\sum_{j=1}^r \beta_j} \prod_{j=1}^r \|\rho_j(g_\iii)\|^{\beta_j}_{V_j/Z_j} \]
for every $\mathbf{b}=(\beta_j)_{j=1}^r \in [0,\infty)^r$. Indeed, given any $\iii \in \Gamma_\I^o$, for each $j=1,\ldots,r$ pick a unit vector $[u_j] \in V_j/Z_j$ such that $\|[\rho_j(g_\iii)u_j]\|_{V_j/Z_j} = \|\rho_j(g_\iii)\|_{V_j/Z_j}$, define a further unit vector by $[v_j]:=\|\rho_j(g_\iii)\|_{V_j/Z_j}^{-1}[\rho_j(g_\iii)u_j] \in V_j/Z_j$, and fix coset representatives $u_j, v_j \in V_j$ such that $\|v_j\|=\|[v_j]\|_{V_j/Z_j}=1$ and $\|u_j\|=\|[u_j]\|_{V_j/Z_j}=1$. Applying the previous claim, there exists $\kkk \in F_1$ such that for every $\mathbf{b}=(\beta_j)_{j=1}^r \in [0,\infty)^r$
\[\prod_{j=1}^r \|Q_j \rho_j(g_\kkk)v_j\|^{\beta_j} \geq \tau^{\sum_{j=1}^r \beta_j}\]
and therefore
\begin{align*}
\prod_{j=1}^r \|Q_j \rho_j(g_\kkk g_\iii)\|^{\beta_j} \geq \prod_{j=1}^r \|Q_j \rho_j(g_\kkk)\rho_j( g_\iii)u_j\|^{\beta_j}&= \prod_{j=1}^r \|\rho_j(g_\iii)\|^{\beta_j}_{V_j/Z_j} \|Q_j \rho_j(g_\kkk)v_j\|^{\beta_j} 
\\
&\geq \tau^{\sum_{j=1}^r \beta_j} \prod_{j=1}^r \|\rho_j(g_\iii)\|^{\beta_j}_{V_j/Z_j}  \end{align*}
as required to prove the claim. 

Now let $F_2 \subset \Gamma_\I$ be a finite set such that $\bigcup_{\jjj \in F_2} g_\jjj G^o =G$, and define \[F:=F_1F_2=\left\{\kkk\jjj \colon \kkk \in F_1 \text{ and }\jjj \in F_2\right\}\]
and
\[C:=\max_{1 \leq j \leq r} \max_{\kkk \in F_1 \cup F_2}\max\{\|\rho_j(g_\kkk)\|, \|\rho_j(g_\kkk)^{-1}\|\}.\]
Given $\mathbf{b}=(\beta_j)_{j=1}^r \in [0,\infty)^r$ and $\iii \in \Gamma_\I$, choose $\jjj_2\in F_2$ such that 
\[\Psi_{\mathbf{Z},\mathbf{b}}(\iii)=\prod_{j=1}^r \|\rho_j(g_\iii)\|^{\beta_j}_{\frac{V_j}{\rho_j\left(g_{\jjj_2}\right)Z_j} \to \frac{V_j}{\rho_j\left(g_{\iii\jjj_2}\right)Z_j} }\]
and choose $\jjj_1 \in F_2$ such that $g_{\jjj_1\iii \jjj_2} \in G^o$. Applying the previous claim to the word $\jjj_1\iii\jjj_2 \in \Gamma_\I^o$, we may choose $\kkk \in F_1$ such that 
\begin{align*}\prod_{j=1}^r \|Q_j \rho_j(g_\kkk g_{\jjj_1\iii\jjj_2})\|^{\beta_j} &\geq \tau^{\sum_{j=1}^r \beta_j} \prod_{j=1}^r \|\rho_j(g_{\jjj_1\iii\jjj_2})\|^{\beta_j}_{V_j/Z_j}\\
&\geq \left(C^{-2}\tau\right)^{\sum_{j=1}^r\beta_j}\prod_{j=1}^r \|\rho_j(g_\iii)\|^{\beta_j}_{\frac{V_j}{\rho_j(g_{\jjj})Z_j} \to \frac{V_j}{\rho_j(g_{\iii\jjj})Z_j}} \\
&= \left( C^{-2}\tau\right)^{\sum_{j=1}^r \beta_j}\Psi_{\mathbf{Z},\mathbf{b}}(\iii)\end{align*}
and hence clearly
\[\prod_{j=1}^r \|Q_j \rho_j(g_{\kkk\jjj_1\iii})\|^{\beta_j}\geq \left( C^{-3}\tau\right)^{\sum_{j=1}^r \beta_j}\Psi_{\mathbf{Z},\mathbf{b}}(\iii).\]
Since $\kkk\jjj_1 \in F$ the lemma is proved.
\end{proof}
We may now prove the assertions of Theorem \ref{th:main-tech}\ref{it:q-pressure}. We begin with \eqref{eq:q-potential-ratio}. Lemma \ref{le:pupper} directly yields the upper bound 
\[\sum_{|\iii|=n} \prod_{j=1}^r \|Q_j\rho_j(g_\iii)\|^{\beta_j} \leq K^{\sum_{j=1}^r \beta_j} \sum_{|\iii|=n}\Psi_{\mathbf{Z},\mathbf{b}}(\iii)\]
for every $n \geq 1$. Now let $F \subset \Gamma_\I$ and $\kappa>0$ be as given by Lemma \ref{le:plower}, define once more
\[C_1:=\max_{1 \leq j \leq r} \max_{\kkk \in F} \max\{\|\rho_j(g_\kkk)\|, \|\rho_j(g_\kkk)^{-1}\|\}\]
and let $m$ denote the maximum length of a word in $F$. 
Fix $n \geq 1$. By Lemma \ref{le:plower} we have
\begin{align*}\sum_{\ell=1}^{m} \sum_{|\iii|=n+\ell} \prod_{j=1}^r \|Q_j \rho_j(g_\iii)\|^{\beta_j}& \geq \sum_{\kkk \in F}\sum_{|\jjj|=n} \prod_{j=1}^r \|Q_j \rho_j(g_{\kkk\jjj})\|^{\beta_j}\\
&\geq \sum_{|\jjj|=n}\max_{\kkk \in F} \prod_{j=1}^r \|Q_j \rho_j(g_{\kkk}g_\jjj )\|^{\beta_j}\geq \kappa^{\sum_{j=1}^r \beta_j}\sum_{|\jjj|=n}\Psi_{\mathbf{Z},\mathbf{b}}(\jjj)\end{align*}
so in particular we may choose an integer $k$ in the range $1 \leq k \leq m$ such that 
\[\sum_{|\iii|=n+k} \prod_{j=1}^r \|Q_j \rho_j(g_\iii)\|^{\beta_j} \geq \frac{1}{m}\sum_{\ell=1}^{m} \sum_{|\iii|=n+\ell} \prod_{j=1}^r \|Q_j \rho_j(g_\iii)\|^{\beta_j}\geq \frac{\kappa^{\sum_{j=1}^r \beta_j}}{m} \sum_{|\iii|=n} \Psi_{\mathbf{Z},\mathbf{b}}(\iii)\]
and since clearly
\begin{align*}\sum_{|\iii|=n+k} \prod_{j=1}^r \|Q_j \rho_j(g_\iii)\|^{\beta_j} &\leq \left(\sum_{|\iii|=n}\sum_{|\jjj|=k} \prod_{j=1}^r \|Q_j \rho_j(g_\iii)\|^{\beta_j}\|\rho_j(g_\jjj)\|^{\beta_j}\right)\\
&\leq \left(|\I| \cdot C_1^{\sum_{j=1}^r \beta_j}\right)^{k}\sum_{|\iii|=n} \prod_{j=1}^r \|Q_j \rho_j(g_\iii)\|^{\beta_j}\end{align*}
we conclude that
\[\sum_{|\iii|=n} \prod_{j=1}^r \|Q_j\rho_j(g_\iii)\|^{\beta_j} \geq \frac{1}{m |\I|^m} \left(\frac{\kappa}{C_1^m}\right)^{\sum_{j=1}^r \beta_j}\sum_{|\iii|=n} \Psi_{\mathbf{Z},\mathbf{b}}(\iii)\]
for every $n \geq 1$ as required for the lower bound in \eqref{eq:q-potential-ratio}.  
Since the sequence $\log \sum_{|\iii|=n}\Psi_{\mathbf{Z},\mathbf{b}}(\iii)$ is subadditive as a consequence of the results of Theorem \ref{th:main-tech}\ref{it:potential-basics}, the existence and description of the limit in \eqref{eq:general-limit} follow immediately.

We now consider the ergodic theory of the functions $\underline{\iii}\mapsto \prod_{j=1}^r \|Q_j\rho_j(\underline{\iii}|_n)\|^{\beta_j}$. We begin by demonstrating that for every ergodic measure $\mu \in \mathcal{M}_\sigma(\I^\N)$, the limit
\begin{equation}\label{eq:exists-ae-limit} \lim_{n \to \infty} \frac{1}{n}\log \left(\frac{\prod_{j=1}^r \|Q_j \rho_j(g_{\underline{\iii}|_n})\|^{\beta_j}}{\mu([\underline{\iii}|_n])}\right)\end{equation}
exists $\mu$-almost-everywhere. By the Shannon-McMillan-Breiman theorem the sequence $\frac{1}{n}\log \mu([\underline{\iii}|_n])$ converges a.e.\, so it is sufficient to show that for each $j=1,\ldots,r$ the limit
\[\lim_{n \to \infty}\frac{1}{n} \log \|Q_j\rho_j(g_{\underline{
\iii}|_n})\|\]
also exists $\mu$-a.e. For this we will appeal to the Oseledets multiplicative ergodic theorem. Since in our work it is necessary for us to extend products on the right --- as $\rho_j(g_{i_1})\rho_j(g_{i_2})\cdots \rho_j(g_{i_n})$, rather than  $\rho_j(g_{i_n})\cdots \rho_j(g_{i_2})\rho_j(g_{i_1})$ --- where standard statements of Oseledets' theorem instead extend products by multiplication on the left, we will find it easier to perform the equivalent task of showing that the limit
\[\lim_{n \to \infty}\frac{1}{n} \log \|(Q_j\rho_j(g_{\underline{
\iii}|_n}))^*\|=\lim_{n \to \infty} \frac{1}{n} \log \|\rho_j(g_{i_n})^* \cdots \rho_j(g_{i_1})^*Q_j^*\|\]
exists $\mu$ a.e.\ for every $j=1,\ldots,r$.

To this end, fix $j$ and define a continuous function $\mathcal{A}_j \colon \I^\N \times \N \to \GL(V_j^*)$ by $\mathcal{A}_j(\underline{\iii},n):=\rho_j(g_{\underline{\iii}|_n})^*=\rho_j(g_{i_n})^*\cdots \rho_j(g_{i_1})^*$. The function $\mathcal{A}_j$ satisfies the left cocycle identity $\mathcal{A}_j(\underline{\iii},n+m) \equiv \mathcal{A}_j(\sigma^m \underline{\iii},n)\mathcal{A}_j(\underline{\iii},m)$. By the non-invertible formulation of the Oseledets multiplicative ergodic theorem \cite{oseledets} there exist a Borel set $\Lambda \subseteq \I^\N$ satisfying $\mu(\Lambda)=1$, a natural number $p$, real numbers $\lambda_1>\lambda_2>\cdots>\lambda_p$, and Borel measurable functions $\mathcal{W}_1,\ldots,\mathcal{W}_{p+1}$ from $\Lambda$ to the Grassmannian of $V_j^*$ such that for every $\underline{\iii} \in \Lambda$,\[V_j^*=\mathcal{W}_1(\underline{\iii})>\mathcal{W}_2(\underline{\iii})>\cdots >\mathcal{W}_p(\underline{\iii})>\mathcal{W}_{p+1}(\underline{\iii})=\{0\}\]
and for each $k=1,\ldots,p$
\[\lim_{n \to \infty} \frac{1}{n} \log \|\mathcal{A}_j(\underline{\iii},n)v\|=\lambda_k\]
for all  $v \in \mathcal{W}_k(\underline{\iii}) \setminus \mathcal{W}_{k+1}(\underline{\iii})$. Given $\underline{\iii} \in \Lambda$, let $\ell$ be the largest integer such that the image of $Q_j^*$ is contained in $\mathcal{W}_\ell(\underline{\iii})$. Since $Q_j$ is nonzero we have $\ell<p+1$. If  $u_1,\ldots,u_d$ is a basis for $V_j^*$ then clearly
\[\limsup_{n \to \infty} \frac{1}{n}\log \|\mathcal{A}_j(\underline{\iii},n)Q_j^*\| \leq \max_{1 \leq i \leq d} \limsup_{n \to \infty} \frac{1}{n}\log \left\|\mathcal{A}_j(\underline{\iii},n)Q_j^*u_i \right\| \leq \lambda_\ell \]
since each of the vectors $Q_j^*u_i$ is an element of $\mathcal{W}_\ell(\underline{\iii})$. On the other hand, 
by the maximality of $\ell$ there exists $u \in V_j^*$ such that $Q_j^*u \in \mathcal{W}_\ell(\underline{\iii})\setminus \mathcal{W}_{\ell+1}(\underline{\iii})$, whence
\[\liminf_{n\to\infty} \frac{1}{n}\log \|\mathcal{A}_j(\underline{\iii},n)Q_j^* \| \geq \lim_{n\to\infty} \frac{1}{n}\log \|\mathcal{A}_j(\underline{\iii},n)Q_j^* u\|=\lambda_\ell. \]
It follows that for this particular $\underline{\iii}$
\[\lim_{n \to \infty} \frac{1}{n}\log \|Q_j\rho_j(g_{\underline{\iii}|_n})\|=\lim_{n \to \infty} \frac{1}{n}\log \|\mathcal{A}_j(\underline{\iii},n)Q_j^*\|=\lambda_\ell\]
and since $\underline{\iii} \in \Lambda$ was arbitrary we have proved that the limit \eqref{eq:exists-ae-limit} exists $\mu$-a.e.\ as required.

It remains to consider the relationship of the limit \eqref{eq:exists-ae-limit} to the pressure of $\Psi_{\mathbf{Z},\mathbf{b}}$. If $\mu \in \mathcal{M}_\sigma(\I^\N)$ is ergodic, then using Lemma \ref{le:pupper}, the Shannon-McMillan-Breiman theorem, the subadditive ergodic theorem and Theorem \ref{th:properties-of-potentials}\ref{it:variational-principle},
\begin{align}\label{eq:esssup}{\lefteqn{{\ess\sup}_{\mu,\underline{\iii}}\lim_{n \to \infty} \frac{1}{n}\log \left(\frac{\prod_{j=1}^r \|Q_j \rho_j(g_{\underline{\iii}|_n})\|^{\beta_j}}{\mu([\underline{\iii}|_n])}\right)}} &\\\nonumber
&\leq {\ess\sup}_{\mu,\underline{\iii}} \lim_{n \to \infty} \frac{1}{n}\log \left(\frac{\Psi_{\mathbf{Z},\mathbf{b}}(\underline{\iii}|_n)}{\mu([\underline{\iii}|_n])}\right)\\\nonumber
&=h(\mu) + \lim_{n \to \infty}\frac{1}{n}\int_{\I^\N} \log \Psi_{\mathbf{Z},\mathbf{b}}(\underline{\jjj}|_n)d\mu(\underline{\jjj})\leq P(\Psi_{\mathbf{Z},\mathbf{b}}).\end{align}
Clearly  if the first expression is equal to the last, then all  expressions are equal and $\mu$ is an equilibrium state of $\Psi_{\mathbf{Z},\mathbf{b}}$ as required. To complete the proof of Theorem \ref{th:main-tech}\ref{it:q-pressure} we will show that if $\mu$ is an ergodic equilibrium state of  $\Psi_{\mathbf{Z},\mathbf{b}}$ then the first expression in \eqref{eq:esssup} must be equal to the pressure $P(\Psi_{\mathbf{Z},\mathbf{b}})$. To this end fix an ergodic equilibrium state $\mu$ of $\Psi_{\mathbf{Z},\mathbf{b}}$ and let $F \subset \Gamma_\I$ be as given by Lemma \ref{le:plower}. Using Lemma \ref{le:plower} we find that for $\mu$-a.e.\ $\underline{\iii}\in \I^\N$,
\begin{align*}\limsup_{n \to \infty}\frac{1}{n}\log \max_{\kkk \in F} \left(\frac{\prod_{j=1}^r \|Q_j \rho_j(g_\kkk g_{\underline{\iii}|_n})\|^{\beta_j}}{\mu([\underline{\iii}|_n])}\right) &\geq \lim_{n \to \infty} \frac{1}{n}\log \left(\frac{\Psi_{\mathbf{Z},\mathbf{b}}(\underline{\iii}|_n)}{\mu([\underline{\iii}|_n])}\right)=P(\Psi_{\mathbf{Z},\mathbf{b}}),\end{align*}
so in particular there exist $\kkk_0 \in F$ and a Borel set $\Omega\subseteq \I^\N$ such that $\mu(\Omega)>0$ and 
\begin{equation}\label{eq:limsup}\limsup_{n \to \infty}\frac{1}{n}\log  \left(\frac{\prod_{j=1}^r \|Q_j \rho_j(g_{\kkk_0} g_{\underline{\iii}|_n})\|^{\beta_j}}{\mu([\underline{\iii}|_n])}\right)\geq P(\Psi_{\mathbf{Z},\mathbf{b}})\end{equation}
for every $\underline{\iii} \in \Omega$. Now, by parts \ref{it:potential-basics} and \ref{it:max-over-composition-factors} of Theorem \ref{th:main-tech} together with Theorem \ref{th:properties-of-potentials}\ref{it:quasi-multiplicative}, there exist a potential $\Psi_{\mathbf{U},\mathbf{b}}$ with $\mathbf{U} \in \mathfrak{U}_S$ and $\mathbf{U}\preceq \mathbf{Z}$, and a constant $C_2\geq 1$, such that 
\[C_2^{-1}\leq \frac{e^{|\iii|P(\Psi_{\mathbf{U},\mathbf{b}})}\mu([\iii])}{ \Psi_{\mathbf{U},\mathbf{b}}(\iii)} \leq C_2\]
for every $\iii \in \Gamma_\I$. Consequently
\[\frac{\mu([\kkk_0\iii])}{\mu([\iii])} \leq C_2^2e^{-|\kkk_0|P(\Psi_{\mathbf{U},\mathbf{b}})} \frac{\Psi_{\mathbf{U},\mathbf{b}}(\kkk_0\iii)}{\Psi_{\mathbf{U},\mathbf{b}}(\iii)}\leq C_2^2e^{-|\kkk_0|P(\Psi_{\mathbf{U},\mathbf{b}})} \Psi_{\mathbf{U},\mathbf{b}}(\kkk_0)\]
and
\[\frac{\mu([\kkk_0\iii])}{\mu([\iii])} \geq C_2^{-2}e^{-|\kkk_0|P(\Psi_{\mathbf{U},\mathbf{b}})} \frac{\Psi_{\mathbf{U},\mathbf{b}}(\kkk_0\iii)}{\Psi_{\mathbf{U},\mathbf{b}}(\iii)}\geq C_1^{-\sum_{j=1}^r \beta_j} C_2^{-2}e^{-|\kkk_0|P(\Psi_{\mathbf{U},\mathbf{b}})}\]
for every $\iii \in \Gamma_\I$,
from  which we conclude that there exists $C_3\geq 1$ such that
\begin{equation}\label{eq:cylinder-distortion}
C_3^{-1}\mu([\iii])\leq \mu([\kkk_0\iii])=\mu\left([\kkk_0] \cap \sigma^{-|\kkk_0|}[\iii]\right) \leq C_3\mu([\iii])\end{equation}
for every $\iii \in \Gamma_\I$. It follows trivially that
\[C_3^{-1}\mu(Y)\leq \mu([\kkk_0] \cap \sigma^{-|\kkk_0|}Y) \leq C_3 \mu(Y)\]
whenever $Y \subseteq \I^\N$ is a finite union of cylinder sets, and by an approximation argument the same holds whenever $Y\subseteq \I^\N$ is Borel. The set $[\kkk_0]\cap \sigma^{-|\kkk_0|}\Omega$ thus has positive measure, and for every $\underline{\jjj}$ belonging to that set
\[\limsup_{n \to \infty}\frac{1}{n}\log  \left(\frac{\prod_{j=1}^r \|Q_j \rho_j( g_{\underline{\jjj}|_n})\|^{\beta_j}}{\mu([\underline{\jjj}|_n])}\right) \geq P(\Psi_{\mathbf{Z},\mathbf{b}})\]
using \eqref{eq:limsup}, where we have also used \eqref{eq:cylinder-distortion} to control the behaviour of the denominator.
Since the limit exists almost everywhere we conclude that 
\[{\ess\sup}_{\mu,\underline{\iii}}\lim_{n \to \infty} \frac{1}{n}\log \left(\frac{\prod_{j=1}^r \|Q_j \rho_j(g_{\underline{\iii}|_n})\|^{\beta_j}}{\mu([\underline{\iii}|_n])}\right) \geq P(\Psi_{\mathbf{Z},\mathbf{b}})\]
and this completes the proof of Theorem \ref{th:main-tech}\ref{it:q-pressure}.

\subsection{Proof of Theorem \ref{th:main-tech}\ref{it:exist-subvarieties}}
The proof rests on the following purely algebraic lemma:
\begin{lemma}\label{le:minimal-submodule}
Let $R$ be a ring, $M$ an $R$-module of finite length, and $X$ a simple $R$-module. Then there exists a submodule $N\leq M$ with the following property: for every submodule $M'$ of $M$, the quotient module $M/M'$ has no composition factors isomorphic to $X$ if and only if it satisfies $N \leq M'$.
\end{lemma}
\begin{proof}
Given an $R$-module $M$ of finite length, let us say that a submodule $N \leq M$ \emph{excludes $X$} if the quotient $M/N$ has no composition factors isomorphic to $X$, and let $\mathscr{E}_X(M)$ denote the set of all submodules of $M$ which exclude $X$. We note that if $N \in \mathscr{E}_X(M)$ and $N \leq N' \leq M$ then $N' \in \mathscr{E}_X(M)$ 
also, since the composition factors 
of $M/N'$ form a subset of those of $M/N$. We also note that by similar reasoning, if $M' \in \mathscr{E}_X(M)$ then $\mathscr{E}_X(M')\subseteq \mathscr{E}_X(M)$. Clearly $\mathscr{E}_X(M)$ is always nonempty since it contains at least the module $M$. 
For every finite-length $R$-module $M$ the set $\mathscr{E}_X(M)$ has at least one minimal element with respect to inclusion. We claim that
$\mathscr{E}_X(M)$ always has a \emph{unique} minimal element.
To see that this claim implies the lemma, fix $M$ and suppose that $N$ is the unique minimal element of $\mathscr{E}_X(M)$. If $M'\leq M$ excludes $X$, choose any  minimal element $N' \in \mathscr{E}_X(M')$ and note that it is also a minimal element of $\mathscr{E}_X(M)$, hence equal to $N$, so $N=N'\leq M'$. If on the other hand $N \leq M'$ then $M'$ excludes $X$ since the composition factors of $M/M'$ are a subset of those of $M/N$. The lemma thus follows from the claim.

It remains to prove the claim. 
We use induction on the composition length of $M$. If the composition length of $M$ is zero then clearly $\{0\}$ is the desired unique minimal submodule. Otherwise, fix $M$, suppose that the result has been proved for all modules of strictly smaller composition length, and for the remainder of the proof fix an arbitrary simple submodule $S$ of $M$. 
By the induction hypothesis $\mathscr{E}_X(M/S)$ contains a unique minimal element, which can be
written as $\hat{N}/S$ for a  unique module $\hat{N}$ such that $S \leq \hat{N} \leq M$.
We then have $M/\hat{N} \simeq (M/S)/(\hat{N}/S)$, and by the definition of $\hat{N}$ the latter has no composition factors isomorphic to $X$, so we deduce that $\hat{N} \in \mathscr{E}_X(M)$.    

We now claim that if $N$ is a minimal element of $\mathscr{E}_X(M)$, then one of two outcomes holds: either $\hat{N}=N$, or  $\hat{N}= N \oplus S$ and also $S \not\simeq X$. Suppose that $\hat{N}\neq N$. We show in turn that $\hat{N} \leq N+S$, that $N \cap S=\{0\}$, that $S \not\simeq X$ and that $N+ S\leq\hat{N}$. Clearly, $N+S\in \mathscr{E}_X(M)$, and since $(M/S)/((N+S)/S) \simeq M/(N+S)$, we have $(N+S)/S \in \mathscr{E}_X(M/S)$.
By the induction hypothesis there exists a unique minimal submodule of $(N+S)/S\leq M/S$  which excludes $X$, and since this submodule is also a minimal element of $\mathscr{E}_X(M/S)$ it must by the induction hypothesis equal $\hat{N}/S$, so $\hat{N}/S \leq (N+S)/S$ and therefore $\hat{N} \leq N+S$. 
Since $S$ is simple, the intersection $N \cap S$ is either the zero module or $S$. In the latter case we would have $\hat{N} \leq N+S=N \neq \hat{N}$ yielding $\hat{N}<N$ which would contradict the minimality of $N$ in $\mathscr{E}_X(M)$, so $N \cap S$ must be $\{0\}$ and the sum $N+S$ must be direct. Now, the composition factors of $M/N$ are the composition factors of $M/(N\oplus S)$ together with those of $(N\oplus S)/N$, so in particular $S$ is a composition factor of $M/N$ and is therefore not isomorphic to $X$. 
It remains only to demonstrate that $N\oplus S\leq \hat{N}$. Since $N+S \leq N+\hat{N} \leq N+S$ we have $N+\hat{N}=N \oplus S$ and therefore $\hat{N}/(N\cap \hat{N})\simeq (N+\hat{N})/N = (N\oplus S)/N \simeq S$ by the second isomorphism theorem. 
The composition factors of $M/(N \cap \hat{N})$ are the composition factors of $M/\hat{N}$ together with the composition factors of $\hat{N}/(N\cap \hat{N})$, and since $S \not\simeq X$ and $\hat{N}\in \mathscr{E}_X(M)$ it follows that $N \cap \hat{N} \in \mathscr{E}_X(M)$. Since $N$ is minimal in $\mathscr{E}_X(M)$ we cannot have $N \cap \hat{N}<N$, so $N \cap \hat{N}=N$, which means precisely that $N \leq\hat{N}$ and thence $N+S \leq\hat{N}$. The claim is proved.

To complete the proof of the induction step it suffices to show that there cannot be two distinct minimal elements of $\mathscr{E}_X(M)$. If two such elements $N_1, N_2$ exist then it follows from the previous claim that 
$\hat{N}=N_1\oplus S = N_2 \oplus S$ with $S\not\simeq X$.
We now claim that $((N_1 \cap N_2)+S)/S$ is a proper submodule of $\hat{N}/S$ which excludes $X$, contradicting the minimality of $\hat{N}/S$ in $\mathscr{E}_X(M/S)$. To show this claim, note that by the minimality and distinctness of $N_1$ and $N_2$, we have $N_1 \cap N_2<N_1$ and therefore $(N_1 \cap N_2)+S < N_1+S=\hat{N}$ using the fact that the sum $N_1+S$ is direct; therefore 
$((N_1 \cap N_2)+S)/S<\hat{N}/S$.
We now show that $((N_1 \cap N_2)+S)/S \in \mathscr{E}_X(M/S)$. Every composition factor of $(M/S)/(((N_1 \cap N_2)+S)/S)$ is either a composition factor of $(M/S)/(\hat{N}/S)$ or a composition factor of  $(\hat{N}/S)/((N_1 \cap N_2)+S)/S)$. By definition of $\hat{N}/S$ the former has no composition factors isomorphic to $X$, while on the other hand
\begin{align*}\frac{\hat{N}/S}{(N_1\cap N_2)+S)/S} &\simeq \frac{\hat{N}}{(N_1 \cap N_2)+S} \\
&=\frac{N_1\oplus S}{(N_1 \cap N_2)\oplus S}\simeq \frac{N_1}{N_1 \cap N_2}\simeq \frac{N_1+N_2}{N_2}=\frac{\hat{N}}{N_2} \simeq S \not\simeq X\end{align*}
using the standard isomorphism theorems and the previous claim. This completes the proof of the induction step and hence proves the lemma.
\end{proof}

We may now prove Theorem \ref{th:main-tech}\ref{it:exist-subvarieties}. For every $j=1,\ldots,r$ let $X_{j,1},\ldots,X_{j,n_j}$ be a maximal list of non-pairwise-isomorphic composition factors of the $\R[G^o]$-module $V_j$. Using Lemma \ref{le:minimal-submodule}, for every $\ell=1,\ldots,n_j$ let $M_{j,\ell}$ be the submodule of $V_j$ such that for every submodule $N$ of $V_j$, the quotient module $V_j/N$ has no composition factors isomorphic to $X_{j,\ell}$  if and only if $M_{j,\ell} \leq N$. For each $\mathbf{U}=(U_j/U_j')_{j=1}^r \in \mathfrak{U}_S$,  by the Jordan-H\"older theorem there exists for every $j=1,\ldots,r$ a unique integer $\ell_j(\mathbf{U})$ such that $U_j/U_j'$ is isomorphic to $X_{j,\ell_j(\mathbf{U})}$ as an $\R[G^o]$-module. 

Now let $t \in \R$ and $\mathbf{b}=(\beta_j)_{j=1}^r \in [0,\infty)^r$
and consider the two sets
\begin{equation}\label{eq:first-set}\left\{(Q_j)_{j=1}^r \in \prod_{j=1}^r \End(V_j) \colon \lim_{n \to \infty} \frac{1}{n}\log \sum_{|\iii|=n} \prod_{j=1}^r \|Q_j \rho_j(g_\iii)\|^{\beta_j}\leq t\right\}\end{equation}
and
\begin{equation}\label{eq:second-set}\bigcap_{\substack{\mathbf{U} \in \mathfrak{U}_S \\ P(\Psi_{\mathbf{U},\mathbf{b}})> t}} \bigcup_{\substack{1 \leq i \leq r \\ \beta_i \neq 0}} \bigcap_{v \in M_{i,\ell_i(\mathbf{U})}} \left\{(Q_j)_{j=1}^r \in \prod_{j=1}^r \End(V_j) \colon Q_i v=0\right\},\end{equation}
where we understand the latter as being equal to $\prod_{j=1}^r \End(V_j)$ in the case where $P(\Psi_{\mathbf{U},\mathbf{b}})\leq t$ for every $\mathbf{U} \in \mathfrak{U}_S$. We claim that the two sets are necessarily equal. 
We can assume without loss of generality that $\beta_j>0$ for every $j=1,\ldots,r$. 

 Consider now an arbitrary tuple $(Q_j)_{j=1}^r \in \prod_{j=1}^r \End(V_j)$.  If $Q_i=0$ for some value of $i$  then $(Q_j)_{j=1}^r$ clearly belongs to both of the two sets, so suppose instead that every $Q_i$ is nonzero. For each $j=1,\ldots,r$ let $Z_j<V_j$ be the unique maximal $\R[G^o]$-submodule of $V_j$ which is a subset of the vector space $\ker Q_j$, and define $\mathbf{Z}=(V_j/Z_j)_{j=1}^r \in \mathfrak{U}$. We claim that the following statements are all equivalent:
\begin{enumerate}[1.]
\item 
The limit
\[\lim_{n \to \infty} \frac{1}{n}\log \sum_{|\iii|=n} \prod_{j=1}^r \|Q_j \rho_j(g_\iii)\|^{\beta_j}\]
is less than or equal to $t$.
\item
The pressure $P(\Psi_{\mathbf{Z}, \mathbf{b}})$ is less than or equal to $t$.
\item 
There does not exist $\mathbf{U} \in \mathfrak{U}_S$ such that $\mathbf{U} \preceq \mathbf{Z}$ and $P(\Psi_{\mathbf{U}, \mathbf{b}})>t$.
    \item 
For every $\mathbf{U} \in \mathfrak{U}_S$ such that $P(\Psi_{\mathbf{U}, \mathbf{b}})>t$, there exists $j$ such that $X_{j,\ell_j(\mathbf{U})}$ is not isomorphic to any composition factor of the module $V_j/Z_j$. 
\item 
For every $\mathbf{U} \in \mathfrak{U}_S$ such that $P(\Psi_{\mathbf{U}, \mathbf{b}})>t$, there exists $j$ such that $M_{j,\ell_j(\mathbf{U})} \leq Z_j$.   
\item 
For every $\mathbf{U} \in \mathfrak{U}_S$ such that $P(\Psi_{\mathbf{U}, \mathbf{b}})>t$, there exists $j$ such that $M_{j,\ell_j(\mathbf{U})} \subseteq \ker Q_j$. 
\item 
For every $\mathbf{U} \in \mathfrak{U}_S$ such that $P(\Psi_{\mathbf{U}, \mathbf{b}})>t$, there exists $j$ such that $Q_jv=0$ for every $v \in M_{j,\ell_j(\mathbf{U})}$. 
\end{enumerate}
Indeed, the equivalence of the first and second statements is given directly by Theorem \ref{th:main-tech}\ref{it:q-pressure} and that of the second and third statements by Theorem \ref{th:main-tech}\ref{it:max-over-composition-factors}. To compare the third and fourth statements we consider their negations. If the third statement is false then there exists $\mathbf{U}=(U_j/U_j')_{j=1}^r \in \mathfrak{U}_S$ such that $\mathbf{U} \preceq \mathbf{Z}$ and $P(\Psi_{\mathbf{U}, \mathbf{b}})>t$, and for every $j=1,\ldots,r$, by definition $X_{j,\ell_j(\mathbf{U})}$ is isomorphic to the composition factor $U_j/U_j'$ of $V_j/Z_j$. Thus the fourth statement is false in this case. Conversely, if the fourth statement is false then there exists $\mathbf{U} \in \mathfrak{U}_S$ such that $P(\Psi_{\mathbf{U},\mathbf{b}})>t$ and such that for every $j=1,\ldots,r$, the module $X_{j,\ell_j(\mathbf{U})}$ is isomorphic to some composition factor $\hat{U}_j/\hat{U}_j'$ of $V_j/Z_j$. Now define $\hat{\mathbf{U}}:=(\hat{U}_j/\hat{U}_j')_{j=1}^r \preceq \mathbf{Z}$ and note that $\hat{\mathbf{U}} \in \mathfrak{U}_S$ and that $P(\Psi_{\hat{\mathbf{U}},\mathbf{b}})=P(\Psi_{\mathbf{U},\mathbf{b}})>t$ by Theorem \ref{th:main-tech}\ref{it:potential-basics}. Thus in this case the third statement is false. We have shown so far that the first four statements are equivalent. But the remaining equivalences are straightforward: the fourth and fifth statements coincide due to the defining property of the modules $M_{j,\ell}$; the fifth and sixth statements agree due to the defining property of the modules $Z_j$; and the final equivalence simply re-states the definition of the kernel of a linear map. The seven statements are therefore equivalent as claimed. 
 Since in particular the first and final statements are equivalent, we conclude that $(Q_j)_{j=1}^r$ belongs to the set \eqref{eq:first-set}
if and only if it belongs to \eqref{eq:second-set}. The two sets are thus equal, and since the set \eqref{eq:second-set} is evidently algebraic we have proved that in all cases the sub-level set \eqref{eq:first-set} is algebraic as required. 

To complete the proof of the theorem it suffices to count the number of distinct sets which may arise from the formula \eqref{eq:second-set}. Suppose first that the vector $\mathbf{b} \in [0,\infty)^r$ is fixed. By Theorem \ref{th:main-tech}\ref{it:potential-basics} the value of $P(\Psi_{\mathbf{U}, \mathbf{b}})$ for $\mathbf{U}=(U_j/U_j')_{j=1}^r\in \mathfrak{U}_S$ depends only on the isomorphism class of $U_j/U_j'$ for each value of $j$ such that $\beta_j \neq 0$, so $P(\Psi_{\mathbf{U}, \mathbf{b}})$ can take at most $m=\prod_{1 \leq j\leq r\colon \beta_j \neq 0} n_j$ distinct values as $\mathbf{U}$ varies throughout $\mathfrak{U}_S$. Call these values $t_1 < t_2 < \cdots < t_{m'}$, say, where $1 \leq m'\leq m$. Clearly the value of the set \eqref{eq:second-set} is determined completely by which of the intervals $(-\infty, t_1)$, $[t_1, t_2)$, \ldots, $[t_{m'-1}, t_m)$, $[t_m,\infty)$ contains $t$, so the set \eqref{eq:second-set} -- and hence  
\eqref{eq:first-set} -- can take at most $m+1$ distinct values as $t \in \R$ varies while $\mathbf{b}$ is held fixed. On the other hand if both $\mathbf{b}$ and $t$ are allowed to vary,  then it is clear that the number of possible values of \eqref{eq:second-set}, and hence again of \eqref{eq:first-set}, is not greater than the number of ways of first choosing a subset of the integers $\{1,\ldots,r\}$ in order to determine which parameters $\beta_j$ are to be nonzero, and then choosing for each selected integer a subset of the isomorphism classes of composition factors of the module $V_j$ in order to determine over which equivalence classes of potentials $\mathbf{U}\in \mathfrak{U}_S$ the first union will be taken. The total number of such choices is thus $\sum_{J \in \mathcal{P}(\{1,\ldots,r\})} \prod_{j \in J} 2^{n_j} = \prod_{j=1}^r (1+2^{n_j})$ as required.
The proof of the theorem is complete.

\section{Derivation of Theorem \ref{th:pressure-detail}}\label{se:derivation}

\subsection{Approximate subadditivity, variational properties and sub-level sets} 
We first establish those clauses of Theorem \ref{th:pressure-detail} whose proof is directly dependent on Theorem \ref{th:main-tech}, namely parts \ref{it:exist-subadditive}, \ref{it:thermodynamic}  and \ref{it:sublevel-algebraic}. To this end define $W:=\R^d$ and $g_i:=A_i \in \GL(W)$ for every $i\in \I$ and let $G$ denote the Zariski closure in $\GL_d(\R)$ of the semigroup $\{A_\iii \colon \iii \in \Gamma_\I\}$.

We first treat those clauses of \ref{it:exist-subadditive} and \ref{it:thermodynamic} which pertain to the case of a fixed map $Q \in \End(\R^d)$. Fix nonzero $Q \in \End(\R^d)$, define $r:=\rank Q$ and define $V_j:=\wedge^j\R^d$,  $Q_j:=Q^{\wedge j} \neq 0$
 and $\rho_j (g):=g^{\wedge j}$ for every $j=1,\ldots,r$ and $g \in G$. For each $j=1,\ldots,r$ let $n_j$ denote the number of distinct composition factors of $V_j$ as an $\R[G^o]$-module, and note that $1 \leq n_j\leq \dim V_j={d \choose j}$. For each $s \in [0,\rank Q]$ define a vector $\mathbf{b}(s)=(\beta_j(s))_{j=1}^r \in [0,\infty)^r$ by
\[\beta_j(s)=\left\{\begin{array}{cl}
0&\text{if }j<\lfloor s\rfloor\\
j+1-s&\text{if }j=\lfloor s \rfloor\\
s+1-j&\text{if }j=\lceil s \rceil \\
0&\text{if }j>\lceil s\rceil,
\end{array}\right.\]
and let $\mathfrak{U}$ and $\mathfrak{U}_S$ be as defined in the statement of Theorem \ref{th:main-tech}. We observe that $\sum_{j=1}^r\beta_j(s)\equiv 1$. 
For each $j=1,\ldots,r$ let $Z_j$ denote the maximal $\R[G^o]$-submodule of $\wedge^j \R^d$ which is a subset of $\ker Q_j$ and define $\mathbf{Z}:=(V_j/Z_j)_{j=1}^r\in \mathfrak{U}$. In view of the identity
\begin{equation}\label{eq:yay-for-adding-more-equations}\varphi^s(QA_\iii) = \left\|(QA_\iii)^{\wedge \lfloor s\rfloor}\right\|^{1+\lfloor s\rfloor-s}\left\|(QA_\iii)^{\wedge \lceil s\rceil}\right\|^{s-\lfloor s\rfloor}=\prod_{j=1}^r \left\|Q_j\rho_j(g_\iii)\right\|^{\beta_j(s)}\end{equation}
and Theorem \ref{th:main-tech}\ref{it:q-pressure}, there exist $C,\kappa>0$ such that for every $n \geq 1$ and $s \in [0,r]$
\[\frac{\kappa}{C} \sum_{|\iii|=n} \Psi_{\mathbf{Z},\mathbf{b}(s)}(\iii) \leq \sum_{|\iii|=n} \varphi^s(QA_\iii) \leq C\sum_{|\iii|=n} \Psi_{\mathbf{Z},\mathbf{b}(s)}(\iii)\]
 where $\Psi_{\mathbf{Z},\mathbf{b}(s)}$ is as defined in the statement of Theorem \ref{th:main-tech}. The subadditivity of the sequence
\[n \mapsto \log \left(\frac{C^3}{\kappa^2}\sum_{|\iii|=n}\varphi^s(QA_\iii)\right)\]
for every $s \in [0,r]$ follows directly from these inequalities together with the submultiplicativity of $\Psi_{\mathbf{Z},\mathbf{b}(s)}$, and this proves Theorem \ref{th:pressure-detail}\ref{it:exist-subadditive}.
By further application of Theorem \ref{th:main-tech}\ref{it:q-pressure} in combination with \eqref{eq:yay-for-adding-more-equations} we observe also that for each $s \in [0,r]$ and each ergodic measure $\mu \in \mathcal{M}_\sigma(\I^\N)$,
\[P_Q(\A,s)=P(\Psi_{\mathbf{Z},\mathbf{b}(s)})\geq 
{\ess\sup}_{\mu,\underline{\iii}} \lim_{n \to \infty} \frac{1}{n} \log \left(\frac{\varphi^s(QA_{\underline{\iii}|_n})}{\mu([\underline{\iii}|_n])}\right)\]
with equality if and only if $\mu$ is an equilibrium state of $\Psi_{\mathbf{Z},\mathbf{b}(s)}$.  The a.e. existence of the relevant limits for every ergodic measure $\mu$ also follows from Theorem \ref{th:main-tech}\ref{it:q-pressure}. These properties easily extend to arbitrary invariant measures on $\I^\N$ by an ergodic decomposition argument.

Fix $s \in [0,r]$. It follows from the final clause of Theorem \ref{th:main-tech}\ref{it:potential-basics} that the number of equivalence classes of potentials of the form $\Psi_{\mathbf{U},\mathbf{b}(s)}$ where $\mathbf{U} \in \mathfrak{U}_S$ is at least one and is not greater than $n_{\lfloor s\rfloor}n_{\lceil s\rceil}\leq {d \choose \lfloor s\rfloor }{d \choose \lceil s\rceil}$ if $s$ is non-integer, or $n_s \leq {d \choose s}$ if $s$ is integer. Let $\Psi_{\mathbf{U}_1,\mathbf{b}(s)},\ldots,\Psi_{\mathbf{U}_p,\mathbf{s}}$ be a maximal list of non-pairwise-equivalent potentials of this form, and observe that the length of this list and the potentials which are in it both have no dependence on $Q$. By Theorem \ref{th:main-tech}\ref{it:max-over-composition-factors} there exists at least one ergodic equilibrium state for $\Psi_{\mathbf{Z},\mathbf{b}(s)}$, and every such equilibrium state is the unique equilibrium state of some $\Psi_{\mathbf{U},\mathbf{b}(s)}$ such that $\mathbf{U} \in \mathfrak{U}_S$ and $\mathbf{U} \preceq \mathbf{Z}$. Since $\Psi_{\mathbf{U},\mathbf{b}(s)}$ must be equivalent to one of the potentials $\Psi_{\mathbf{U}_t,\mathbf{b}(s)}$, we conclude that every ergodic $(\varphi^s,Q)$-equilibrium state is an equilibrium state of one of the potentials $\Psi_{\mathbf{U}_t,\mathbf{b}(s)}$. By Theorem \ref{th:main-tech}\ref{it:potential-basics} each of these potentials is quasi-multiplicative, so by appeal to Theorem \ref{th:properties-of-potentials} we conclude that there exists at least one $(\varphi^s,Q)$-equilibrium state for $\A$, that all such equilibrium states have the form described in Theorem \ref{th:pressure-detail}\ref{it:thermodynamic}\ref{it:eqm-gibbs}, and that every such equilibrium state is the unique equilibrium state of one of the potentials 
$\Psi_{\mathbf{U}_1,\mathbf{b}(s)},\ldots, \Psi_{\mathbf{U}_p,\mathbf{b}(s)}$, though not all equilibrium states of those potentials are necessarily $(\varphi^s,Q)$-equilibrium states for any particular $Q$. These observations collectively establish all parts of Theorem \ref{th:pressure-detail}\ref{it:thermodynamic} except for \ref{it:eqm-psi-mixing}, which follows by direct application of Theorem \ref{th:main-tech}\ref{it:psi-mixing-tech}.

We now consider Theorem \ref{th:pressure-detail}\ref{it:sublevel-algebraic}. Here we slightly vary the parameters in our application of Theorem \ref{th:main-tech} by taking $r:=d$ and then defining $W$, $g_i$, $G$, $V_j$, $\rho_j$, $n_j$ and $\mathbf{b}(s)$ in the same way as before. For every $s \in [0,d]$ and $t \in \R$, the set
 \[\hat{\mathcal{V}}_{s,t}:=\left\{(Q_j)_{j=1}^d \in \prod_{j=1}^d \End(\wedge^j \R^d)\colon \lim_{n \to \infty} \frac{1}{n}\log \sum_{|\iii|=n} \prod_{j=1}^d \left\|Q_j\rho_j(A_\iii)\right\|^{\beta_j(s)} \leq t \right\}\]
is an affine subvariety of $\prod_{j=1}^d \End(\wedge^j \R^d)$ by Theorem \ref{th:main-tech}\ref{it:exist-subvarieties}, and additionally we have
\[\left|\left\{\hat{\mathcal{V}}_{s,t} \colon t\in \R\right\}\right| \leq 1+\prod_{\substack{1 \leq j\leq d\\\beta_j(s)\neq 0}} n_j = \left\{\begin{array}{cl} 1+n_s&\text{if }n_s \in \Z\\
1+n_{\lfloor s\rfloor}n_{\lceil s\rceil}&\text{if }n_s\notin \Z
\end{array} \right.\]
for all $s \in [0,d]$, and 
\[\left|\left\{\hat{\mathcal{V}}_{s,t} \colon s \in [0,d]\text{ and }t\in \R\right\}\right| \leq \prod_{j=1}^d \left(1+2^{n_j}\right) .\]
The sub-level sets
\[\mathcal{V}_{s,t}:=\left\{Q \in \End(\R^d) \colon  \lim_{n \to \infty} \frac{1}{n}\log \sum_{|\iii|=n} \varphi^s(QA_\iii)\leq t \right\}\]
clearly satisfy
\[\mathcal{V}_{s,t}=\left\{Q \in \End(\R^d) \colon  (Q^{\wedge j})_{j=1}^d \in {\hat{\mathcal{V}}}_{s,t}\right\}\]
for every $s \in [0,d]$ and $t \in \R$, and each is thus the preimage of an algebraic set with respect to an algebraic function. Thus each $\mathcal{V}_{s,t}$ is an algebraic set as claimed, and the claimed bound on the number of distinct such sets is immediate. 

To complete the proof of Theorem \ref{th:pressure-detail}\ref{it:sublevel-algebraic} it remains only to prove the invariance of the sub-level sets $\mathcal{V}_{s,t}$ with respect to right multiplication by the linear maps $A_i$. For fixed $i\in\I$, $Q \in \End(\R^d)$, $s \in [0,d]$ and $t \in \R$ we  clearly have
\[P_{QA_i}(\A,s) =\lim_{n \to \infty} \frac{1}{n}\log \sum_{|\jjj|=n} \varphi^s(QA_i A_\jjj) \leq \lim_{n \to \infty} \frac{1}{n}\log \sum_{|\iii|=n+1} \varphi^s(Q A_\iii) = P_Q(\A,s)\]
which directly yields $\mathcal{V}_{s,t}A_i \subseteq \mathcal{V}_{s,t}$. It follows immediately that $\mathcal{V}_{s,t}A_i^{n+1} \subseteq \mathcal{V}_{s,t}A_i^{n}$ for every $n \geq 0$, and since $A_i$ is linear and invertible each of these sets is an affine subvariety of $\End(\R^d)$. By the Noetherian property we deduce that $\mathcal{V}_{s,t}A_i^{n+1}=\mathcal{V}_{s,t}A_i^n$ for all large enough $n$ from which it follows that $\mathcal{V}_{s,t}A_i=\mathcal{V}_{s,t}$ as required. This concludes the proof of Theorem \ref{th:pressure-detail}\ref{it:sublevel-algebraic}.

The proofs of the remaining parts of Theorem \ref{th:pressure-detail} are comparatively elementary and make no further direct use of Theorem \ref{th:main-tech}.

\subsection{Lipschitz continuity and convexity} The proof of Theorem \ref{th:pressure-detail}\ref{it:lipschitz-convexity} modifies \cite[\S4]{Fa88} to accommodate the presence of the non-invertible $Q$. Choose $\kappa>0$ such that $e^{-\kappa} \leq \sigma_d(A_i)\leq \sigma_1(A_i) \leq e^\kappa$ for every $i\in \I$. If $0 \leq s\leq s+t \leq \rank Q$ then clearly
\[P_Q(\A, s+t) = \lim_{n \to \infty} \frac{1}{n}\log \sum_{|\iii|=n} \varphi^{s+t}(QA_\iii)
\leq  P_Q(\A, s)+\kappa t,\]
and likewise
\[P_Q(\A,s+t) 
\geq   \lim_{n \to \infty} \frac{1}{n}\log \sum_{|\iii|=n} \varphi^{s}(QA_\iii)\sigma_{\lceil s+t\rceil}(Q)^t \sigma_d(A_\iii)^t \geq P_Q(\A, s)-\kappa t,\]
using the inequality $\sigma_k(B_1B_2) \geq \sigma_k(B_1)\sigma_d(B_2) \geq \sigma_{\lceil s+t\rceil} (B_1)\sigma_d(B_2) $ for $\lceil s\rceil \leq k\leq\lceil s+t \rceil$. It follows that  $\left|P_Q(\A, s_1) - P_Q(\A, s_2)\right| \leq \kappa |s_1-s_2|$ for all $s_1, s_2 \in [0, \rank Q]$. To demonstrate convexity on $[\ell-1,\ell]$ we simply note that if $\ell-1 \leq s_1 < s_2 \leq \ell$ and $\alpha\in (0,1)$ then for every $n \geq 1$
\begin{align*}\sum_{|\iii|=n} \varphi^{\alpha s_1 + (1-\alpha)s_2}(QA_\iii) &=\sum_{|\iii|=n}\varphi^{s_1}(QA_\iii)^\alpha \varphi^{s_2}(QA_\iii)^{1-\alpha}\\
&\leq\left(\sum_{|\iii|=n}\varphi^{s_1}(QA_\iii)\right)^\alpha \left(\sum_{|\iii|=n}\varphi^{s_2}(QA_\iii)\right)^{1-\alpha }\end{align*}
by H\"older's inequality with $p=\frac{1}{\alpha}$ and $q=\frac{1}{1-\alpha}$, and the result follows.

\subsection{Monotonicity properties} We conclude by establishing clauses \ref{it:monotone-A} and \ref{it:monotone-ker-q} of Theorem \ref{th:pressure-detail}. To deal with \ref{it:monotone-A} we apply the arguments used in the case $Q=\id$ in \cite{BoMo18,KaMo18} essentially verbatim. If $Q \in \End(\R^d)$, $s \in [0,\rank Q]$ and $\A'=(A_i)_{i \in \J}$ where $\emptyset\neq \J \subsetneq \I$ then trivially $P_Q(\A',s) \leq P_Q(\A,s)$ by inspection of the definition, so to prove \ref{it:monotone-A} we assume equality and derive a contradiction. If $P_Q(\A',s) = P_Q(\A,s)$
then by Theorem \ref{th:pressure-detail}\ref{it:thermodynamic}\ref{it:eqm-at-least-one} there exists at least one $(\varphi^s,Q)$-equilibrium state of $\A'$, and by Theorem \ref{th:pressure-detail}\ref{it:thermodynamic}\ref{it:eqm-gibbs} the support of this measure is precisely $\J^\N$. This measure may be viewed as a shift-invariant measure on $\I^\N$ with support equal to $\J^\N\subsetneq \I^\N$, and since $P_Q(\A',s) = P_Q(\A,s)$ this measure is easily seen to be a $(\varphi^s,Q)$-equilibrium state of $\A$; but its support is a proper subset of $\I^\N$, contradicting Theorem \ref{th:pressure-detail}\ref{it:thermodynamic}\ref{it:eqm-gibbs}, and we conclude that the equation $P_Q(\A',s) = P_Q(\A,s)$ is impossible as required.
 
We now turn our attention to \ref{it:monotone-ker-q}. Suppose first that $Q_1, Q_2 \in \End(\R^d)$ satisfy $\ker Q_1 \subseteq \ker Q_2$. It follows that there exists $B \in \End(\R^d)$ such that $BQ_1=Q_2$ and given such a matrix $B$ we may now compute that for every $s \in [0, \rank Q_2]$
\begin{align*}P_{Q_2}(\A,s)=P_{BQ_1}(\A,s)&=\lim_{n \to \infty}\frac{1}{n}\log \sum_{|\iii|=n}\varphi^s(BQ_1A_\iii)\\
& \leq  \lim_{n \to \infty}\frac{1}{n}\log\left(\varphi^s(B)\cdot \sum_{|\iii|=n}\varphi^s(Q_1A_\iii)\right)=P_{Q_1}(\A,s)\end{align*}
as required. If $\ker Q_1 = \ker Q_2$ then since $\ker Q_1 \subseteq \ker Q_2 \subseteq \ker Q_1$ the preceding argument directly yields $P_{Q_2}(\A,s) \leq P_{Q_1}(\A,s) \leq P_{Q_2}(\A,s)$ for every $s \in [0,\rank Q_1]=[0,\rank Q_2]$
as required to finish the proof of Theorem \ref{th:pressure-detail}.

\section{Dimension estimates: proof of Theorem \ref{th:dimension-results}}\label{se:dimension}
The proof of Theorem \ref{th:dimension-results} is rooted in recent modifications (notably those of \cite{BaSiSo23} and \cite{FeLoMa22}) to the classic arguments of Falconer, K\"aenm\"aki and Jordan-Pollicott-Simon used in \cite{Fa88,JoPoSi07,Ka04}. Due to the critical role played by Theorem \ref{th:dimension-results} in this article, we include a relatively detailed proof.  
\subsection{Upper bounds}
When treating the upper bound we will suppress the notational dependence on $\mathbf{v} \in (\R^d)^{\I}$. We thus fix an affine iterated function system $(T_i)_{i \in 
\I}$ acting on $\R^d$ which is contracting with respect to some norm $\threebar{\cdot}$ and which has linearisation $(A_i)_{i \in \I}$, attractor $X$ and coding map $\Pi$.  For each $\delta>0$ we will say that a \emph{$\delta$-mesh cube} in $\R^d$ is precisely a set of the form $\prod_{j=1}^d [k_j\delta, (k_j+1)\delta)$ where $k_1,\ldots,k_d \in \Z$. If $Z\subset \R^d$ is an arbitrary bounded set, we let $\mathcal{N}_Z(\delta)$ \label{notation:NZdelta} denote the number of $\delta$-mesh cubes which intersect $Z$. 

When studying both sets and measures we will apply the following simple extension of the classic covering argument of \cite[\S5]{Fa88}, which may be proved by the same arguments.
\begin{lemma}\label{le:elementary-falconer}
Let $\mathbf{B}\subset \R^d$ be a closed Euclidean ball of diameter $L \geq 1$, let $B \in \End(\R^d)$ and let $s \in (0,\rank B]$. If $T \colon \R^d \to \R^d$ is an affine transformation of the form $Tx\equiv Bx+v$, then $\mathcal{N}_{T\mathbf{B}}(\sigma_{\lceil s\rceil}(B)) \leq CL^d \varphi^s(B)/\sigma_{\lceil s\rceil}(B)^s$ where $C>0$ depends only on $d$.
\end{lemma}
\begin{proof}[Proof of Theorem \ref{th:dimension-results}\ref{it:dimension-upper-bounds}]

We first establish the bound on the upper box dimension of the attractor. Fix a linear map $Q \in \End(\R^d)$, and fix also a closed Euclidean ball $\mathbf{B}\subset \R^d$ which contains the attractor $X$ and which has diameter $L \geq 1$. To bound the dimension of $X$ we will apply the formula
\[\dimaffQ \A=\min\left\{\rank Q, \inf\left\{s \geq 0 \colon \sum_{n=1}^\infty \sum_{|\iii|=n} \varphi^s(QA_\iii)<\infty\right\}\right\}.\]
If $\dimaffQ =\rank Q$ then the trivial bound $\overline{\dim}_\B X \leq \rank Q$ concludes the proof, so we suppose this not to be the case and fix $s \in (0,\rank Q)$ such that the series $\sum_{n=1}^\infty \sum_{|\iii|=n}\varphi^s(QA_\iii)$ converges. To prove the required dimension bound we will show that $\mathcal{N}_{QX}(\delta) =O(\delta^{-s})$ as $\delta \to 0$.

Fix $\delta \in (0, \min_{i \in \I} \sigma_{\lceil s\rceil}(QA_i))$ for the remainder of the proof. 
For every $\underline{\iii}\in \I^\N$ let $m(\underline{\iii},\delta)$ denote the smallest integer $m \geq 1$ such that $\sigma_{\lceil s\rceil}(QA_{\underline{\iii}|_m})< \delta$. By the choice of $\delta$ we always have $m(\underline{\iii},\delta) \geq 2$, and we deduce that 
\[
\sigma_{\lceil s\rceil}\left(QA_{\underline{\iii}|_{m(\underline{\iii},\delta)}}\right) \geq \delta \cdot \min_{i\in \I} \sigma_d(A_i)\]
since otherwise the minimality of $m(\underline{\iii},\delta)$ would be contradicted. Now the cylinders $[\underline{\iii}|_{m(\underline{\iii},\delta)}]$ for $\underline{\iii} \in \I^\N$ together constitute an open cover of $\I^\N$, so let us choose a finite subcover of $\I^\N$ by distinct cylinders $[\kkk_1],\ldots,[\kkk_p]$, say. Clearly
\begin{equation}\label{eq:cover-tight} \left(\min_{i \in \I}\sigma_d(A_i)\right)\cdot\delta \leq \sigma_{\lceil s\rceil} (QA_{\kkk_j})< \delta \end{equation}
for every $j=1,\ldots,p$. Since clearly
\[QX=Q\Pi \I^\N = \bigcup_{j=1}^p Q\Pi [\kkk_j] =\bigcup_{j=1}^p QT_{\kkk_j}X \subseteq \bigcup_{j=1}^p QT_{\kkk_j}\mathbf{B}\]
we have
\begin{align*} \mathcal{N}_{QX}(\delta) \leq \sum_{j=1}^p \mathcal{N}_{QT_{\kkk_j}\mathbf{B}} (\delta)
&\leq 2^d \sum_{j=1}^p \mathcal{N}_{QT_{\kkk_j}\mathbf{B}} (\sigma_{\lceil s\rceil}(QA_{\kkk_j}))\\
&\leq C(2L)^d \sum_{j=1}^p \frac{\varphi^s(QA_{\kkk_j})}{\sigma_{\lceil s\rceil}(QA_{\kkk_j})^s}\\
&\leq C\left(\frac{2L}{\min_{i \in \I}\sigma_d(A_i)}\right)^d\sum_{j=1}^p \frac{\varphi^s(QA_{\kkk_j})}{\delta^s}\\
&\leq C\delta^{-s}\cdot \left(\frac{2L}{\min_{i \in \I}\sigma_d(A_i)}\right)^d\left(\sum_{n=1}^\infty \sum_{|\iii|=n}  \varphi^s(QA_\iii)\right),\end{align*}
where we have used \eqref{eq:cover-tight}, Lemma \ref{le:elementary-falconer} and the elementary inequality $\mathcal{N}_Z(\varepsilon') \leq 2^d\mathcal{N}_Z(\varepsilon)$ which holds whenever $0<\varepsilon<\varepsilon'$ and $Z\subseteq \R^d$. Since the series converges this gives the desired bound $\mathcal{N}_{QX}(\delta) =O(\delta^{-s})$ and proves the box dimension bound stated in Theorem \ref{th:dimension-results}\ref{it:dimension-upper-bounds}.

We now turn our attention to the bound for the upper local dimensions of measures. We retain the linear map $Q$ and ball $\mathbf{B}$ of diameter $L\geq 1$ but we otherwise discard the parameters from the preceding argument. Fix an ergodic measure $\mu \in \mathcal{M}_\sigma(\I^\N)$. For every $n \geq 1$, $k=1,\ldots,\rank Q$ and $\underline{\iii} \in \I^\N$ define
$\mathcal{Q}_{n,k}(\underline{\iii})\subseteq \I^\N$ to be the set of all $\underline{\jjj} \in [\underline{\iii}|_n]$ such that $Q\Pi(\underline{\jjj})$ belongs to the same $\sigma_k(QA_{\underline{\iii}|_n})$-mesh cube as $Q\Pi(\underline{\iii})$. Clearly for every $\iii \in \Gamma_\I$ the set
\[\mathcal{P}_{k}(\iii):=\left\{\mathcal{Q}_{|\iii|,k}(\underline{\jjj}) \colon \underline{\jjj}\in [\iii]\right\}\]
is a Borel partition of the cylinder $[\iii]$. The cardinality of $\mathcal{P}_k(\iii)$ is precisely the number of distinct $\sigma_k(QA_\iii)$-mesh cubes which intersect the set $Q\Pi[\iii]$, and since $Q\Pi[\iii]=QT_\iii X\subseteq QT_\iii \mathbf{B}$ it follows from Lemma \ref{le:elementary-falconer} that  \begin{equation}\label{eq:falconer-bound}|\mathcal{P}_{\lceil s\rceil}(\iii)|\leq CL^d \left(\frac{\varphi^s(QA_\iii)}{\sigma_{\lceil s\rceil}(QA_\iii)^s}\right)\end{equation}
for all $\iii \in \Gamma_\I$ and $s \in (0,\rank Q]$. Let us write $\mathbf{B}_d(x,r)$ for the closed ball in $\R^d$ with centre $x$ and radius $r$. We observe that each  function $\underline{\iii}\mapsto \mu(\mathcal{Q}_{n,k}(\underline{\iii}))$ is Borel measurable, and that
\begin{equation}\label{eq:Qnk-inclusion}\mathcal{Q}_{n,\lceil s\rceil}(\underline{\iii})\subseteq (Q\Pi)^{-1} \mathbf{B}_d\left(Q\Pi\left(\underline{\iii}\right),\sqrt{d}\cdot\sigma_{\lceil s\rceil}\left(QA_{\underline{\iii}|_n}\right)\right)\end{equation}
for every $\underline{\iii} \in \I^\N$, $s \in (0,\rank Q]$ and $n \geq1$, since every $\sigma_{\lceil s\rceil}(QA_{\underline{\iii}|_n})$-mesh cube has diameter precisely $\sqrt{d}\cdot\sigma_{\lceil s\rceil}(QA_{\underline{\iii}|_n})$.

Since trivially
\[{\ess\sup}_{(Q\Pi)_*\mu,Q\Pi(\underline{\iii})} \dimlocu((Q\Pi)_*\mu, Q\Pi(\underline{\iii}))  \leq \dimp \supp (Q\Pi)_*\mu \leq \rank Q,\]
the set
\[\Omega_0:=\left\{\underline{\iii} \in \I^\N \colon \dimlocu((Q\Pi)_*\mu, Q\Pi(\underline{\iii})) \leq \rank Q\right\}\]
satisfies $\mu(\Omega_0)=1$. For every $k=1,\ldots,\rank Q$ define also
\[\Omega_{k}:=\left\{\underline{\iii}\in \I^\N \colon \mu(\mathcal{Q}_{n,k}(\underline{\iii})) > \frac{1}{n^2} \cdot  \frac{\mu([\underline{\iii}|_n])}{|\mathcal{P}_k(\underline{\iii}|_n)|} \text{ for all sufficiently large }n\right\}.\]
Since for every $n \geq 1$,
\begin{align*}\mu\left(\left\{\underline{\iii} \in \I^\N \colon \mu(\mathcal{Q}_{n,k}(\underline{\iii} )) \leq \frac{1}{n^2}\cdot\frac{\mu([\underline{\iii}|_n])}{|\mathcal{P}_k(\underline{\iii}|_n)|}    \right\}\right)&= \sum_{|\iii|=n} \sum_{\substack{Z \in \mathcal{P}_k(\iii) \\ \mu(Z) \leq \mu([\iii])/n^2|\mathcal{P}_k(\iii)|}} \mu(Z)\\
&\leq \sum_{|\iii|=n} \sum_{Z \in \mathcal{P}_k(\iii)} \frac{1}{n^2}\cdot \frac{\mu([\iii])}{|\mathcal{P}_k(\iii)|} \\
&= \sum_{|\iii|=n} \frac{1}{n^2} \cdot\mu([\iii]) = \frac{1}{n^2},\end{align*}
it follows by the Borel-Cantelli Lemma that $\mu(\Omega_{k})=1$ for each $k=1,\ldots,\rank Q$. Now define $\Omega:=\bigcap_{k=0}^{\rank Q}\Omega_k$,
which obviously has full measure. To prove the upper local dimension bound  claimed in Theorem \ref{th:dimension-results}\ref{it:dimension-upper-bounds} we will show that if $\underline{\iii} \in \Omega$ then 
\begin{equation}\label{eq:dimlocu}\dimlocu((Q\Pi)_*\mu, Q\Pi(\underline{\iii})) \leq \inf\left\{s \in (0,\rank Q]\colon \limsup_{n \to \infty} \frac{1}{n}\log\left(\frac{\varphi^s(QA_{\underline{\iii}|_n})}{\mu([\underline{\iii}|_n])}\right)<0\right\}.\end{equation}
Suppose then that $\underline{\iii} \in \Omega$, $s>0$,  and
\[ \limsup_{n \to \infty} \frac{1}{n}\log\left(\frac{\varphi^s(QA_{\underline{\iii}|_n})}{\mu([\underline{\iii}|_n])}\right)<0.\]
If $s >\rank Q$ then $\dimlocu((Q\Pi)_*\mu, Q\Pi(\underline{\iii})) \leq \rank Q <s$ since $\underline{\iii} \in \Omega_0$, so suppose instead that $s \in (0,\rank Q]$. 
Since $\underline{\iii} \in 
\Omega_{\lceil s\rceil}$ we have for all large enough $n$
\begin{align*}\log \mu(\mathcal{Q}_{n,\lceil s\rceil}(\underline{\iii})) &>-2\log n + \log \mu([\underline{\iii}|_n]) - \log |\mathcal{P}_{\lceil s\rceil}(\underline{\iii}|_n)|\\
&\geq -3\log n+\log \mu([\underline{\iii}|_n]) -\log \varphi^s(QA_{\underline{\iii}|_n})+s\log \sigma_{\lceil s\rceil}(QA_{\underline{\iii}|_n})\end{align*}
where we have used \eqref{eq:falconer-bound}. Since $\sigma_{\lceil s\rceil}(QA_{\underline{\iii}|_n})<1$ for all large enough $n$, it follows that
\[\frac{\log \mu(\mathcal{Q}_{n,\lceil s\rceil}(\underline{\iii} ))}{\log \sigma_{\lceil s\rceil}(QA_{\underline{\iii} |_n})} \leq s+ \frac{-3\log n +\log \mu([\underline{\iii} |_n])  -\log \varphi^s(QA_{\underline{\iii} |_n})}{\log \sigma_{\lceil s\rceil}(QA_{\underline{\iii} |_n})}\]
whenever $n$ is sufficiently large. For all large enough $n$ the numerator of the fraction on the right-hand side is positive and the denominator negative, and therefore
\[\limsup_{n \to \infty}\frac{\log \mu(\mathcal{Q}_{n,\lceil s\rceil}(\underline{\iii}|_n ))}{\log \sigma_{\lceil s\rceil}(QA_{\underline{\iii} |_n})}\leq s.\]
Since the ratio of any two successive terms of the sequence $n \mapsto \sqrt{d}\cdot \sigma_{\lceil s\rceil}(QA_{\underline{\iii} |_n})$ is bounded, we may apply this to calculate 
 \begin{align*}\dimlocu ((Q\Pi)_*\mu, Q\Pi(\underline{\iii} ))&=\limsup_{n \to \infty} \frac{\log((Q\Pi)_*\mu)(\mathbf{B}_d(Q\Pi(\underline{\iii} ), \sqrt{d}\cdot\sigma_{\lceil s\rceil}(QA_{\underline{\iii} |_n})))}{\log(\sqrt{d}\cdot \sigma_{\lceil s\rceil}(QA_{\underline{\iii} |_n}))}\\
&=\limsup_{n \to \infty}\frac{\log \mu((Q\Pi)^{-1}(\mathbf{B}_d(Q\Pi(\underline{\iii} ), \sqrt{d}\cdot \sigma_{\lceil s\rceil}(QA_{\underline{\iii} |_n}))))}{ \log (\sqrt{d}\cdot \sigma_{\lceil s\rceil}(QA_{\underline{\iii} |_n})) }\\
&\leq \limsup_{n \to \infty}\frac{\log\mu(\mathcal{Q}_{n,\lceil s\rceil}(\underline{\iii} ))}{\log(\sqrt{d}\cdot \sigma_{\lceil s\rceil}(QA_{\underline{\iii} |_n}))}\\
&= \limsup_{n \to \infty}\frac{\log\mu(\mathcal{Q}_{n,\lceil s\rceil}(\underline{\iii} ))}{\log \sigma_{\lceil s\rceil}(QA_{\underline{\iii} |_n})}\leq  s,\end{align*}
 where we have used \eqref{eq:Qnk-inclusion} and the eventual negativity of the denominator. We conclude that every $\underline{\iii} \in \Omega$ satisfies \eqref{eq:dimlocu} as required. By Theorem \ref{th:pressure-detail}\ref{it:thermodynamic} the limit superior in \eqref{eq:dimlocu} is almost everywhere a limit for all $s \geq 0$, so this completes the proof of 
Theorem \ref{th:dimension-results}\ref{it:dimension-upper-bounds}.
\end{proof}

\subsection{Lower bounds}
For every pair of distinct infinite words $\underline{\iii},\underline{\jjj}\in \I^\N$, define $\underline{\iii}\wedge \underline{\jjj}$ to be the maximal finite word which prefixes both $\underline{\iii}$ and $\underline{\jjj}$. If $i_1 \neq j_1$, $A_{\underline{\iii}\wedge \underline{\jjj}}$ should be understood as the identity map.  We will at times identify $(\R^d)^{\I}$ with $\R^{d|\I|}$ without comment. 

The next lemma follows from the arguments of \cite[\S10.4]{BaSiSo23} together with a simple extension of \cite[Lemma 2.2]{Fa88} to the case of non-invertible linear maps:
\begin{lemma}\label{le:falconer-solomyak}
Let $Q \in \End(\R^d)$,  $(A_i)_{i \in \I} \in \GL_d(\R)^{\I}$ and $s\in (0,\rank Q)\setminus \Z$.  Suppose that $\threebar{\cdot}$ is a norm on $\R^d$ such that
\[\max_{i,j \in \I \colon i \neq j} \threebar{A_i}+\threebar{A_j}<1.\]
For every $\mathbf{v}=(v_i)_{i \in \I} \in (\R^d)^\I$ define an affine iterated function system $(T_i^{\mathbf{v}})_{i \in \I}$  by $T_i^{\mathbf{v}}x \equiv A_ix+v_i$, and let $\Pi^{\mathbf{v}}$ denote the associated coding map. 
Then:
\begin{enumerate}[(i)]
    \item 
    For every $r>0$ there exists $C>0$ such that
\[\int_{\mathbf{B}_{d|\I|}(0,r)} \frac{d\mathbf{v}}{\|Q\Pi^{\mathbf{v}}(\underline{\iii})  - Q\Pi^{\mathbf{v}}(\underline{\jjj})\|^s} \leq \frac{C}{\varphi^s(QA_{\underline{\iii}\wedge\underline{\jjj}})}\]
for all distinct $\underline{\iii},\underline{\jjj} \in \I^\N$.
    \item For every $r>0$ there exists $C>0$ such that
\[\mathcal{L}^{d|\I|} \left(\left\{\mathbf{v}\in \mathbf{B}_{d|\I|}(0,r) \colon\! \|Q\Pi^{\mathbf{v}}(\underline{\iii})  - Q\Pi^{\mathbf{v}}(\underline{\jjj})\|\leq\delta\right\}\right) \leq C \prod_{\ell=1}^{\rank Q} \min\left\{1,  \frac{\delta}{\sigma_\ell(QA_{\underline{\iii}\wedge\underline{\jjj}})}\right\}\]
for all distinct $\underline{\iii},\underline{\jjj} \in \I^\N$, where $\mathcal{L}^k$ denotes $k$-dimensional Lebesgue measure.
\end{enumerate}

\end{lemma}
We will also require the following integral estimates for dimension and absolute continuity of measures. The second clause of the below lemma is classical; the first clause was recently applied to measures on self-affine sets by D.-J. Feng,  C.-H. Lo and C.-Y. Ma in \cite{FeLoMa22}, and is proved implicitly in earlier sources such as \cite[Theorem 3.5]{SaYo97} and \cite[\S3.4]{BiPe17}.
\begin{lemma}\label{le:dimlocl-s}
Let $\nu$ be a finite Borel measure on a real vector space $W$ of dimension $k$. Then:
\begin{enumerate}[(i)]
\item\label{it:absco1}
For every $x \in \supp \nu$,
\[\dimlocl(\nu,x) \geq \sup\left\{s \geq 0 \colon \int_W \|x-y\|^{-s}d\nu(y)<\infty \right\}.\]
\item \label{it:absco2}
If
\[\int_W \liminf_{\delta \to 0} \frac{\nu(\mathbf{B}_k(x,\delta))}{\delta^k}d\nu(x)<\infty,\]
then $\nu$ is absolutely continuous with respect to Lebesgue measure on $W$. 
\end{enumerate}
\end{lemma}
\begin{proof}
If the integral in \ref{it:absco1} is finite for some $s$ then $\nu(\mathbf{B}_d(x,r))=O(r^s)$ by Markov's inequality and the result follows. The second clause may be found in, for example, \cite[Theorem 2.12]{Ma95}.\end{proof}
\begin{proof}[Proof of Theorem \ref{th:dimension-results}\ref{it:dimension-lower-bounds}] We begin by proving the clause concerning local dimensions of measures. In view of the upper bound established in Theorem \ref{th:dimension-results}\ref{it:dimension-upper-bounds} we need only consider the lower local dimension. Given $r>0$, define $\Xi_r \subset {\mathbf{B}}_{d|\I|}(0,r) \times \I^\N$ to be the set of all $(\mathbf{v}, \underline{\iii})$ such that 
\begin{equation}\label{eq:the-inequality-we-like}\dimlocl((Q\Pi^{\mathbf{v}})_*\mu, Q\Pi^{\mathbf{v}}(\underline{\iii})) \geq \sup\left\{s\in [0,\rank Q] \colon \liminf_{n \to \infty} \frac{1}{n}\log \left(\frac{\varphi^s(QA_{\underline{\iii}|_n})}{\mu([\underline{\iii}|_n]) }\right) \geq 0 \right \}. \end{equation}
Standard arguments  demonstrate that $\Xi_r$ is Borel measurable, and by Theorem \ref{th:pressure-detail}\ref{it:thermodynamic}, for $\mu$-almost-every $\underline{\iii}$ the limit inferior is a limit for all $s$. To prove this clause of the theorem it is enough to show that for every $r>0$,
\[\{\mathbf{v} \in {\mathbf{B}}_{d|\I|}(0,r) \colon (\mathbf{v},\underline{\iii})\in \Xi_r\text{ for }\mu\text{-a.e. }\underline{\iii}\}\]
has full Lebesgue measure in $\mathbf{B}_{d|\I|}(0,r)$. By Fubini's theorem applied to the characteristic function of $\Xi_r$, this holds if and only if $\Xi_r$ itself has full measure with respect to ${\mathrm{Lebesgue}} \times \mu$, if and only if the set
\[\{\underline{\iii}\in  \I^\N \colon (\mathbf{v},\underline{\iii})\in \Xi_r\text{ for Lebesgue a.e. }\mathbf{v} \in \mathbf{B}_{d|\I|}(0,r)\}\]
has full measure with respect to $\mu$. We will show that this set precisely equals $\I^\N$. Indeed, given $\underline{\iii} \in \I^\N$ we observe that \eqref{eq:the-inequality-we-like} holds trivially for all $\mathbf{v}\in \mathbf{B}_{d|\I|}(0,r)$ if the right-hand side is zero. Otherwise, it suffices to show that for every non-integer rational number $s \in (0,\rank Q)$ satisfying
\begin{equation}\label{eq:choice-of-s}\liminf_{n \to \infty} \frac{1}{n}\log \left(\frac{\varphi^s(QA_{\underline{\iii}|_n})}{ \mu([\underline{\iii}|_n])}\right)>0,\end{equation}
the set
\[\left\{\mathbf{v}\in\mathbf{B}_{d|\I|}(0,r) \colon \dimlocl((Q\Pi^{\mathbf{v}})_*\mu, Q\Pi^{\mathbf{v}})(\underline{\iii})) \geq s\right\}\]
has full Lebesgue measure in $\mathbf{B}_{d|\I|}(0,r)$. By Lemma \ref{le:dimlocl-s} this holds if the integral 
\[\int_{\mathbf{B}_{d|\I|}(0,r)} \int_{\R^d} \|Q\Pi^{\mathbf{v}}(\underline{\iii})-y\|^{-s}d((Q\Pi^{\mathbf{v}})_*\mu)(y) d\mathbf{v}\]
is finite; but we have
\begin{align*}
{\lefteqn{\int_{\mathbf{B}_{d|\I|}(0,r)} \int_{\R^d} \|Q\Pi^{\mathbf{v}}(\underline{\iii})-y\|^{-s}d((Q\Pi^{\mathbf{v}})_*\mu)(y) d\mathbf{v}}}&\\
&=\int_{\mathbf{B}_{d|\I|}(0,r)}  \int_{\I^\N} \|Q\Pi^{\mathbf{v}}(\underline{\iii})-Q\Pi^{\mathbf{v}}(\underline{\jjj})\|^{-s}d\mu(\underline{\jjj}) d\mathbf{v}\\
&= \int_{\I^\N} \int_{\mathbf{B}_{d|\I|}(0,r)}\|Q\Pi^{\mathbf{v}}(\underline{\iii})-Q\Pi^{\mathbf{v}}(\underline{\jjj})\|^{-s}d\mathbf{v}d\mu(\underline{\jjj})\\
&\leq \int_{\I^\N} \frac{C}{\varphi^s(QA_{\underline{\iii}\wedge \underline{\jjj}})} d\mu(\underline{\jjj})
= C\sum_{n=0}^\infty  \frac{\mu([\underline{\iii}|_n])-\mu([\underline{\iii}|_{n+1}]) }{\varphi^s(QA_{\underline{\iii}|_n})}\leq C\sum_{n=0}^\infty  \frac{\mu([\underline{\iii}|_n])}{\varphi^s(QA_{\underline{\iii}|_n})}<\infty\end{align*}
where we have used Fubini's theorem, Lemma \ref{le:falconer-solomyak} and \eqref{eq:choice-of-s}. Since $r>0$ was arbitrary this completes the proof of the clause concerning dimensions of measures. 

The argument concerning the absolute continuity of $(Q\Pi^{\mathbf{v}})_*\mu_\Upsilon$ proceeds along similar lines. 
For each pair of integers $m,\ell \geq 1$ define
\[\Upsilon_{m,\ell}:=\left\{\underline{\iii}\in \I^\N \colon  \frac{1}{n}\log \left(\frac{\varphi^k(QA_{\underline{\iii}|_n})}{\mu([\underline{\iii}|_n])}\right)>\frac{1}{\ell} \text{ for all }n \geq m\right\}\]
and let $\mu_{\Upsilon_{m,\ell}}$ denote the restriction of $\mu$ to the set $\Upsilon_{m,\ell}$. It is enough to show that for every pair of integers $m$ and $\ell$, for Lebesgue a.e. $\mathbf{v}$ the measure $(Q\Pi^{\mathbf{v}})_*\mu_{\Upsilon_{m,\ell}}$ is absolutely continuous with respect to Lebesgue measure on $\im Q$. Defining $k:=\rank Q$, in view of Lemma \ref{le:dimlocl-s} it suffices to show that for every $r>0$
\[\int_{\mathbf{B}_{d|\I|}(0,r)}\int_{\im Q} \liminf_{\delta \to 0} \frac{(Q_*\Pi_*^{\mathbf{v}}\mu_{\Upsilon_{m,\ell}})(\mathbf{B}_k(x,\delta))}{\delta^k} d(Q_*\Pi_*^{\mathbf{v}}\mu_{\Upsilon_{m,\ell}})(x)d\mathbf{v}<\infty.\]
For fixed $m$, $\ell$ and $r$, using Fatou's lemma, Fubini's theorem, Lemma \ref{le:falconer-solomyak} and the definition of $\Upsilon_{m,\ell}$ yields the upper bound
\begin{align*}{\lefteqn{\liminf_{\delta \to 0} \int_{\mathbf{B}_{d|\I|}(0,r)}\int_{\im Q}  \frac{(Q_*\Pi_*^{\mathbf{v}}\mu_{\Upsilon_{m,\ell}})(\mathbf{B}_k(x,\delta)) }{\delta^k}d(Q_*\Pi_*^{\mathbf{v}}\mu_{\Upsilon_{m,\ell}})(x)d\mathbf{v}}}&\\
&=\liminf_{\delta \to 0} \int_{\mathbf{B}_{d|\I|}(0,r)}\int_{\I^\N} \frac{\mu_{\Upsilon_{m,\ell}}(\{\underline{\jjj} \in \I^\N \colon Q\Pi^{\mathbf{v}}(\underline{\jjj} )\in \mathbf{B}_k( Q\Pi^{\mathbf{v}}(\underline{\iii}),\delta)\}) }{\delta^k}d\mu_{\Upsilon_{m,\ell}}(\underline{\iii})d\mathbf{v}\\
&\leq\liminf_{\delta \to 0} \int_{\mathbf{B}_{d|\I|}(0,r)}\int_{\I^\N} \frac{\mu(\{\underline{\jjj}  \in \I^\N \colon Q\Pi^{\mathbf{v}}(\underline{\jjj} )\in \mathbf{B}_k( Q\Pi^{\mathbf{v}}(\underline{\iii}),\delta)\}) }{\delta^k}d\mu_{\Upsilon_{m,\ell}}(\underline{\iii})d\mathbf{v}\\
&=\liminf_{\delta \to 0}\int_{\I^\N}\int_{\I^\N}\frac{\mathcal{L}^{d|\I|}(\{\mathbf{v}\in \mathbf{B}_{d|\I|}(0,r) \colon \|Q\Pi^{\mathbf{v}}(\underline{\iii})-Q\Pi^{\mathbf{v}}(\underline{\jjj})\|\leq \delta\})}{\delta^k}  d\mu_{\Upsilon_{m,\ell}}(\underline{\iii})d\mu(\underline{\jjj})  \\
&\leq \liminf_{\delta \to 0}\int_{\I^\N}\int_{\I^\N}\frac{C}{\delta^k}\prod_{\ell=1}^k \min\left\{1,\frac{\delta}{\sigma_\ell(QA_{\underline{\iii}\wedge\underline{\jjj}})} \right\}d\mu_{\Upsilon_{m,\ell}}(\underline{\iii})d\mu(\underline{\jjj})  \\
&\leq C\int_{\I^\N}\int_{\I^\N} \frac{1}{\varphi^k(QA_{\underline{\iii}\wedge\underline{\jjj}})} d\mu_{\Upsilon_{m,\ell}}(\underline{\iii})d\mu(\underline{\jjj})  \leq C\int_{\I^\N}\sum_{n=0}^\infty  \frac{\mu([\underline{\iii}|_n])}{\varphi^k(QA_{\underline{\iii}|_n})}d\mu_{\Upsilon_{m,\ell}}(\underline{\iii})\\
&\leq C\sum_{n=0}^{m-1} \max_{|\iii|=n} \frac{1}{\varphi^k(QA_\iii)} + C\sum_{n=m}^\infty e^{-n/\ell} <\infty.\end{align*}
Since $m$, $\ell$ and $r$ were arbitrary, we conclude that for Lebesgue almost every $\mathbf{v} \in (\R^d)^\I$ the (not \emph{a priori} nonzero) finite measure $(Q\Pi^{\mathbf{v}})_*\mu_\Upsilon$ is absolutely continuous with respect to Lebesgue measure on $\im Q$.

It remains to consider the dimensions and measures of the sets $QX^{\mathbf{v}}$.
Define $s:=\dimaffQ \A \in [0,\rank Q]$. By Theorem \ref{th:pressure-detail}\ref{it:thermodynamic}\ref{it:eqm-at-least-one} there exists a $(\varphi^s,Q)$-equilibrium state $\mu \in \mathcal{M}(\I^\N)$ for $\A$, which by definition satisfies 
\[{\ess \sup}_{\mu, \underline{\iii}} \lim_{n \to \infty} \frac{1}{n}\log \left(\frac{\varphi^s(QA_{\underline{\iii}|_n})}{\mu([\underline{\iii}|_n])}\right) =P_Q(\A,s)\geq0.\]
Applying the lower bound for local dimensions of measures we conclude that for Lebesgue almost every $\mathbf{v} \in (\R^d)^{\I}$
\[{\ess \sup}_{(Q\Pi^{\mathbf{v}})_* \mu, x} \dimloc ((Q\Pi^{\mathbf{v}})_* \mu,x) \geq \dimaffQ \A\]
and therefore
\[\dimh QX^{\mathbf{v}} \geq \dim_{\mathsf{H}}^* (Q\Pi^{\mathbf{v}})_* \mu \geq \dimaffQ \A\]
as required. Similarly, if $P_Q(\A, \rank Q)>0$ then taking $\mu$ to be a $(\varphi^{\rank Q}, Q)$-equilibrium state for $\A$, we observe that the measure $\mu_\Upsilon$ is nonzero and therefore $(Q\Pi^{\mathbf{v}})_*\mu_\Upsilon$ is a nonzero absolutely continuous measure for Lebesgue almost every $\mathbf{v} \in (\R^d)^{\I}$, hence $QX^{\mathbf{v}}$ has positive Lebesgue measure in $\im Q$ for Lebesgue almost every $\mathbf{v} \in (\R^d)^{\I}$. Combining these results with the upper bound given by Theorem \ref{th:dimension-results}\ref{it:dimension-upper-bounds} completes the proof of the theorem.
\end{proof}
\section{Construction of examples: proof of Theorem \ref{th:general-examples} and Theorem \ref{intro:sumset}}\label{se:general-examples}

\subsection{Proof of Theorem \ref{th:general-examples}}
Applying a theorem of Mostow \cite{mostow} we equip $\R^d$ with an inner product structure with respect to which $G$ is self-adjoint, and we equip $\wedge^k\R^d$ with the corresponding induced inner product. A linear subspace of $\wedge^k\R^d$ is thus $G$-invariant (respectively $G^o$-invariant) if and only if its orthogonal complement is.

\subsubsection{Proof of Theorem \ref{th:general-examples}\ref{it:split}} We claim that $\wedge^k\R^d$ may be written as a direct sum of two $G$-invariant subspaces each of which contains a nonzero pure $k$-wedge and which are additionally pairwise orthogonal. Let $F_1$ and $F_2$ be two transverse $G$-invariant subspaces that contain pure wedges. Since $\A$ is $k$-dominated, we can choose an element $g \in G$ such that $g^{\wedge k}$ is proximal. Let $E_1$ be the smallest $G$-invariant subspace of $\wedge^k \R^d$ which contains the leading eigenspace of $g^{\wedge k}$;  obviously $E_1$ is $G$-irreducible. Since $g^{\wedge k}$ is proximal, the multiplicity of $E_1$ must be $1$, i.e.\ no other irreducible subspace of $\wedge^k \R^d$ can be isomorphic to $E_1$ as an $\R[G]$-module. Since $E_1^\perp$ is $G$-invariant and $E_1$ has multiplicity one, every $G$-invariant subspace which does not contain $E_1$ must be contained in $E_1^\perp$. Hence either $F_1$ or $F_2$ must be contained in $E_1^\perp$, and the claim is proved. 

We now fix a $G$-invariant orthogonal splitting $\wedge^k \R^d=E_1\oplus E_2$ such that each of $E_1$ and $E_2$ contains a pure wedge.  Using $k$-domination, choose $n_0 \geq 1$ such that $(A_\iii^T A_\iii)^{\wedge k}$ has a simple leading eigenspace for all words $\iii$ of length at least $n_0$. Clearly this leading eigenspace is always contained in either $E_1$ or $E_2$, and a simple argument by contradiction shows that there exists $n_1 \geq n_0$ such that the simple leading eigenspace of $(A_\iii^T A_\iii)^{\wedge k}$ is contained in the same subspace $E_1$, say, for all words $\iii$ of length at least $n_1$. Hence for all long enough words $\iii$,
\[\left\|A_\iii^{\wedge k}\right\|_{E_1}=\sigma_1\left(A_\iii^{\wedge k}\right)=\prod_{j=1}^k \sigma_j(A_\iii)\]
and
\[\|A_\iii^{\wedge k}\|_{E_2}\leq\sigma_2\left(A_\iii^{\wedge k}\right)= \left(\frac{\sigma_{k+1}(A_\iii)}{\sigma_k(A_\iii)}\right)\prod_{j=1}^k \sigma_j(A_\iii). \]
Now let $Q \in \End(\R^d)$ be a rank-$k$ orthogonal projection whose image subspace corresponds to a wedge in $E_2$. Since $Q^{\wedge k}$ is also an orthogonal projection its kernel includes the orthogonal complement of $E_2$, which is $E_1$. Hence for some real number $\kappa>0$,
 \[\frac{\left\|\left(QA_{\iii}\right)^{\wedge k}\right\|} 
{\left\|A_{\iii}^{\wedge k}\right\|}\leq\frac{\left\|A_{\iii}^{\wedge k}\right\|_{E_2}}{\left\|A_{\iii}^{\wedge k}\right\|_{E_1}} \leq \frac{\sigma_{k+1}(A_\iii)}{\sigma_k(A_\iii)} \leq e^{-\kappa|\iii|}\]
for all long enough words $\iii$. The inequality
\[P_Q(\A,s) \leq P(\A,s)-\kappa(s-k+1)<P(\A,s)\]
for all $s \in (k-1,k]$ follows directly.

\subsubsection{Proof of Theorem \ref{th:general-examples}\ref{it:no-split}} Let $Q \in \End(\R^d)$ have rank $k$. By Theorem \ref{th:pressure-detail}\ref{it:monotone-ker-q} we may without loss of generality suppose that $Q$ is an orthogonal projection, in which case $Q^{\wedge k} \in \End(\wedge^k\R^d)$ is an orthogonal projection of rank one.

We claim that there exist $m\geq 1$ and $\kappa_1>0$ such that for every $v_1\wedge \cdots \wedge v_k$ in the unit sphere of $ \wedge^k \R^d$,
\[ \max_{|\jjj|\leq m}\left\|(QA_\jjj)^{\wedge k} (v_1\wedge\cdots \wedge v_k)\right\| \geq \kappa_1.\]
Suppose for a contradiction that the claim is false. By a simple compactness argument there exists $v_1\wedge \cdots \wedge v_k \in \wedge^k \R^d$ with norm $1$ such that $(QA_\jjj)^{\wedge k}(v_1\wedge \cdots \wedge v_k)=0$ for all $\jjj \in \Gamma_\I$. If this is the case then the linear subspace
\[\Span \left\{A_\jjj^{\wedge k} (v_1\wedge\cdots \wedge v_k) \colon \jjj \in \Gamma_\I\right\}\]
is a $G$-invariant subspace of $\ker Q^{\wedge k}$ which contains a pure wedge. On the other hand, the orthogonal complement of this space is also $G$-invariant and contains $\im Q^{\wedge k}$, which is a one-dimensional subspace spanned by a pure wedge. This contradicts the hypothesis, and the claim follows.

We next claim that there exists $\kappa_2> 0$ such that for every $s \in [0, k]$ and $\iii \in \Gamma_\I$,
\begin{equation*}
\max_{|\jjj|\leq m} \varphi^s\left(QA_\jjj A_\iii\right) \geq \kappa_2 \varphi^s(A_\iii).
\end{equation*}
 Indeed, given $\iii \in \Gamma_\I$, choose a pure wedge $u=u_1\wedge\cdots \wedge u_k \in \wedge^k \R^d$ of norm $1$ such that $\|A_\iii^{\wedge k}u\|=\|A_\iii^{\wedge k}\|$ and define $v:=\|A_\iii^{\wedge k}\|^{-1} A_\iii^{\wedge k}u$, which is of course also a pure wedge of norm $1$. By the definition of $\kappa_1$ there exists a word $\jjj_0$ of length at most $m$ such that $\|(QA_{\jjj_0})^{\wedge k}v\| \geq \kappa_1$. 
Now for every non-negative integer $\ell \leq k$,
\begin{align*}\kappa_1 \leq \left\|(QA_{\jjj_0})^{\wedge k}v\right\| &=  \frac{\|(QA_{\jjj_0} A_\iii)^{\wedge k}u\|}{\|A_\iii^{\wedge k}\|}\\
&\leq  \frac{\|(QA_{\jjj_0} A_\iii)^{\wedge k}\|}{\|A_\iii^{\wedge k}\|}\\
&=\frac{\left(\prod_{j=1}^\ell
\sigma_j(QA_{\jjj_0} A_\iii)\right)\left(\prod_{j=\ell+1}^k \sigma_j(QA_{\jjj_0} A_\iii)\right)}{\prod_{j=1}^k \sigma_j(A_\iii)}\\
&\leq \frac{\|A_{\jjj_0}\|^{k-\ell} \left(\prod_{j=1}^\ell \sigma_j(QA_{\jjj_0} A_\iii)\right)\left(\prod_{j=\ell+1}^k \sigma_j(A_\iii)\right)}{\prod_{j=1}^k \sigma_j(A_\iii)}\\
&=\frac{\|A_{\jjj_0}\|^{k-\ell}\cdot  \prod_{j=1}^\ell\sigma_j(QA_{\jjj_0} A_\iii)}{\prod_{j=1}^\ell \sigma_j( A_\iii)}=\frac{\|A_{\jjj_0}\|^{k-\ell}  \|(QA_{\jjj_0} A_\iii)^{\wedge \ell}\|}{\|A_\iii^{\wedge \ell}\|}\end{align*}
so that
\[\left\|(QA_{\jjj_0} A_\iii)^{\wedge \ell}\right\|
\geq \kappa_1 \|A_{\jjj_0}\|^{\ell-k} \left\|A_\iii^{\wedge \ell}\right\|
\geq \kappa_1 \left(\min_{\substack{|\kkk|\leq m\\ 0 \leq r \leq k}} \|A_\kkk\|^{r-k}\right) \left\|A_\iii^{\wedge \ell}\right\| = \kappa_2 \left\|A_\iii^{\wedge \ell}\right\|,\]
say, where $\kappa_2$ does not depend on $\ell$ or $\iii$. It follows directly that 
\[\max_{|\jjj|\leq m} \varphi^s\left(QA_{\jjj}A_\iii\right)\geq \varphi^s\left(QA_{\jjj_0}A_\iii\right) \geq \kappa_2 \varphi^s(A_\iii)\]
for every $s \in [0,k]$, and since $\iii$ was arbitrary this proves the claim. 

We may now complete the proof. For every $n\geq 1$ and $s \in [0,k]$,
\[\sum_{\ell=1}^m \sum_{|\iii|=n+\ell} \varphi^s(QA_\iii) = \sum_{|\iii|=n} \sum_{1 \leq |\jjj|\leq m}\varphi^s(QA_\jjj A_\iii)
 \geq \kappa_2 \sum_{|\iii|=n} \varphi^s(A_\iii)\]
and therefore
\begin{align*}P_Q(\A,s) &=\lim_{n \to \infty} \frac{1}{n} \log \sum_{\ell=1}^{m} \sum_{|\iii|=n+\ell} \varphi^s(QA_\iii)\geq \lim_{n \to \infty} \frac{1}{n} \log 
 \sum_{|\iii|=n}\kappa_2 \varphi^s(A_\iii) =P(\A,s) \end{align*}
 for every $s \in [0,k]$. This proves \ref{it:no-split} since we also have $P(\A,s)\geq P_Q(\A,s)$. 

\subsubsection{Proof of Theorem \ref{th:general-examples}\ref{it:bad-measure-projection}} By the one-sided version of the Oseledets multiplicative ergodic theorem there exist an invariant Borel set $\Lambda \subseteq \I^\N$, real numbers $\lambda_1 \geq \cdots \geq \lambda_k>\lambda_{k+1}\geq \cdots \geq\lambda_d$ and a measurable function $\mathfrak{v}$ from $\Lambda$ to the $1$-codimensional Grassmannian of $\wedge^k \R^d$ with the following properties. For all $j=1,\ldots,d$ and $\underline{\iii} \in \Lambda$,
\begin{equation}
    \label{eq:LE-exist}
\lambda_j=\lim_{n \to \infty} \frac{1}{n}\log \sigma_j\left(A_{\underline{\iii}|_n}^T\right).\end{equation}
For all $\underline{\iii} \in \Lambda$,
\[\lim_{n \to \infty} \frac{1}{n}\log \left\|\left(A_{\underline{\iii}|_n}^T\right)^{\wedge k} v\right\|\leq \sum_{j=1}^{k-1}\lambda_j + \lambda_{k+1}\]
if $v \in \mathfrak{v}(\underline{\iii})$, and
\[\lim_{n \to \infty} \frac{1}{n}\log \left\|\left(A_{\underline{\iii}|_n}^T\right)^{\wedge k} v\right\|= \sum_{j=1}^{k}\lambda_j\]
otherwise. Finally, for all $\underline{\iii} \in \Lambda$ and $n \geq 1$,
\[\left(A_{\underline{\iii}|_n}^T\right)^{\wedge k} \mathfrak{v}(\underline{\iii})=\mathfrak{v}\left(\sigma^n \underline{\iii}\right).\]
It follows from later proofs of  Oseledets' theorem (e.g.\ \cite[Theorem 1.6]{ruelle.oseledets}) that the one-dimensional subspace $\mathfrak{v}(\underline{\iii})^\perp$ is the limit as $n \to \infty$ of the one-dimensional top eigenspaces of the eventually proximal sequence of linear maps $(A_{\underline{\iii}|_n}A_{\underline{\iii}|_n}^T)^{\wedge k}$.

Write $\wedge^k \R^d = E_1\oplus E_2$ where each of $E_1$ and $E_2$ is $G^o$-invariant and non-trivial. Since the splitting $E_1 \oplus E_2$ is preserved by each $A_{\underline{\iii}|_n}A_{\underline{\iii}|_n}^T$, the subspace $\mathfrak{v}(\underline{\iii})^\perp$ is always contained in either $E_1$ or $E_2$. By relabelling the two spaces if necessary, we suppose without loss of generality that $\mathfrak{v}(\underline{\iii})^\perp \subseteq E_1$ on a set of positive measure. By replacing $E_1$ with a smaller subspace and $E_2$ with the orthogonal complement of the smaller subspace, we may also assume without loss of generality that there is no proper $G^o$-invariant subspace of $E_1$ which contains $\mathfrak{v}(\underline{\iii})^\perp$ on a set of positive measure, and also that the spaces $E_1$ and $E_2$ are mutually orthogonal. 

We claim that additionally $\mathfrak{v}(\underline{\iii})^\perp \subseteq E_2$ on a set of positive measure. Suppose for a contradiction that this is false, in which case we must have $\mathfrak{v}(\underline{\iii})^\perp \subseteq E_1$ for all $\underline{\iii}$ belonging to some full-measure invariant set $\Lambda'\subseteq \Lambda$. Since by hypothesis $E_1$ cannot be $G$-invariant we may choose $\kkk \in \Gamma_\I$ of length $n$, say, such that 
$A_\kkk^{\wedge k}E_1 \neq E_1$. Now, for each $\underline{\iii} \in [\kkk]\cap \Lambda'$ we have $(A_{\underline{\iii}|_n}^T)^{\wedge k}\mathfrak{v}(\underline{\iii}) = \mathfrak{v}(\sigma^n\underline{\iii})$ and therefore $A_{\underline{\iii}|_n}^{\wedge k}\mathfrak{v}(\sigma^n\underline{\iii})^\perp = \mathfrak{v}(\underline{\iii})^\perp$, with both $\mathfrak{v}(\sigma^n\underline{\iii})^\perp$ and  $\mathfrak{v}(\underline{\iii})^\perp$ being contained in $E_1$. Thus $\mathfrak{v}(\underline{\iii})^\perp \subseteq E_1 \cap A_\kkk^{\wedge k} E_1$ for a.e.\ $\underline{\iii} \in [\kkk]$. Since  $\mu$ is fully supported, the $G^o$-invariant proper subspace $A_\kkk^{\wedge k}E_1 \cap E_1$ of $E_1$ contains $\mathfrak{v}(\underline{\iii})^\perp$ on a set of positive measure, contradicting the minimality of $E_1$. The claim is proved. 

Now choose a pure wedge $u=u_1\wedge \cdots \wedge u_k \in E_1$ such that the one-dimensional subspace spanned by $u$ lies in the support of the distribution of $\mathfrak{v}^\perp$ on $\mathbf{P}(\wedge^k \R^d)$. Let $Q \in \End(\R^d)$ be the rank-$k$ orthogonal projection whose image is spanned by $u_1,\ldots,u_k$. By construction there is a positive-measure set of $\underline{\iii} \in \Lambda$ such that the image of $Q^{\wedge k}$ is transverse to $\mathfrak{v}(\underline{\iii})$, and as a consequence of the preceding claim there is also a positive-measure set of $\underline{\iii} \in \Lambda$ such that $\im Q^{\wedge k}$ is a subspace of $\mathfrak{v}(\underline{\iii})$.
Now, for all $\underline{\iii} \in \Lambda$ such that $\im Q^{\wedge k}$ is transverse to $\mathfrak{v}(\underline{\iii})$, we have
\begin{align*}\lim_{n \to \infty} \frac{1}{n}\log \left\|\left(QA_{\underline{\iii}|_n}\right)^{\wedge k} \right\|=\lim_{n \to \infty} \frac{1}{n}\log \left\|\left(A_{\underline{\iii}|_n}^TQ\right)^{\wedge k} \right\|= \lim_{n \to \infty} \frac{1}{n}\log \left\|\left(A_{\underline{\iii}|_n}^T\right)^{\wedge k} u\right\|\end{align*}
which is equal to $\sum_{j=1}^k \lambda_j$ by the defining properties of $Q$ and $\mathfrak{v}$. Combining this result with \eqref{eq:LE-exist} yields
\[\lim_{n \to \infty} \frac{1}{n}\log \left\|\left(QA_{\underline{\iii}|_n}\right)^{\wedge \ell} \right\|=\sum_{j=1}^\ell \lambda_j\]
for every $\ell=1,\ldots,k$, and we conclude that for all such $\underline{\iii} \in \Lambda$
\[\lim_{n \to \infty}\frac{1}{n}\log \varphi^s(QA_{\underline{\iii}|_n})= \lim_{n \to \infty}\frac{1}{n}\log \varphi^s(A_{\underline{\iii}|_n})\]
for all $s \in [0,k]$. On the other hand, for all $\underline{\iii}\in \Lambda$ such that $\im Q^{\wedge k} \subseteq\mathfrak{v}(\underline{\iii})$,
\[\lim_{n \to \infty} \frac{1}{n}\log \left\|\left(QA_{\underline{\iii}|_n}\right)^{\wedge k} \right\|=\lim_{n \to \infty} \frac{1}{n}\log \left\|\left(A_{\underline{\iii}|_n}^TQ\right)^{\wedge k} \right\|\leq \sum_{j=1}^{k-1}\lambda_j +\lambda_{k+1}<\sum_{j=1}^k \lambda_j\]
and it follows easily that for all such $\underline{\iii}\in \Lambda$
\[\lim_{n \to \infty}\frac{1}{n}\log \varphi^s(QA_{\underline{\iii}|_n})< \lim_{n \to \infty}\frac{1}{n}\log \varphi^s(A_{\underline{\iii}|_n})\]
for all $s \in (k-1,k]$. Since each situation occurs on a set of positive measure, the theorem is proved.

\subsection{Proof of Theorem \ref{intro:sumset}}
Suppose that we are given tuples $(A_i')_{i \in \I} \in \GL_{d_1}(\R)^\I$, $(B_j')_{j \in \J} \in \GL_{d_1}(\R)^\J$,  $(A_i'')_{i \in \I}\in \GL_{d_2}(\R)^\I$, $(B_j'')_{j \in \J} \in \GL_{d_2}(\R)^\J$ 
such that if we define $A_i:=A_i'\otimes A_i''$, $B_j:=B_j'\otimes B_j''$, $\A:=(A_i)_{i \in \I}$, $\B:=(B_j)_{j \in \J}$ and $\A\oplus \B:=(A_i\oplus B_j)_{(i,j) \in \I\times \J}$, then there hold the contraction properties
\[\max_{i\in \I} \|A_i\|<\frac{1}{2},\qquad \max_{j\in\J} \|B_j\|<\frac{1}{2},\]
and the dimension bound
\[s:=\dimaff\A+\dimaff \B\in(d_1d_2-1,d_1d_2)  ,\]
and such that additionally for some $k_1, k_2$ such that $1 <k_i\leq d_i$ for $i=1,2$,
\begin{equation}\label{eq:domination-liminf1}\liminf_{n \to \infty}\min_{|\iii|=n}\frac{1}{n}\log \left(\frac{\sigma_{k_1-1}(A_\iii')}{\sigma_{k_1}(A_\iii')}\right) > \frac{1}{s-(d_1d_2-1)} P(\A\oplus \B, d_1d_2-1),\end{equation}
\begin{equation}\label{eq:domination-liminf2}
  \liminf_{n \to \infty}\min_{|\jjj|=n}\frac{1}{n}\log \left(\frac{\sigma_{k_2-1}(B_\jjj'')}{\sigma_{k_2}(B_\jjj'')}\right) > \frac{1}{s-(d_1d_2-1)} P(\A\oplus \B, d_1d_2-1).\end{equation}
We observe that these conditions are open: for the contraction property this is trivial, for the dimension bound it follows from the main result of Feng and Shmerkin \cite{FeSh14}, and for \eqref{eq:domination-liminf1}--\eqref{eq:domination-liminf2} this follows from the domination theorems of Bochi and Gourmelon \cite{BoGo09} together with the continuity of the pressure as established in \cite{FeSh14}. This set of conditions is easily seen to be nonempty for large enough $\I, \J$ by considering cases in which $\A$ and $\B$ are constant tuples each given by a suitable diagonal matrix whose diagonal entries take exactly two distinct values.  
Let $(T_i)_{i \in \I}$, $(T_j')_{j \in \J}$ be iterated function systems with respective linearisations $\A$, $\B$ and with attractors $X$, $Y$ having Hausdorff dimension equal to $\dimaff \A$ and $\dimaff \B$ respectively. Let $Q \in \End(\R^{d_1d_2} \oplus \R^{d_1d_2})$ be given by $Q(u\oplus v):=(u+v)\oplus 0$. Clearly $X+Y$ is isometric to the image $Q(X\times Y)$ of the self-affine set $X\times Y$ which is the attractor of the iterated function system $(T_i\oplus T_j')_{(i,j)\in \I\times \J}$, so in view of Theorem \ref{th:pressure-detail}\ref{it:lipschitz-convexity} and Theorem \ref{th:dimension-results}\ref{it:dimension-upper-bounds}, to obtain the desired dimension estimate on $X+Y$ it is sufficient to show that $P_Q(\A\oplus \B,s)< 0$. 
Let $n$ be large enough that the inequalities implied by \eqref{eq:domination-liminf1} and \eqref{eq:domination-liminf2} are realised, and suppose that $|\iii|=|\jjj|=n$. Clearly
\begin{equation}\label{eq:truncate}\varphi^s(Q(A_\iii \oplus B_\jjj)) \leq 2^{\frac{d_1d_2-1}{2}} \varphi^{d_1d_2-1}(A_\iii \oplus B_\jjj) \sigma_{d_1d_2}(Q(A_\iii \oplus B_\jjj))^{s-(d_1d_2-1)}\end{equation}
and to derive the desired bound we estimate the final term in this expression. Taking $u \in \R^{d_1}, v \in \R^{d_2}$ to be unit vectors such that
\[\left\|A_\iii' (A_\iii')^Tu\right\|=\sigma_{k_1}\left(A_\iii'\right)^2, \qquad \left\|B_\jjj'' (B_\jjj'')^T v\right\|=\sigma_{k_2}\left(B_\jjj''\right)^2,\]
we may bound
\begin{align*}&\sigma_{d_1d_2}(Q(A_\iii \oplus B_\jjj))^2 
= \sigma_{d_1d_2}\left( \begin{pmatrix}A_\iii &B_\jjj\\ 0&0\end{pmatrix}\right)^2
=\sigma_{d_1d_2}\left(\begin{pmatrix}A_\iii &B_\jjj\\ 0&0\end{pmatrix}\begin{pmatrix}A_\iii &B_\jjj\\ 0&0\end{pmatrix}^T\right)\\
&=\sigma_{d_1d_2}\left(A_\iii A_\iii^T + B_\jjj B_\jjj^T\right)\\
&=\sigma_{d_1d_2}\left((A_\iii' \otimes A_\iii'')(A_\iii' \otimes A_\iii'')^T + (B_\jjj' \otimes B_\jjj'')(B_\jjj' \otimes B_\jjj'')^T\right)\\
&\leq \left\|\left((A_\iii' \otimes A_\iii'')(A_\iii' \otimes A_\iii'')^T + (B_\jjj' \otimes B_\jjj'')(B_\jjj'\otimes B_\jjj'')^T\right)(u\otimes v)\right\|  \\
&\leq \sigma_{k_1}\left(A_\iii'\right)^2\sigma_1\left(A_\iii''\right)^2 + \sigma_1\left(B_\jjj'\right)^2\sigma_{k_2}\left(B_\jjj''\right)^2\\
&\leq \exp\left(-\frac{2nP(\A\oplus\B,d_1d_2-1)}{s-(d_1d_2-1)}\right) \left(\sigma_{k_1-1}\left(A_\iii'\right)^2\sigma_1\left(A_\iii''\right)^2 + \sigma_1\left(B_\jjj'\right)^2\sigma_{k_2-1}\left(B_\jjj''\right)^2\right)\\
&<2^{1-4n}\exp\left(-\frac{2nP(\A\oplus\B,d_1d_2-1)}{s-(d_1d_2-1)}\right)\end{align*}
so that
\[ \sigma_{d_1d_2}\left(Q(A_\iii \oplus B_\jjj)\right)^{s-(d_1d_2-1)}<2^{\frac{(s-(d_1d_2-1))(1-4n)}{2}}\exp\left(-nP(\A\oplus \B, d_1d_2-1)\right).\]
Combining this expression with \eqref{eq:truncate} it follows that for all large enough $n$
\[\sum_{|\iii|=|\jjj|=n} \varphi^s(Q(A_\iii \oplus B_\jjj)) \leq 2^{\frac{s}{2}}\left(\frac{\sum_{|\iii|=|\jjj|=n} \varphi^{d_1d_2-1}(A_\iii \oplus B_\jjj)}{4^{n(s-(d_1d_2-1))}e^{nP(\A\oplus \B, d_1d_2-1)}}\right) \]
and $P_Q(\A\oplus\B,s)<0$ follows directly.

\appendix
\section{Guide to notation}\label{ap:notation}
For the reader's convenience, a summary of key notation is given below.


\begin{small}
\begin{adjustwidth}{-0.8cm}{-0.05cm}
\renewcommand{\arraystretch}{1.1}


\begin{longtable}{p{0.23\textwidth} p{0.72\textwidth}}
\caption{Index of notation}\\
\hline\\

$W,U,V$ & finite-dimensional real vector spaces (sometimes considered as $\R[G]$-modules with respect to a linear algebraic group $G$). \\

$\End(W)$ & vector space of linear maps $W \to W$. \symref{notation:EndW}  \\

$\GL(W)$;  $\GL_d(\R)$ & group of invertible linear maps $W \to W$; group of invertible $d\times d$ real matrices. \symref{notation:GLW} \\

$\mathbf{P}(W)$ &  the projective space of lines in $W$. \symref{notation:projW} \\

$\mathbf{S}_W$
& the unit sphere of the vector space $W$. \symref{notation:sphere} \\

$\mathbf{B}_d(x,r)$
& closed Euclidean ball with centre $x$ and radius $r$ in dimension $d$. \symref{notation:ball} \\

$\mathcal{I}$ & a non-empty finite set. \symref{notation:I} \\

$\Gamma_{\mathcal{I}}$ & set of all nonempty finite words over $\mathcal{I}$, sometimes viewed as a semigroup under concatenation. \symref{notation:Gamma_I}\\

$\iii,\jjj,\kkk$ & finite words over $\mathcal{I}$. \symref{notation:ijk} \\

$|\iii|$ & length of $\iii \in \Gamma_{\mathcal{I}}$.  \symref{notation:|i|}\\

$\underline{\iii}$ & an infinite word over $\mathcal{I}$, i.e. an element of $\I^\N$. \symref{notation:underlinei} \\

$\underline{\iii}|_n$ & length-$n$ prefix of $\underline{\iii}$. \symref{notation:i_n} \\

$[\iii]$ & cylinder set
$\{\underline{\jjj}\in\I^\N:\underline{\jjj}|_{|\iii|}=\iii\}$. \symref{notation:icylnd}\\

$\sigma_1(A)\geq\cdots\geq\sigma_d(A)$
& singular values of $A\in\End(\R^d)$ in decreasing order. \symref{notation:sing.values}\\

$\varphi^s$ & $s$-singular value potential $\End(\R^d)\to\R$, with $s\in[0,d]$. \symref{notation:varphis} \\

$A^{\wedge k}$ & $k$th exterior power of an endomorphism $A$ of $\R^d$. \symref{notation:Awedgek} \\

$\A$ & tuple $(A_i)_{i\in\I}\in\GL_d(\R)^\I$. \symref{notation:A} \\

$P(\A,s)$ & pressure of $\A$ at $s$. \symref{notation:P(A,s)} \\

$\mathcal{M}_\sigma(\I^\N)$
& set of shift-invariant measures on $\I^\N$. \symref{notation:M_sigma} \\

$h(\mu)$ & metric entropy of $\mu\in\mathcal{M}_\sigma(\I^\N)$. \symref{notation:hmu}\\

$\dimlyap(\A,\mu)$
& Lyapunov dimension of $\mu$ with respect to $\A$. \symref{notation:dimlyapAmu}\\

$Q$ or $Q_U$ & an endomorphism of $\R^d$ or orthogonal projection onto $U$. \symref{notation:Q} \\

$P_Q(\A,s)$ & $Q$-projected pressure of $\A$ at $s$. \symref{notation:PQ(A,s)} \\

$\dimaffQ \A$ & $Q$-projected affinity dimension of $\A$. \symref{notation:dimaffQ} \\

$\dimlyapQ(\A,\mu)(\underline{\iii})$
& $Q$-projected local Lyapunov dimension. \symref{notation:dimlyapQ}\\

$\dimh X$ & Hausdorff dimension of $X$. \symref{notation:dimhX} \\

$\dimbu X$ & upper box dimension of $X$. \symref{notation:dimbu} \\

$\Psi$ & potential $\Gamma_\mathcal{I}\to(0,\infty)$. \symref{notation:Psi} \\

$P(\Psi)$ & pressure of the potential $\Psi$. \symref{notation:P(Psi)} \\

$G$ & a linear algebraic group. \symref{notation:G} \\

$G^o$ & Zariski-connected component of $G$. \symref{notation:Go} \\

$\Gamma$ & a subsemigroup of $G$. \symref{notation:Gamma} \\

$\Gamma^o$ & the subsemigroup $\rho^{-1}(G^o) \cap \Gamma$ of $\Gamma$, relative to a fixed representation $\rho \colon \Gamma \to G\leq \GL(W)$. \symref{notation:Gammao} \\ 

$H \leq G$ & $H $ is a subgroup of $G$.  \\

$H \unlhd G$ & $H $ is a normal subgroup of $G$. \symref{notation:normal} \\

$\rho$ & finite-dimensional representation into $\GL(W)$. \symref{notation:rho} \\

$\R[G^o]$ & group ring of $G^o$. \symref{notation:grp.alg} \\

$V \leq W$& $V$ is a submodule of $W$. \symref{notation:submodule}\\

$U_j>U_j',\hat U_j>\hat U_j'$
& $\R[G^o]$-submodules of an $\R[G^o]$-module $V_j$. \symref{notation:Ujs} \\

$\mathbf{U}$
& an $r$-tuple of pairs $U_j>U_j'$ of submodules of the $\R[G^o]$-modules $V_j$. \symref{notation:bfU} \\

$\mathfrak{U}$
& set of all such tuples $\mathbf{U}$ relative to a fixed tuple $(V_j)_{j=1}^r$. \symref{notation:frakU} \\

$\mathfrak{U}_S$
& set of all $r$-tuples $\mathbf{U}\in \mathfrak{U}$ such that every quotient $U_j/U_j'$ is a simple $\R[G^o]$-module. \symref{notation:frakUS} \\

$\mathbf{b}$
& an $r$-tuple of non-negative real numbers $\beta_j$. \symref{notation:bfb} \\

$\Psi_{\mathbf{U},\mathbf{b}}$
& a potential $\Gamma_\I \to (0,\infty)$ defined in terms of a tuple of group elements $(g_i)_{i \in \I} \in G^\I$ together with an $r$-tuple of representations $\rho_j \colon G \to \GL(V_j)$ and $r$-tuples $\mathbf{U}\in \mathfrak{U}$ and $\mathbf{b}$ as above. \symref{notation:PsiUb} \\

$Z_j$ & maximal $\R[G^o]$-submodule contained in $\ker Q_j$. \symref{notation:Zj} \\

$\mathfrak r(\underline{\iii},n)$
& $n^{th}$ return time to $G^o$ of the product of $(g_i)_{ \in \I}$ along $\underline{\iii}$. \symref{notation:return} \\

$\spectralradius(B)$
& spectral radius of an endomorphism $B$. \symref{notation:specrad} \\

$\mathbf{v}$
& tuple $(v_i)_{i\in\I}\in(\R^d)^\I$. \symref{notation:bfv} \\

$T_i^{\mathbf v}$
& affine map $x\mapsto A_i x+v_i$. \symref{notation:T_iv} \\

$X^{\mathbf v}$
& attractor of $(T_i^{\mathbf v})_{i\in\I}$. \symref{notation:Xv} \\

$\Pi^{\mathbf v}$
& coding map $\I^\N\to X^{\mathbf v}$. \symref{notation:Piv} \\

$\mathcal N_Z(\delta)$
& number of $\delta$-mesh cubes intersecting $Z\subset\R^d$. \symref{notation:NZdelta} \\

$\|\cdot\|_W$
& Euclidean norm on $W$, or induced operator norm. \symref{notation:Wnorm} \\

$\|\cdot\|_{W_1\to W_2}$
& operator norm from $W_1$ to $W_2$. \symref{notation:W12} \\

$\threebar{\cdot}$ & arbitrary norm on a finite-dimensional vector space, or induced operator norm \symref{notation:arbitnorm}.\\
\hline
\end{longtable}

\end{adjustwidth}
\end{small}
\subsection*{Acknowledgements}

The authors thank Artur Avila, Bal\'azs B\'ar\'any, Jairo Bochi, Antti K\"aenm\"aki and Fran\c{c}ois Ledrappier for helpful comments and suggestions. They also thank an anonymous referee for their careful reading of this manuscript and for a number of helpful suggestions and references.
\bibliographystyle{acm}
\bibliography{exceptional}

\end{document}